\documentclass[11pt,reqno]{amsart}

\usepackage{amssymb,amsmath,amsthm,amsfonts}
\usepackage{hyperref, multirow, bm, color, marginnote, graphicx, wrapfig, subcaption, esint}
\usepackage{diagbox}

\makeatletter
\renewcommand{\tocsection}[3]{%
  \indentlabel{\@ifnotempty{#2}{\bfseries\ignorespaces#1 #2\quad}}\bfseries#3}
\renewcommand{\tocsubsection}[3]{%
  \indentlabel{\@ifnotempty{#2}{\ignorespaces#1 #2\quad}}#3}

\newcommand\@dotsep{4.5}
\def\@tocline#1#2#3#4#5#6#7{\relax
  \ifnum #1>\c@tocdepth 
  \else
    \par \addpenalty\@secpenalty\addvspace{#2}%
    \begingroup \hyphenpenalty\@M
    \@ifempty{#4}{%
      \@tempdima\csname r@tocindent\number#1\endcsname\relax
    }{%
      \@tempdima#4\relax
    }%
    \parindent\z@ \leftskip#3\relax \advance\leftskip\@tempdima\relax
    \rightskip\@pnumwidth plus1em \parfillskip-\@pnumwidth
    #5\leavevmode\hskip-\@tempdima{#6}\nobreak
    \leaders\hbox{$\m@th\mkern \@dotsep mu\hbox{.}\mkern \@dotsep mu$}\hfill
    \nobreak
    \hbox to\@pnumwidth{\@tocpagenum{\ifnum#1=1\bfseries\fi#7}}\par
    \nobreak
    \endgroup
  \fi}
\AtBeginDocument{%
\expandafter\renewcommand\csname r@tocindent0\endcsname{0pt}
}
\def\l@subsection{\@tocline{2}{0pt}{2.5pc}{5pc}{}}
\makeatother

\usepackage[margin=1in]{geometry}

\newcommand{\R}{\mathbb{R}}
\newcommand{\Z}{\mathbb{Z}}
\newcommand{\N}{\mathbb{N}}

\newcommand{\T}{\mathbb{T}}
\renewcommand{\P}{\mathbb{P}}

\newcommand{\bR}{\bm{R}}
\newcommand{\barR}{\overline{\bm{R}}}

\newcommand{\bartheta}{\wh{\theta}}
\newcommand{\bars}{\wh s}

\newcommand{\bu}{\bm{u}}

\newcommand{\bx}{\bm{x}}
\newcommand{\by}{\bm{y}}

\newcommand{\X}{\bm{X}}
\newcommand{\be}{\bm{e}}

\newcommand{\p}{\partial}
\renewcommand{\div}{{\rm{div}\,}}

\newcommand{\abs}[1]{\left\lvert #1 \right\rvert}
\newcommand{\norm}[1]{\left\lVert #1 \right\rVert}
\newcommand{\wh}[1]{\widehat{#1}}
\newcommand{\wt}[1]{\widetilde{#1}}
\newcommand{\mc}[1]{\mathcal{#1}}

\newtheorem{theorem}{Theorem}[section]
\newtheorem{lemma}[theorem]{Lemma}
\newtheorem{proposition}[theorem]{Proposition}
\newtheorem{corollary}[theorem]{Corollary}

\theoremstyle{definition}
\newtheorem*{remark}{Remark}

\allowdisplaybreaks

\begin{document}
\title{On an angle-averaged Neumann-to-Dirichlet map for thin filaments}

\author{Laurel Ohm}
\address{Department of Mathematics, University of Wisconsin - Madison, Madison, WI 53706}
\email{lohm2@wisc.edu}

\begin{abstract} 
We consider the Laplace equation in the exterior of a thin filament in $\R^3$ and perform a detailed decomposition of a notion of slender body Neumann-to-Dirichlet (NtD) and Dirichlet-to-Neumann (DtN) maps along the filament surface. The decomposition is motivated by a filament evolution equation in Stokes flow for which the Laplace setting serves as an important toy problem.
Given a general curved, closed filament with constant radius $\epsilon>0$, we show that both the slender body DtN and NtD maps may be decomposed into the corresponding operator about a straight, periodic filament plus lower order remainders. For the straight filament, both the slender body NtD and DtN maps are given by explicit Fourier multipliers and it is straightforward to compute their mapping properties. The remainder terms are lower order in the sense that they are small with respect to $\epsilon$ or smoother.
 While the strategy here is meant to serve as a blueprint for the Stokes setting, the Laplace problem may be of independent interest. 
\end{abstract}

\maketitle

\tableofcontents

\setlength{\parskip}{6pt}
\section{Introduction}
We consider the Laplace equation in the exterior of a thin filament in $\R^3$ and study a \emph{slender body Neumann-to-Dirichlet (NtD) map} and its inverse, the slender body Dirichlet-to-Neumann (DtN) map, along the filament surface. The boundary value problem corresponding to the slender body NtD map was introduced by the author, together with Mori and Spirn, in \cite{closed_loop,free_ends} in the context of the Stokes equations as mathematical justification for \emph{slender body theory}, an approximation for the interaction between a thin filament and a 3D viscous fluid. Slender body theory, developed and improved in various classical works \cite{batchelor1970slender, cox1970motion, gotz2000interactions, gray1955propulsion, johnson1980improved, keller1976slender, lighthill1976flagellar, tornberg2004simulating}, approximates the immersed filament as a 1D curve evolving according to a 1D force-to-velocity map along the fiber centerline. Making rigorous sense of this 1D force-to-velocity map for a truly 3D filament motivated the development of the slender body NtD map in \cite{closed_loop,free_ends} for a static filament, i.e., a fixed filament geometry. The PDE results in this paper are motivated by a filament evolution problem in the Stokes setting using the slender body NtD map. Here we study the Laplace version of the slender body NtD map as an important toy model for the Stokes setting.

The main result of this paper is Theorem \ref{thm:main}, which provides a detailed decomposition of the Laplace slender body NtD and DtN maps about a general curved, closed filament with constant radius $\epsilon>0$. We show that for both operators, the leading order mapping properties are given by the corresponding operator about a straight, periodic filament with radius $\epsilon$. In this simple geometry, the slender body NtD and DtN maps are both given by explicit Fourier multipliers, and it is relatively straightforward to compute their mapping properties. We extract the straight operator from the expression for the general curved filament and show that the remainder terms are lower order in the sense that they are small with respect to $\epsilon$ or smoother. The arguments presented here are meant to serve as a blueprint for performing the same type of decomposition in the Stokes setting.

\emph{Main motivation:} Our long-term aim is to develop a complete well-posedness theory for the evolution equation \eqref{eq:Stokes} describing the dynamics of an immersed elastic filament in Stokes flow. This would provide a more complete mathematical justification of the myriad computational results for thin filament dynamics based on slender body theory. Let $\X(s,t)$ denote the centerline position of a filament immersed in 3D Stokes flow at time $t$. We consider the evolution of $\X$ according to 
\begin{equation}\label{eq:Stokes}
\frac{\p\X}{\p t} = -({\rm SB \,NtD})\big[(\X_{sss}-\lambda \X_s)_s\big]\,, \qquad \abs{\X_s}=1\,.
\end{equation}
Here (SB NtD) denotes a version of Neumann-to-Dirichlet map in 3D Stokes flow taking force data defined along a 1D curve to the motion of the filament centerline. The way in which we propose to make sense of this map when the filament is interacting with a 3D fluid will be made more precise later (in the context of the Laplace equation--see equation \eqref{eq:SB_PDE}). 
The form of the force data $(\X_{sss}-\lambda \X_s)_s$ comes from Euler-Bernoulli beam theory \cite{camalet2000generic, camalet1999self, hines1978bend, tornberg2004simulating, wiggins1998flexive, wiggins1998trapping} and is a simple elastic response along the filament with $\lambda(s,t)$, the filament tension, serving as a Lagrange multiplier to enforce the inextensibility constraint $\abs{\X_s}=1$. 
The evolution \eqref{eq:Stokes} bears analogies to the Peskin problem for a 1D immersed elastic filament in 2D \cite{cameron2021critical, chen2021peskin, gancedo2020global, garcia2020peskin, lin2019solvability, kuo2023tension, mori2019well, tong2021regularized, tong2023geometric}, but here the filament is immersed in a 3D fluid and part of the difficulty in analyzing this problem is simply making sense of the operator (SB NtD) along the filament.

The PDE evolution \eqref{eq:Stokes} is studied in \cite{mori2023well,ohm2022well} when (SB NtD) is replaced with \emph{local} slender body theory $(\text{SB}_{\rm loc} \,\text{NtD})$, given simply by multiplication by the matrix $c\abs{\log\epsilon}({\bf I}+\X_s\X_s^{\rm T})$. In this setting, in addition to establishing well-posedness for the analogue of \eqref{eq:Stokes}, we show how the addition of a time-periodic forcing can give rise to swimming, i.e., net forward motion. However, this map incorporates only a very simplified description of the effects of the surrounding fluid on the filament. In particular, this is not yet a truly coupled fluid-structure free boundary problem. 
To incorporate more of the physics of the fluid-structure interaction, one can try using \emph{nonlocal} slender body theory $(\text{SB}_{\rm nloc} \,\text{NtD})$ \cite{keller1976slender,johnson1980improved, tornberg2004simulating}, which is a 1D integral operator incorporating nonlocal interactions along the length of the filament. However, from a PDE perspective, this 1D integral operator cannot yield a well-posed evolution problem in the form \eqref{eq:Stokes} due to issues at high wavenumbers leading to nonsensical mapping properties \cite{gotz2000interactions, inverse, tornberg2004simulating, shelley2000stokesian}. It is a nontrivial issue to incorporate more of the fluid physics while yielding a well-posed PDE. The challenge here is to make sense of a map from 1D force data to a 1D filament velocity when the filament is in 3D.

In \cite{closed_loop,free_ends}, we developed a notion of the operator (SB NtD) as a novel type of boundary value problem in the exterior of a filament with small but finite radius. In our setup, at each instant in time, the total force density $\bm{f}(s)$ over each 3D cross section of the filament is prescribed, and we solve for the corresponding filament velocity $\bu(s)$, unknown but constrained to be constant on cross sections, as a Stokes boundary value problem (compare with equation \eqref{eq:SB_PDE} in the Laplace setting). 
 The resulting boundary value problem is nonstandard but well posed in a natural energy space \cite{closed_loop,free_ends,rigid}. Crucially, this formulation also incorporates much more of the surrounding fluid physics than local slender body theory does. With this operator, \eqref{eq:Stokes} is a true free boundary problem.

To develop the well-posedness theory of \eqref{eq:Stokes} for the full (SB NtD) operator, we need a detailed understanding of its mapping properties. In particular, given the form $(\X_{sss}-\lambda\X_s)_s$ of the forcing terms, we need to understand how this operator interacts with derivatives along the filament.
Here we perform a full decomposition of the analogous (SB NtD) operator in the Laplace setting, providing an important stepping stone towards tackling the Stokes version. In particular, all of the major tools used here have an analogue in the Stokes setting. The analysis is already technical in the simpler Laplace setting, which underscores why such a blueprint is useful.

Our strategy is to consider the general curved, closed filament as a perturbation of the straight filament. This strategy has been used to study vortex filament solutions of the Ginzburg-Landau, Navier-Stokes, and Euler equations \cite{bedrossian2018vortex, davila2022interacting, davila2022travelling}, 
and bears analogy to the Dirichlet-to-Neumann operator decompositions used to study the Muskat problem \cite{alazard2020paralinearization, flynn2021vanishing, nguyen2020paradifferential} and water waves \cite{alazard2014cauchy, alazard2009paralinearization, lannes2005well}, where the strategy is to perturb around a flat interface. 
Extracting the straight operator as the leading order behavior is useful since it is given by an explicit Fourier multiplier, computed in \cite{inverse}. In particular, we have a Fourier multiplier representation for both the Dirichlet-to-Neumann and Neumann-to-Dirichlet directions, and we know exactly how this operator interacts with derivatives along the filament.

\emph{Additional motivation:}
Despite no clear analogue of the full dynamic problem \eqref{eq:Stokes} for Laplace, the Laplace slender body DtN and NtD maps have many conceivable applications for which a detailed understanding of mapping properties may be useful. Examples include modeling blood or chemical perfusion in tissue outside of a thin capillary using Darcy flow (see, e.g., \cite{kuchta2021analysis}), describing the electrostatic potential outside of a conducting wire, and computing the equilibrium temperature distribution outside of a heating wire. In many of these cases, the full power of the decomposition performed here is likely not necessary to completely understand the model behavior. Nevertheless, this paper provides a foundational result for further work on problems involving the Laplace equation outside a thin domain.

\subsection{Geometry}\label{subsec:geom}
Let $\X : \T\equiv \R / \Z \to \R^3$ denote the coordinates of a closed curve $\Gamma_0\in\R^3$, parameterized by arclength $s$ (see figure \ref{fig:filament}). We will require $\X(s)$ to belong to the H\"older space $C^{2,\beta}(\T)$ (see \eqref{eq:Ckalpha}) and to satisfy the non-self-intersection condition
\begin{equation}\label{eq:cGamma}
 \inf_{s\neq s'}\frac{\abs{\X(s)-\X(s')}}{\abs{s-s'}} \ge c_\Gamma 
 \end{equation}
for some constant $c_\Gamma>0$. Letting $\be_{\rm t}(s):= \frac{\p\X}{\p s}$ denote the unit tangent vector to $\X(s)$, we parameterize points sufficiently close to $\Gamma_0$ using the $C^{1,\beta}$ orthonormal frame $(\be_{\rm t}(s),\be_{\rm n_1}(s),\be_{\rm n_2}(s))$ satisfying the ODEs 
\begin{equation}\label{eq:frame}
\frac{d}{ds}
\begin{pmatrix}
\be_{\rm t}(s)\\
\be_{\rm n_1}(s)\\
\be_{\rm n_2}(s)
\end{pmatrix}
 = \begin{pmatrix}
 0 & \kappa_1(s) & \kappa_2(s) \\
 -\kappa_1(s) & 0 & \kappa_3 \\
-\kappa_2(s) & -\kappa_3& 0
 \end{pmatrix}\begin{pmatrix}
\be_{\rm t}(s)\\
\be_{\rm n_1}(s)\\
\be_{\rm n_2}(s)
\end{pmatrix}\,.
\end{equation}
Here the coefficients $\kappa_1$ and $\kappa_2$ satisfy 
\begin{align*}
\kappa_1^2(s) +\kappa_2^2(s) = \kappa^2(s)\,,
\end{align*}
where $\kappa(s)$ is the curvature of $\X(s)$, and $\kappa_3$ is a constant satisfying $\abs{\kappa_3}\le \pi$ (see \cite[Lemma 1.1]{closed_loop}). We denote
 \begin{equation}\label{eq:kappastar}
 \kappa_* := \max_{s\in\T} \abs{\kappa(s)}, \quad \kappa_{*,\beta} := \max_{s\in\T} \frac{\abs{\kappa(s)-\kappa(s')}}{\abs{s-s'}^\beta} \,.
 \end{equation}
It will be useful to define the following curved version of cylindrical basis vectors:
\begin{equation}\label{eq:er_etheta}
\be_r(s,\theta) = \cos\theta\be_{\rm n_1}(s) + \sin\theta\be_{\rm n_2}(s)\,,
\quad \be_\theta(s,\theta) = -\sin\theta\be_{\rm n_1}(s) + \cos\theta\be_{\rm n_2}(s)\,.
\end{equation}
Within a neighborhood ${\rm dist}(\bx,\Gamma_0)< r_*(c_\Gamma,\kappa_*)< \frac{1}{2\kappa_*}$ of $\Gamma_0$, we may parameterize points $\bx$ as
\begin{align*} 
\bx = \X(s) + r\be_r(s,\theta), \quad 0\le r <r_* \,.
\end{align*}
For $\epsilon<r_*/4$, we define a slender filament of uniform radius $\epsilon$ as
\begin{equation}\label{eq:SB_def} 
\Sigma_\epsilon:= \big\{\bx\in \R^3 \; : \; \bx = \X(s) + r\be_r(s,\theta)\,,  \; s\in\T\,, \; r < \epsilon\,, \; \theta\in2\pi\T \big\}\,.
\end{equation}
We define the filament surface $\Gamma_\epsilon = \p \Sigma_\epsilon$ as
\begin{align*}
 \Gamma_\epsilon :=  \big\{\bx\in \R^3 \; : \; \bx = \X(s) + \epsilon\be_r(s,\theta)\,,  \; s\in\T\,, \; \theta\in2\pi\T \big\}
 \end{align*}
and parameterize the surface element $\mc{J}_\epsilon(s,\theta)$ along the filament surface as
\begin{equation}\label{eq:jacfac}
\mc{J}_\epsilon(s,\theta) = \epsilon(1-\epsilon\wh\kappa(s,\theta))\,, \qquad
\wh\kappa(s,\theta):= \kappa_1(s)\cos\theta+\kappa_2(s)\sin\theta\,. 
\end{equation}

\begin{figure}[!ht]
\centering
\includegraphics[scale=0.3]{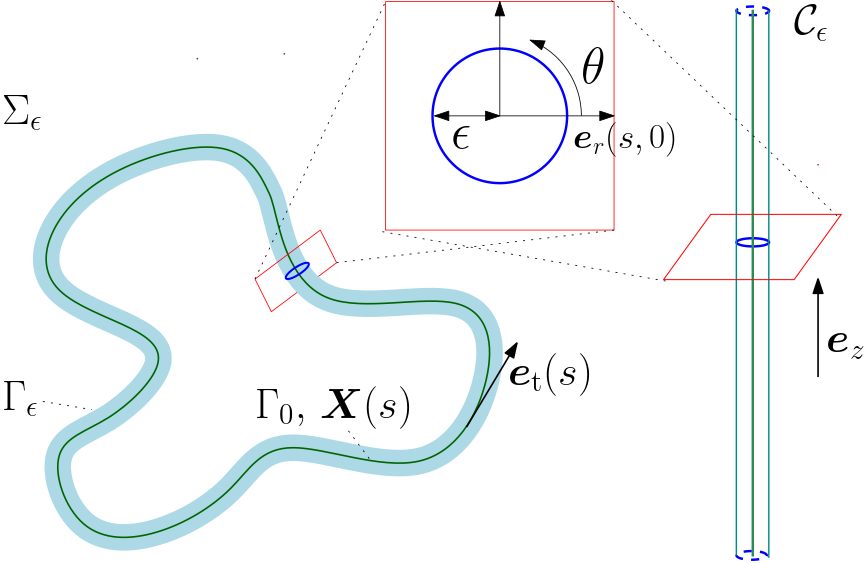}
\caption{An example of curved filament $\Sigma_\epsilon$ described in section \ref{subsec:geom} and the straight filament $\mc{C}_\epsilon$ described in \eqref{eq:Cepsilon}.}
\label{fig:filament}
\end{figure}

\subsection{Slender body Neumann-to-Dirichlet map}
We need to make sense of a map between 1D Neumann data and 1D Dirichlet data in $\R^3$. For the Laplace equation, this map will be defined via the following \emph{slender body boundary value problem}:
\begin{equation}\label{eq:SB_PDE}
\begin{aligned}
\Delta u &= 0 \qquad\qquad \text{in }\Omega_\epsilon = \R^3\backslash \overline{\Sigma_\epsilon} \\
\int_0^{2\pi}\frac{\p u}{\p \bm{n}_x} \, \mc{J}_\epsilon(s,\theta)\,d\theta &= f(s) \qquad\,\; \text{on }\Gamma_\epsilon \\
u\big|_{\Gamma_\epsilon} &= v(s)\,, \qquad \text{independent of }\theta\,.   
\end{aligned}
\end{equation}
In the ``forward" version of the problem, we prescribe the \emph{total} Neumann data $f(s)$ per cross section of the filament surface, and we seek the unique function $u$ which is harmonic in $\Omega_\epsilon=\R^3\backslash\overline{\Sigma_\epsilon}$ and constant over every cross section of $\Gamma_\epsilon$, i.e. a function of arclength $s$ only (see figure \ref{fig:filament}).
 In the ``inverse" version of the problem, we prescribe Dirichlet data $v(s)$ that is independent of the cross sectional angle parameter $\theta$ and solve for $u$ in $\Omega_\epsilon$ and, subsequently, the integral of $\frac{\p u}{\p\bm{n}_x}$ over each cross section.

Our analysis centers on the relationship between the total Neumann boundary data $f(s)$ and the $\theta$-independent Dirichlet data $v(s)$. Our aim is to understand the \emph{slender body Neumann-to-Dirichlet (NtD) map} 
\begin{equation}\label{eq:SB_NtD0}
\mc{L}_\epsilon: f(s)\mapsto v(s)
\end{equation} 
as well as its inverse, the \emph{slender body Dirichlet-to-Neumann (DtN) map} $\mc{L}_\epsilon^{-1}: v(s)\mapsto f(s)$. These maps are the slender body analogues of the standard Neumann-to-Dirichlet and Dirichlet-to-Neumann maps for harmonic functions, defined only between functions which depend solely on the arclength parameter $s$ along $\Gamma_\epsilon$.

The slender body inverse problem is simpler than the forward problem, as it is just a Dirichlet problem with $\theta$-independent data, and classical regularity results apply. Our strategy for developing the regularity theory of the trickier forward problem in many ways exploits the fact that the inverse problem is simpler to study. The forward problem was introduced in the Stokes setting in \cite{closed_loop,free_ends} and well-posedness was established in a natural energy space. Here we develop the regularity theory for the forward problem in the setting of H\"older spaces (see Lemma~\ref{lem:SB_PDE_holder} and its proof in section \ref{sec:SBNtD_Holder}), since this will be a convenient setting for extracting detailed mapping information from the layer potential formulation of \eqref{eq:SB_PDE} considered in section \ref{subsubsec:layer}. 

For $h$ defined along $\Gamma_\epsilon$, we denote 
\begin{equation}\label{eq:norms}
\norm{h}_{L^\infty} = \sup_{\bx\in\Gamma_\epsilon}\abs{h(\bx)}\,, \qquad
\abs{h}_{\dot C^{0,\alpha}} = \sup_{\bx\neq\bx'}\frac{\abs{ h(\bx)-h(\bx')}}{\abs{\bx-\bx'}^\alpha}\,, \quad 0<\alpha<1\,, 
\end{equation}
and recall the definition of the H\"older spaces $C^{k,\alpha}$ along $\Gamma_\epsilon$ as
\begin{equation}\label{eq:Ckalpha}
C^{k,\alpha}(\Gamma_\epsilon) = \{ h \; : \; \norm{h}_{C^{k,\alpha}} <\infty\}\,, \qquad 
\norm{h}_{C^{k,\alpha}} = \sum_{\ell=0}^k\norm{\nabla^\ell h}_{L^\infty}+\abs{\nabla^k h}_{\dot C^{0,\alpha}}\,.
\end{equation}

\subsubsection{Slender body NtD for the straight filament}
In the simple geometry of a straight filament, an explicit Fourier multiplier representation of $\mc{L}_\epsilon$ was derived in \cite{inverse}.
This explicit information will play a major role in our analysis of $\mc{L}_\epsilon$ and $\mc{L}_\epsilon^{-1}$ in the more general curved setting. 

In the case of the straight cylinder, the filament centerline $\X(s)$ may be written $\X(s)=s\be_z$, $s\in \T$, where $\be_z$ is the standard Cartesian basis vector (see figure \ref{fig:filament}). The vectors $\be_r(s,\theta)$ and $\be_\theta(s,\theta)$ \eqref{eq:er_etheta} reduce to the usual cylindrical basis vectors $\be_r(\theta)$, $\be_\theta(\theta)$. We will use $\mc{C}_\epsilon$ to denote the surface $\Gamma_\epsilon$ in the straight setting, i.e. 
\begin{equation}\label{eq:Cepsilon} 
\mc{C}_\epsilon :=  \big\{\bx\in \R^2\times\T \; : \; \bx = s\be_z + \epsilon\be_r(\theta)\,,  \; s\in\T\,, \; \theta\in2\pi\T \big\}\,.
\end{equation}
Letting $\overline{\mc{L}}_\epsilon$, $\overline{\mc{L}}_\epsilon^{-1}$ denote the slender body NtD and DtN maps along $\mc{C}_\epsilon$, from \cite{inverse} we have that the symbol of $\overline{\mc{L}}_\epsilon^{-1}$ on the unit circle $\T$ is given by 
\begin{equation}\label{eq:symbol}
m_\epsilon^{-1}(k) = 4\pi^2 \epsilon \abs{k}\frac{K_1(2\pi\epsilon \abs{k})}{K_0(2\pi\epsilon \abs{k})}\,, \quad \abs{k}=1,2,3,\dots \,,
\end{equation}
where $K_0$ and $K_1$ denote zeroth and first order modified Bessel functions of the second kind. In \cite{inverse}, it is shown that \textcolor{black}{at low wavenumbers $k\lesssim \frac{1}{\epsilon}$, the growth of $m_\epsilon^{-1}(k)$ is logarithmic:
\begin{equation}
 m_\epsilon^{-1}(k) \sim \abs{\log(\epsilon \abs{k})}^{-1}\,, \qquad k\lesssim \frac{1}{\epsilon}\,, 
 \end{equation} 
 }
while $m_\epsilon^{-1}(k)$ grows linearly as $\abs{k}\to \infty$: 
\begin{equation}\label{eq:growth}
4\pi^2\epsilon \abs{k} \le m_\epsilon^{-1}(k) \le 4\pi^2\epsilon \abs{k} + 2\pi\qquad \text{for all }k\,.
\end{equation}
From \eqref{eq:growth}, the $L^2$ mapping properties of $\overline{\mc{L}}_\epsilon$ and $\overline{\mc{L}}_\epsilon^{-1}$ are clear: $\overline{\mc{L}}_\epsilon$ maps the Sobolev space $H^s(\T)$ to $H^{s+1}(\T)$, and vice versa for $\overline{\mc{L}}_\epsilon^{-1}$.
In addition, from the expression \eqref{eq:symbol}, we can obtain the following H\"older space mapping properties along $\T$. 
\begin{lemma}[Mapping properties of SB NtD and DtN on $\mc{C}_\epsilon$]\label{lem:mapping_Leps}
Let $\overline{\mc{L}}_\epsilon$ denote the slender body NtD map along the straight filament $\mc{C}_\epsilon$. Given $h\in C^{0,\alpha}(\T)$ with $\int_\T h(s)\,ds=0$, we have  
\begin{equation}\label{eq:est_barLeps}
\norm{\overline{\mc{L}}_\epsilon[h]}_{C^{1,\alpha}(\T)} \le c\abs{\log\epsilon}^{1/2}\norm{h}_{L^\infty(\T)} + c\,\epsilon^{-1}\abs{\log\epsilon}^{1/2}|h|_{\dot C^{0,\alpha}(\T)} \,.
\end{equation}

Similarly, let $\overline{\mc{L}}_\epsilon^{-1}$ denote the slender body DtN map along $\mc{C}_\epsilon$. Given $g\in C^{1,\alpha}(\T)$, we have 
\begin{equation}\label{eq:est_barLepsinv}
\norm{\overline{\mc{L}}_\epsilon^{-1}[g]}_{C^{0,\alpha}(\T)} \le c\abs{\log\epsilon}^{-1}\norm{g}_{C^{1,\alpha}(\T)}\,.
\end{equation}
\end{lemma}
The proof of Lemma \ref{lem:mapping_Leps} appears in section \ref{sec:strt_periodic}.

\subsubsection{Statement of main theorem}
Given that we have explicit Fourier multiplier expressions for $\overline{\mc{L}}_\epsilon$ and $\overline{\mc{L}}_\epsilon^{-1}$ in the straight setting, the goal of this paper is to show that this explicit information in fact provides the leading order behavior for more general curved filaments. Using a layer potential representation of $\mc{L}_\epsilon^{-1}$ for a general closed filament as in section \ref{subsec:geom}, we extract the straight operator $\overline{\mc{L}}_\epsilon^{-1}$ and show that the remainders are lower order with respect to regularity or size in $\epsilon$. The decomposition of $\mc{L}_\epsilon^{-1}$ is then used to show that $\mc{L}_\epsilon$ may similarly be decomposed as the straight operator $\overline{\mc{L}}_\epsilon$ plus lower order remainders.


\begin{theorem}[Decomposition of slender body DtN and NtD]\label{thm:main}
Let $0<\alpha<\gamma<\beta<1$. Given a closed filament $\Sigma_\epsilon$ as in section \ref{subsec:geom} with centerline $\X(s)\in C^{2,\beta}$, the slender body Dirichlet-to-Neumann operator $\mc{L}_\epsilon^{-1}$ may be decomposed as 
\begin{equation}\label{eq:thm_Leps_inv}
\mc{L}_\epsilon^{-1} = \overline{\mc{L}}_\epsilon^{-1} + \mc{R}_{\rm d,\epsilon} + \mc{R}_{\rm d,+}
\end{equation}
where, given $v(s)\in C^{1,\alpha}(\T)$, we have 
\begin{equation}\label{eq:thm_Leps_inv_bds}
\begin{aligned}
\norm{\mc{R}_{\rm d,\epsilon}[v]}_{C^{0,\alpha}(\T)} &\le
 \epsilon^{2-\textcolor{black}{\alpha^+}}\,c(\kappa_{*,\textcolor{black}{\alpha^+}},c_\Gamma)\norm{v}_{C^{1,\alpha}(\T)}\\
\norm{\mc{R}_{\rm d,+}[v]}_{C^{0,\gamma}(\T)} &\le 
 c(\kappa_{*,\textcolor{black}{\gamma^+}},c_\Gamma,\epsilon) \norm{v}_{C^{1,\alpha}(\T)}\,.
\end{aligned} 
\end{equation}
\textcolor{black}{for any $\alpha^+\in(\alpha, \beta]$ and any $\gamma^+\in(\gamma, \beta]$.}

For $\epsilon$ sufficiently small we may similarly decompose the slender body Neumann-to-Dirichlet operator $\mc{L}_\epsilon$ defined in \eqref{eq:SB_NtD0} as 
\begin{equation}\label{eq:thm_Leps}
\mc{L}_\epsilon = ({\bf I}+\mc{R}_{\rm n,\epsilon})\overline{\mc{L}}_\epsilon   + \mc{R}_{\rm n,+}
\end{equation}
where, given $f(s)\in C^{0,\alpha}(\T)$ with $\int_{\T}f(s)\,ds=0$, we have
\begin{equation}\label{eq:thm_Leps_bds}
\begin{aligned}
\norm{\mc{R}_{\rm n,\epsilon}\overline{\mc{L}}_\epsilon[f]}_{C^{1,\alpha}(\T)}&\le \epsilon^{1-\textcolor{black}{\alpha^+}}\abs{\log\epsilon}^{1/2}\,c(\kappa_{*,\textcolor{black}{\alpha^+}},c_\Gamma)\norm{\overline{\mc{L}}_\epsilon[f]}_{C^{1,\alpha}(\T)}\\
\norm{\mc{R}_{\rm n,+}[f]}_{C^{1,\gamma}(\T)} &\le c(\textcolor{black}{\kappa_{*,\gamma^+}},c_\Gamma,\epsilon)\norm{f}_{C^{0,\alpha}(\T)}
\end{aligned}
\end{equation}
\textcolor{black}{for any $\alpha^+\in(\alpha, \beta]$ and any $\gamma^+\in(\gamma, \beta]$.}
\end{theorem}

\begin{remark}
 \textcolor{black}{The goal of this work is to establish a direct route for extracting the necessary information for addressing this difficult physical problem \eqref{eq:Stokes}. Here we use a boundary integral representation of the DtN map to explicitly extract the straight filament symbol as the leading order behavior. In principle, it should be possible to frame the analysis of the slender body NtD and DtN in terms of pseudodifferential calculus to obtain more general results for these maps. Here, however, we opt for directness over full generality. }
 \end{remark}

The proof of Theorem \ref{thm:main} relies on a collection of key lemmas, stated below (Lemmas \ref{lem:straight_components} -- \ref{lem:SB_PDE_holder}).
We will immediately use these key lemmas to prove Theorem \ref{thm:main} in section \ref{subsec:thm_pf}. The remainder of the paper will then be devoted to proving each key lemma.

\subsubsection{Layer potential formulation of slender body DtN map}\label{subsubsec:layer}
To understand the mapping properties of $\mc{L}_\epsilon$ along the more general curved filaments described in section \ref{subsec:geom}, we will rely on a representation of $\mc{L}_\epsilon^{-1}$ in terms of layer potentials on $\Gamma_\epsilon$. This representation is inspired by but differs from the boundary integral formulation of the Stokes slender body PDE proposed for numerical purposes in \cite{mitchell2022single}. From this representation, we will then extract the straight operator $\overline{\mc{L}}_\epsilon^{-1}$ as the leading order behavior. In particular, the remainder terms arising due to filament curvature can be shown to be lower order with respect to regularity or with respect to size in terms of the filament radius $\epsilon$. Here and throughout, we will use overline notation $\overline{(\cdot)}$ to denote functions defined along the straight filament $\mc{C}_\epsilon$.

We begin by defining notation and recalling some results from potential theory (see, e.g. \cite{folland1995introduction, kress1989linear, steinbach2007numerical}). 
Let $\mc{G}(\bx,\bx')$ denote the fundamental solution of the Laplace equation in $\R^3$:
\begin{equation}\label{eq:G}
\mc{G}(\bx,\bx') = \frac{1}{4\pi}\frac{1}{\abs{\bx-\bx'}}\,.
\end{equation}
For $\bx'\in\Gamma_\epsilon$, we define 
\begin{equation}\label{eq:KD}
K_\mc{D}(\bx,\bx') = \frac{\p\mc{G}(\bx,\bx')}{\p\bm{n}_{x'}} = \frac{1}{4\pi}\frac{(\bx-\bx')\cdot\bm{n}_{x'}}{\abs{\bx-\bx'}^3}\,, 
\end{equation}
where $\bm{n}_{x'}$ is the unit normal vector to $\Gamma_\epsilon$ at $\bx'$ pointing \emph{out from} the slender body $\Sigma_\epsilon$. 
Given a continuous function $\varphi(\bx')$ defined along the filament surface $\Gamma_\epsilon$, for $\bx\in\Omega_\epsilon$ we define the single layer potential $\mc{S}$ and double layer potential $\mc{D}$ by
\begin{align*}
\mc{S}[\varphi](\bx) = \int_{\Gamma_\epsilon} \mc{G}(\bx,\bx')\varphi(\bx') \, dS_{x'}\,, \qquad
\mc{D}[\varphi](\bx) = \int_{\Gamma_\epsilon} K_\mc{D}(\bx,\bx')\varphi(\bx') \, dS_{x'}\,,
\end{align*}
where $dS_{x'}$ denotes the surface element with respect to $\bx'$ on $\Gamma_\epsilon$. For $\bx\in\Gamma_\epsilon$, the single layer potential is continuous up to $\Gamma_\epsilon$, while the double layer potential satisfies the exterior jump relation 
\begin{equation}\label{eq:ext_jump}
 \lim_{h\to 0^+} \mathcal{D}[\varphi](\bx + h\bm{n}_x) = \mathcal{D}[\varphi](\bx) + \frac{1}{2}\varphi(\bx)\,,
 \end{equation}
Here again $\bm{n}_x$ is the unit normal vector at $\bx\in \Gamma_\epsilon$ exterior to the slender body $\Sigma_\epsilon$ (see \cite{kress1989linear}, Theorems 6.15 and 6.18). 

Now, given sufficiently smooth functions $u_1$ and $u_2$ in $\Omega_\epsilon$ with sufficiently fast decay at infinity, letting $v_1=u_1\big|_{\Gamma_\epsilon}$ and $v_2=u_2\big|_{\Gamma_\epsilon}$, we have Green's formula:
\begin{align*}
 \int_{\Omega_\epsilon} \bigg(u_1 \Delta u_2 - u_2\Delta u_1\bigg)\,d\bx = -\int_{\Gamma_\epsilon}\bigg( v_1 \frac{\p u_2}{\p \bm{n}_x} - v_2\frac{\p u_1}{\p \bm{n}_x}\bigg)\,dS_x, \quad \bm{n}_x \text{ exterior normal to }\Sigma_\epsilon\,. 
 \end{align*}
This follows by applying the divergence theorem to the vector field $u_1\nabla u_2-u_2\nabla u_1$ in $\Omega_\epsilon$. Note that the minus sign on the right hand side arises because $\bm{n}_x$ points out of the slender body $\Sigma_\epsilon$, so it is really the interior normal vector to $\Omega_\epsilon$. If $u_1$ is harmonic, then, fixing $\bx_0\in\Omega_\epsilon$ and taking $u_2(\bx)=\mc{G}(\bx_0,\bx)$, and noting $\Delta_x u_2 = -\delta(\bx-\bx_0)$, we obtain 
\begin{equation}\label{eq:greenform}
\begin{aligned}
u_1(\bx_0) &= \mathcal{S}[w_1](\bx_0)  + \mathcal{D}[v_1](\bx_0); \\
w_1(\bx) &:= -\frac{\p u_1}{\p\bm{n}_x}\bigg|_{\Gamma_\epsilon}, \quad \bm{n}_x \text{ exterior normal to }\Sigma_\epsilon\,, \\
v_1(\bx) &:= u_1 \big|_{\Gamma_\epsilon}\,.
\end{aligned}
\end{equation}
Here, again, the minus sign in $w_1$ arises because $\bm{n}_x$ points into $\Omega_\epsilon$.
Then, dropping the subscripts and using \eqref{eq:ext_jump} in \eqref{eq:greenform}, for a harmonic function $u$ in $\Omega_\epsilon$, the Dirichlet and Neumann boundary data $v$ and $w$, respectively, may be related by  
\begin{equation}\label{eq:usual_NtD}
 \left(\textstyle\frac{1}{2}{\bf I}- \mathcal{D}\right)[v](\bx) = \mathcal{S}[w](\bx)\,, \quad \bx\in \Gamma_\epsilon\,. 
 \end{equation}
Using \eqref{eq:usual_NtD}, the slender body Neumann data $f(s)$ and Dirichlet data $v(s)$ may be related via 
\begin{equation}\label{eq:SB_NtD} 
\begin{aligned}
\left( \textstyle \frac{1}{2}{\bf I} - \mathcal{D} \right) [v(s)] & = \mathcal{S}[w(s,\theta)] \\
\int_0^{2\pi} w(s,\theta)\, \mc{J}_\epsilon(s,\theta) \, d\theta &= f(s)\,. 
\end{aligned}
\end{equation}
Here we consider the Neumann boundary value $w(s,\theta)=w(\bx(s,\theta))$ as a function of the surface parameters $s$ and $\theta$ along $\Gamma_\epsilon$.

In the straight setting, we may use the representation \eqref{eq:SB_NtD} to exactly recover the symbol $m_\epsilon(k)$ in \eqref{eq:symbol}. In particular, using $\overline{\mc{S}}$ and $\overline{\mc{D}}$ to denote the single and double layer operators defined along the straight filament $\mc{C}_\epsilon$, we have the following lemma.
\begin{lemma}[Single and double layer operators on $\mc{C}_\epsilon$]\label{lem:straight_components}
Along the straight filament $\mc{C}_\epsilon$, the behavior of the single and double layer operators $\overline{\mc{S}}$ and $\overline{\mc{D}}$ is given by explicit Fourier multipliers: 
\begin{equation}\label{eq:S_multiplier}
\begin{aligned}
\overline{\mc{S}}[e^{2\pi i k s}e^{i\ell\theta}]&= m_\mc{S}(k,\ell)e^{2\pi i k s}\cos(\ell\theta)\,,\\
m_\mc{S}(k,\ell) &= \epsilon\, I_{\abs{\ell}}(2\pi\epsilon\abs{k})K_{\abs{\ell}}(2\pi\epsilon\abs{k})\,;
\end{aligned}
\end{equation}
\begin{equation}\label{eq:D_multiplier}
\begin{aligned}
\overline{\mc{D}}[e^{2\pi i k s}e^{i\ell\theta}]&= m_\mc{D}(k,\ell)e^{2\pi i k s}\cos(\ell\theta)\,,\\
m_\mc{D}(k,\ell) &= \begin{cases}
\frac{1}{2} - 2\pi\epsilon\abs{k}I_0(2\pi\epsilon\abs{k})K_1(\pi\epsilon\abs{k})\,, & \ell=0 \\
\frac{1}{2}-2\pi\epsilon\abs{k}I_{\abs{\ell}}(2\pi\epsilon\abs{k})\big(K_{\abs{\ell}-1}(2\pi\epsilon\abs{k})+K_{\abs{\ell}+1}(2\pi\epsilon\abs{k})\big)\,, & \ell\neq 0\,.
\end{cases}
\end{aligned}
\end{equation}
where here each $I_j$ and $K_j$ denote modified Bessel functions of the first and second kind, respectively.
\end{lemma}
Furthermore, noting that by \eqref{eq:S_multiplier}, $\overline{\mc{S}}$ is invertible on the space of $\theta$-independent functions, we use $m_{\mc{S}}^{-1}(k)$ to denote
\begin{equation}\label{eq:mS_inv}
m_{\mc{S}}^{-1}(k) = \frac{1}{m_{\mc{S}}(k,0)} = \frac{1}{\epsilon I_0(2\pi\epsilon\abs{k})K_0(2\pi\epsilon\abs{k})}\,.
\end{equation}
Using the form \eqref{eq:mS_inv} of $m_{\mc{S}}^{-1}(k)$, we may show that $\overline{\mc{S}}^{-1}$ satisfies the following mapping properties.
\begin{lemma}[Inverse single layer on $\mc{C}_\epsilon$]\label{lem:S_mapping}
Given $g\in C^{1,\alpha}(\T)$ with $\int_\T g(s)\,ds=0$, we have
\begin{equation}
\norm{\overline{\mc{S}}^{-1}[g]}_{C^{0,\alpha}(\T)}\le \textcolor{black}{c\,\big(\epsilon^{-1}\norm{g}_{C^{0,\alpha}(\T)}+ |g|_{\dot C^{1,\alpha}(\T)}\big) }\,.
\end{equation}
\end{lemma}
The proofs of Lemmas \ref{lem:straight_components} and \ref{lem:S_mapping} appear in section \ref{sec:strt_periodic}.

Motivated by the explicit information for $\mc{C}_\epsilon$, we consider the components of the relation \eqref{eq:SB_NtD} along the curved filament $\Gamma_\epsilon$ as perturbations of the straight setting. \textcolor{black}{Since the straight single layer operator $\overline{\mc{S}}$ is only well defined for functions with zero mean in $s$, we will use $\P_0$ to denote the projection $\P_0 g(s,\theta)=g(s,\theta) - \int_\T g(s,\theta)\,ds$.} We then write the map $v(s)\mapsto f(s)$ as 
\begin{equation}\label{eq:SB_DtN_strt_pert}
\begin{aligned}
\textstyle (\frac{1}{2}{\bf I}-\overline{\mc{D}}-\mc{R}_\mc{D})[v(s)] &= \overline{\mc{S}}[\textcolor{black}{\P_0}w(s,\theta)] + \mc{R}_{\mc{S}}[\textcolor{black}{\P_0}w(s,\theta)] + \mc{S}\bigg[\int_{\T}w\,ds\bigg] \\
\int_0^{2\pi} w(s,\theta)\,\mc{J}_\epsilon(s,\theta)\,d\theta &= f(s)\,. 
\end{aligned}
\end{equation}
Here we are parameterizing $\bx\in\Gamma_\epsilon$ as $\bx(s,\theta)=\X(s)+\epsilon\be_r(s,\theta)$, and we consider functions defined along $\Gamma_\epsilon$ in terms of the surface parameters $s$ and $\theta$. This will allow for direct comparison with the straight filament where $\overline{\bx}\in\mc{C}_\epsilon$ is parameterized by $\overline{\bx}(s,\theta)=s\be_z+\epsilon\be_r(\theta)$.

Given surface densities $\varphi(s,\theta)$ with $\int_{\T}\varphi(s,\theta)\,ds=0$ and $\psi(s,\theta)$, the remainder terms $\mc{R}_\mc{S}$ and $\mc{R}_\mc{D}$ may then be written as
\begin{align}
\mc{R}_\mc{S}[\varphi] &= (\mc{S}-\overline{\mc{S}})[\varphi] = \int_{\Gamma_\epsilon}\big(\mc{G}(s,\theta,s',\theta')-\overline{\mc{G}}(s,\theta,s',\theta')\big)\,\varphi(s',\theta')\,\mc{J}_\epsilon(s',\theta')\,d\theta'ds'   \label{eq:RS_def} \\
\mc{R}_\mc{D}[\psi] &= (\mc{D}-\overline{\mc{D}})[\psi] = \int_{\Gamma_\epsilon}\big(K_{\mc{D}}(s,\theta,s',\theta')-\overline{K}_{\mc{D}}(s,\theta,s',\theta')\big)\,\psi(s',\theta')\,\mc{J}_\epsilon(s',\theta')\,d\theta'ds'\,. \label{eq:RD_def}
\end{align}
Accordingly, we will consider the mapping properties of each of the operators in \eqref{eq:SB_DtN_strt_pert} in terms of $s$ and $\theta$. In particular, for any $h\in C^{0,\alpha}(\Gamma_\epsilon)$, we consider $h=h(s,\theta)$ and understand the $\dot C^{0,\alpha}$ seminorm \eqref{eq:norms} along $\Gamma_\epsilon$ as 
\begin{equation}\label{eq:dot_Calpha_eps}
\abs{h}_{\dot C^{0,\alpha}} = \sup_{(s-s')^2+(\theta-\theta')^2\neq 0}\frac{\abs{h(s,\theta)-h(s',\theta')}}{\sqrt{(s-s')^2+\epsilon^2(\theta-\theta')^2}^{\,\alpha}}\,.
\end{equation}
 Note that \eqref{eq:dot_Calpha_eps} and \eqref{eq:Ckalpha} can be seen to be equivalent characterizations of $\dot C^{0,\alpha}(\Gamma_\epsilon)$ using Lemma \ref{lem:Rests}. This characterization also helps make the $\epsilon$-dependence of the $\dot C^{0,\alpha}$ seminorm along $\Gamma_\epsilon$ more explicit.

Now, noting that $\overline{\mc{D}}[v(s)]$ is independent of $\theta$ and $\int_0^{2\pi}\overline{\mc{S}}[\textcolor{black}{\P_0}w(s,\theta)]\,d\theta=\overline{\mc{S}}[\int_0^{2\pi}\textcolor{black}{\P_0}w(s,\theta)\,d\theta]$, we may multiply the first line of \eqref{eq:SB_DtN_strt_pert} by $\epsilon$ and integrate in $\theta$ to obtain 
\begin{equation}
\begin{aligned}
\textstyle 2\pi\epsilon(\frac{1}{2}{\bf I}&-\overline{\mc{D}})[v(s)]- \displaystyle \int_0^{2\pi}\mc{R}_\mc{D}[v(s)]\,\epsilon\,d\theta \\
 &= \overline{\mc{S}}\int_0^{2\pi}\textcolor{black}{\P_0}w(s,\theta)\,\epsilon\,d\theta + \int_0^{2\pi}\mc{R}_{\mc{S}}[\textcolor{black}{\P_0}w(s,\theta)]\,\epsilon\,d\theta \textcolor{black}{+ \int_0^{2\pi}\mc{S}\bigg[\int_{\T}w\,ds\bigg]\,\epsilon\,d\theta}\,. 
\end{aligned}
\end{equation}
By Lemma \ref{lem:S_mapping}, $\overline{\mc{S}}$ is invertible on the space of $\theta$-independent functions, and we may thus write
\begin{equation}\label{eq:SB_DtN_pert}
\begin{aligned}
&\textstyle \overline{\mc{L}}_\epsilon^{-1}[v(s)]- \displaystyle \overline{\mc{S}}^{-1}\int_0^{2\pi}\mc{R}_\mc{D}[v(s)]\,\epsilon\,d\theta - \overline{\mc{S}}^{-1}\int_0^{2\pi}\mc{R}_{\mc{S}}[\textcolor{black}{\P_0}w(s,\theta)]\,\epsilon\,d\theta \\
&\qquad  \textcolor{black}{- \overline{\mc{S}}^{-1}\int_0^{2\pi}\mc{S}\bigg[\int_{\T}w\,ds\bigg]\,\epsilon\,d\theta} +\int_0^{2\pi}\int_\T w(s,\theta)\,\epsilon\,ds d\theta  -\epsilon^2\int_0^{2\pi} w(s,\theta)\wh\kappa(s,\theta)\,d\theta  = f(s)\,.
\end{aligned}
\end{equation}
Here we have used that $f(s)=\int_0^{2\pi}w(s,\theta)\,\epsilon(1-\epsilon\wh\kappa)\,d\theta$, resulting in the final two terms on the left hand side. 
We thus obtain a representation of the slender body DtN map $\mc{L}_\epsilon^{-1}[v(s)]=f(s)$ as a perturbation of the straight operator $\overline{\mc{L}}_\epsilon^{-1}[v(s)]$. Much of this paper will be devoted to showing that each of the remainder terms appearing in \eqref{eq:SB_DtN_pert} are lower order with respect to regularity or size in $\epsilon$.

In particular, we show that the operators $\mc{R}_\mc{S}$ and $\mc{R}_\mc{D}$ satisfy the following mapping properties:
\begin{lemma}[Mapping properties of $\mc{R}_\mc{S}$ and $\mc{R}_\mc{D}$]\label{lem:RS_and_RD}
Let $0<\alpha<\gamma<\beta<1$ and let $\X\in C^{2,\beta}$ be as in section \ref{subsec:geom}. Given $\varphi\in C^{0,\alpha}(\Gamma_\epsilon)$ with $\int_{\T}\varphi(s,\theta)\,ds=0$, the single layer remainder $\mc{R}_\mc{S}$ as in \eqref{eq:RS_def} may be decomposed as 
\begin{align*}
\mc{R}_\mc{S}[\varphi] &= \mc{R}_{\mc{S},\epsilon}[\varphi] + \mc{R}_{\mc{S},+}[\varphi] 
\end{align*}
where
\begin{equation}\label{eq:RS_mapping}
\begin{aligned}
\norm{\mc{R}_{\mc{S},\epsilon}[\varphi]}_{C^{0,\alpha}(\Gamma_\epsilon)} &\le c(\kappa_{*,\alpha},c_\Gamma)\,\epsilon^{2-\alpha}\norm{\varphi}_{L^\infty(\Gamma_\epsilon)}\,, \;\; 
\abs{\mc{R}_{\mc{S},\epsilon}[\varphi]}_{\dot C^{1,\alpha}(\Gamma_\epsilon)} \le c(\kappa_{*,\textcolor{black}{\alpha^+}},c_\Gamma)\,\epsilon^{1-\textcolor{black}{\alpha^+}}\norm{\varphi}_{C^{0,\alpha}(\Gamma_\epsilon)}  \\
\norm{\mc{R}_{\mc{S},+}[\varphi]}_{C^{0,\alpha}(\Gamma_\epsilon)} &\le c(\kappa_{*,\alpha},c_\Gamma)\,\epsilon^{1-\alpha}\norm{\varphi}_{L^\infty(\Gamma_\epsilon)} \,, \;\;
\abs{\mc{R}_{\mc{S},+}[\varphi]}_{\dot C^{1,\gamma}(\Gamma_\epsilon)} \le c(\kappa_{*,\gamma},c_\Gamma)\,\epsilon^{1-2\gamma}\norm{\varphi}_{C^{0,\alpha}(\Gamma_\epsilon)}
\end{aligned}
\end{equation}
\textcolor{black}{for any $\alpha^+\in(\alpha, \beta]$.}
Furthermore, given $\psi\in C^{0,\gamma}(\Gamma_\epsilon)$, the double layer remainder $\mc{R}_\mc{D}$ as in \eqref{eq:RD_def} satisfies
\begin{equation}\label{eq:RD_mapping}
\norm{\mc{R}_\mc{D}[\psi]}_{C^{1,\gamma}(\Gamma_\epsilon)} \le c(\kappa_{*,\textcolor{black}{\gamma^+}},c_\Gamma)\,\epsilon^{-\textcolor{black}{\gamma^+}} \norm{\psi}_{C^{0,\gamma}(\Gamma_\epsilon)}
\end{equation}
\textcolor{black}{for any $\gamma^+\in(\gamma, \beta]$.}
\end{lemma}
The proof of Lemma \ref{lem:RS_and_RD} appears in section \ref{sec:RS_RD}. Note that the $\epsilon$-dependence in the bounds \eqref{eq:RS_mapping}, \eqref{eq:RD_mapping} is explicit.
\textcolor{black}{
In addition to Lemma \ref{lem:RS_and_RD}, we will need lemma for dealing with the $\theta$-averaged single layer operator applied to a constant-in-$s$ function on $\Gamma_\epsilon$.
\begin{lemma}[Single layer applied to constant-in-$s$]\label{lem:mean_in_s}
For $0<\alpha<\gamma\le\beta<1$, let $\X\in C^{2,\beta}$ be as in section \ref{subsec:geom}. Given $h(\theta)\in C^{0,\alpha}(2\pi\T)$, we may write
\begin{align*}
\overline{\mc{S}}^{-1}\int_0^{2\pi}\mc{S}[h(\theta)]\,\epsilon\,d\theta = \mc{H}_\epsilon[h(\theta)]+\mc{H}_+[h(\theta)]
\end{align*}
where
\begin{equation}
\begin{aligned}
\norm{\mc{H}_\epsilon[h]}_{\cdot C^{0,\alpha}(\T)} 
&\le c(\kappa_{*,\alpha^+},c_\Gamma)\,\epsilon^{2-\alpha^+}\norm{h}_{C^{0,\alpha}(2\pi\T)} \\
 \norm{\mc{H}_+[h]}_{\cdot C^{0,\gamma}(\T)} 
&\le c(\kappa_{*,\gamma},c_\Gamma)\,\epsilon^{1-\gamma}\norm{h}_{C^{0,\alpha}(2\pi\T)}
 \end{aligned}
 \end{equation} 
 \textcolor{black}{for any $\alpha^+\in(\alpha, \beta]$.}
\end{lemma} 
The proof appears in section \ref{subsec:mean_s_pf}.}

Furthermore, using an alternative layer potential representation of the full Dirichlet-to-Neumann map, which we outline in section \ref{sec:w}, we may bound the full Neumann boundary value $w(s,\theta)$ in terms of the Dirichlet data $v(s)$ as follows: 
\begin{lemma}[Bound for $w(s,\theta)$]\label{lem:w}
Let $0<\alpha<\gamma<\beta<1$ and consider $\X(s)\in C^{2,\beta}$ as in section \ref{subsec:geom}.
Given $\theta$-independent Dirichlet data $v(s)\in C^{1,\alpha}(\T)$, consider the Neumann boundary value $w(s,\theta)$ obtained from the full Dirichlet-to-Neumann map $v(s)\mapsto w(s,\theta)$ along $\Gamma_\epsilon$. We have that $w(s,\theta)$ may be decomposed as
\begin{align*}
w(s,\theta) = w_0(s,\theta) + w_+(s,\theta)
\end{align*}
where
\begin{equation}
\begin{aligned}
\norm{w_0}_{C^{0,\alpha}(\Gamma_\epsilon)} &\le c(\kappa_{*,\alpha},c_\Gamma)\,\norm{v}_{C^{1,\alpha}(\T)} \\
\norm{w_+}_{C^{0,\gamma}(\Gamma_\epsilon)} &\le c(\kappa_{*,\textcolor{black}{\gamma^+}},c_\Gamma,\epsilon)\,\norm{v}_{C^{0,\gamma}(\T)}\,.
\end{aligned}
\end{equation}
Here the $\epsilon$-dependence is explicit in the bound for $w_0$ but not for $w_+$.
\end{lemma}
The proof of Lemma \ref{lem:w} appears in section \ref{sec:w}.  


Combining Lemmas \ref{lem:straight_components}, \ref{lem:S_mapping}, \ref{lem:RS_and_RD}, \ref{lem:mean_in_s}, and \ref{lem:w}, we arrive at the decomposition \eqref{eq:thm_Leps_inv} of the slender body DtN map $\mc{L}_\epsilon^{-1}$ stated in Theorem \ref{thm:main} (see section \ref{subsec:thm_pf} for details). To obtain the decomposition \eqref{eq:thm_Leps} of the slender body NtD map $\mc{L}_\epsilon$, we will additionally require the following lemma regarding a general bound for $\mc{L}_\epsilon$ in H\"older spaces.
\begin{lemma}[Slender body NtD in H\"older spaces]\label{lem:SB_PDE_holder}
Given a curved slender body $\Sigma_\epsilon$ as in section \ref{subsec:geom} with centerline \textcolor{black}{$\X(s)\in C^{2,\alpha}(\T)$}, and given slender body Neumann data $f(s)\in C^{0,\alpha}(\T)$ with $\alpha\in(0,1)$, the slender body Neumann-to-Dirichlet map $\mc{L}_\epsilon[f]$ \eqref{eq:SB_NtD} satisfies 
\begin{equation}\label{eq:SB_PDE_holder}
\norm{\mc{L}_\epsilon[f]}_{C^{1,\alpha}(\T)}\le c(\textcolor{black}{\kappa_{*,\alpha}},\epsilon)\norm{f}_{C^{0,\alpha}(\T)}\,.
\end{equation}
\end{lemma}
The $\epsilon$-dependence in \eqref{eq:SB_PDE_holder} is not explicit. The proof of Lemma \ref{lem:SB_PDE_holder} is given in section \ref{sec:SBNtD_Holder}.

In the next section, we proceed directly to showing how the key lemmas \ref{lem:straight_components}, \ref{lem:S_mapping}, \ref{lem:RS_and_RD}, \ref{lem:mean_in_s}, \ref{lem:w}, and \ref{lem:SB_PDE_holder} combine to yield Theorem \ref{thm:main}. In the remaining sections \ref{sec:strt_periodic}, \ref{sec:RS_RD}, \ref{sec:w}, and \ref{sec:SBNtD_Holder}, we provide a proof for each of these key lemmas.

\subsection{Proof of Theorem \ref{thm:main}}\label{subsec:thm_pf} 
We first show the decomposition \eqref{eq:thm_Leps_inv} for $\mc{L}_\epsilon^{-1}$. We begin by specifying the form of the remainder terms $\mc{R}_{\rm d,\epsilon}$ and $\mc{R}_{\rm d,+}$. Using the decompositions of $\mc{R}_{\mc{S}}$ and $w(s,\theta)$ from Lemmas \ref{lem:RS_and_RD} and \ref{lem:w}, respectively, we may write 
\begin{align*}
&\mc{R}_{\rm d,\epsilon}[v(s)] =  - \overline{\mc{S}}^{-1}\int_0^{2\pi}\mc{R}_{\mc{S},\epsilon}[\P_0w_0(s,\theta)]\,\epsilon\,d\theta \textcolor{black}{- \mc{H}_\epsilon\bigg[\int_\T w_0\,ds\bigg]}-\epsilon^2\int_0^{2\pi}w_0(s,\theta)\wh\kappa(s,\theta)\,d\theta\\
&\mc{R}_{\rm d,+}[v(s)] = -\overline{\mc{S}}^{-1}\int_0^{2\pi}\mc{R}_\mc{D}[v(s)]\,\epsilon\,d\theta - \overline{\mc{S}}^{-1}\int_0^{2\pi}\mc{R}_{\mc{S},+}[\P_0w_0(s,\theta)]\,\epsilon\,d\theta \textcolor{black}{- \mc{H}_\epsilon\bigg[\int_\T w_+\,ds\bigg]}\\
&\quad - \overline{\mc{S}}^{-1}\int_0^{2\pi}\mc{R}_{\mc{S}}[w_+(s,\theta)]\,\epsilon\,d\theta  \textcolor{black}{- \mc{H}_+\bigg[\int_\T w\,ds\bigg]} -\epsilon^2\int_0^{2\pi}w_+(s,\theta)\wh\kappa(s,\theta)\,d\theta  + \int_0^{2\pi}\int_\T w \,\epsilon\,dsd\theta \,.
\end{align*}
To estimate $\mc{R}_{\rm d,\epsilon}$, we begin by noting that, using Lemmas \ref{lem:RS_and_RD} and \ref{lem:w}, we have 
\begin{align*}
\norm{\int_0^{2\pi}\mc{R}_{\mc{S},\epsilon}[\P_0w_0(s,\theta)]\,\epsilon\,d\theta}_{C^{0,\alpha}(\T)} &\le 2\pi\epsilon\norm{\mc{R}_{\mc{S},\epsilon}[\P_0w_0(s,\theta)]}_{C^{0,\alpha}(\Gamma_\epsilon)} \\
&\le \epsilon^{3-\alpha}\,c(\kappa_{*,\alpha},c_\Gamma)\norm{w_0}_{L^\infty(\Gamma_\epsilon)} \le \epsilon^{3-\alpha}\,c(\kappa_{*,\alpha},c_\Gamma)\norm{v}_{C^{1,\alpha}(\T)}\,,\\
\abs{\int_0^{2\pi}\mc{R}_{\mc{S},\epsilon}[\P_0w_0(s,\theta)]\,\epsilon\,d\theta}_{\dot C^{1,\alpha}(\T)} &\le 2\pi\epsilon\abs{\mc{R}_{\mc{S},\epsilon}[\P_0w_0(s,\theta)]}_{\dot C^{1,\alpha}(\Gamma_\epsilon)} \\
&\le \epsilon^{2-\textcolor{black}{\alpha^+}}\,c(\kappa_{*,\textcolor{black}{\alpha^+}},c_\Gamma)\norm{w_0}_{C^{0,\alpha}(\Gamma_\epsilon)} \le \epsilon^{2-\textcolor{black}{\alpha^+}}\,c(\kappa_{*,\textcolor{black}{\alpha^+}},c_\Gamma)\norm{v}_{C^{1,\alpha}(\T)}\,. 
\end{align*}
Using Lemmas \ref{lem:S_mapping} and \ref{lem:mean_in_s}, we may then estimate: 
\begin{equation}\label{eq:Rdeps_est}
\begin{aligned}
\norm{\mc{R}_{\rm d,\epsilon}[v(s)]}_{C^{0,\alpha}(\T)} &\le
\epsilon^2 \,c(\kappa_{*,\alpha})\norm{w_0}_{C^{0,\alpha}(\Gamma_\epsilon)} \textcolor{black}{ + \norm{\mc{H}_\epsilon\bigg[\int_{\T}w_0\,ds\bigg]}_{C^{0,\alpha}(\T)}} \\
&\qquad + \norm{\overline{\mc{S}}^{-1}\int_0^{2\pi}\mc{R}_{\mc{S},\epsilon}[\P_0w_0(s,\theta)]\,\epsilon\,d\theta}_{C^{0,\alpha}(\T)}\\
&\le c(\kappa_{*,\textcolor{black}{\alpha^+}},c_\Gamma)\bigg(
\epsilon^2\norm{v}_{C^{1,\alpha}(\T)} + \epsilon^{2-\textcolor{black}{\alpha^+}}\norm{w_0}_{C^{0,\alpha}(\Gamma_\epsilon)}+ \epsilon^{2-\textcolor{black}{\alpha^+}}\norm{v}_{C^{1,\alpha}(\T)}\bigg)\\
&\le \epsilon^{2-\textcolor{black}{\alpha^+}}\,c(\kappa_{*,\textcolor{black}{\alpha^+}},c_\Gamma)\norm{v}_{C^{1,\alpha}(\T)}\,.
\end{aligned}
\end{equation}

We next estimate $\mc{R}_{\rm d,+}[v(s)]$. We first note that, again using Lemmas \ref{lem:RS_and_RD} and \ref{lem:w}, we have
\begin{align*}
\norm{\int_0^{2\pi}\mc{R}_\mc{D}[v(s)]\,\epsilon\,d\theta}_{C^{1,\gamma}(\T)}&\le c(\kappa_{*,\textcolor{black}{\gamma^+}},c_\Gamma) \epsilon^{1-\textcolor{black}{\gamma^+}}\norm{v}_{C^{0,\gamma}(\T)} \\
\norm{\int_0^{2\pi}\mc{R}_{\mc{S}}[\P_0w_+(s,\theta)]\,\epsilon\,d\theta}_{C^{1,\gamma}(\T)} &\le c(\kappa_{*,\textcolor{black}{\gamma^+}},c_\Gamma) \epsilon^{2-2\textcolor{black}{\gamma^+}}\norm{w_+}_{C^{0,\gamma}(\Gamma_\epsilon)} 
\le c(\kappa_{*,\textcolor{black}{\gamma^+}},c_\Gamma,\epsilon) \norm{v}_{C^{0,\gamma}(\T)} \\
\norm{\int_0^{2\pi}\mc{R}_{\mc{S},+}[\P_0w_0(s,\theta)]\,\epsilon\,d\theta}_{C^{1,\gamma}(\T)} &\le c(\kappa_{*,\gamma},c_\Gamma)\epsilon^{2-2\gamma}\norm{w_0}_{C^{0,\alpha}(\Gamma_\epsilon)} 
\le c(\kappa_{*,\gamma},c_\Gamma)\epsilon^{2-2\gamma}\norm{v}_{C^{1,\alpha}(\T)}\,.
\end{align*}
The $\epsilon$-dependence is explicit in the first and third bounds but not the second. 
Using the above bounds along with Lemmas \ref{lem:S_mapping}, \ref{lem:mean_in_s}, and \ref{lem:w}, we may estimate 
\begin{equation}\label{eq:Rdplus_est}
\begin{aligned}
\norm{\mc{R}_{\rm d,+}[v(s)]}_{C^{0,\gamma}(\T)} &\le \norm{\overline{\mc{S}}^{-1}\int_0^{2\pi}\mc{R}_\mc{D}[v(s)]\,\epsilon\,d\theta}_{C^{0,\gamma}(\T)} + \norm{\overline{\mc{S}}^{-1}\int_0^{2\pi}\mc{R}_{\mc{S}}[\P_0w_+(s,\theta)]\,\epsilon\,d\theta}_{C^{0,\gamma}(\T)} \\
&\quad   + \norm{\overline{\mc{S}}^{-1}\int_0^{2\pi}\mc{R}_{\mc{S},+}[\P_0w_0(s,\theta)]\,\epsilon\,d\theta}_{C^{0,\gamma}(\T)} \textcolor{black}{+ \norm{\mc{H}_\epsilon\bigg[\int_\T w_+\,ds\bigg]}_{C^{0,\gamma}(\T)}}  \\
&\quad \textcolor{black}{+ \norm{\mc{H}_+\bigg[\int_\T w\,ds\bigg]}_{C^{0,\gamma}(\T)} } +\epsilon^2c(\kappa_{*,\gamma})\norm{w_+}_{C^{0,\gamma}(\Gamma_\epsilon)} +\epsilon\norm{w}_{L^\infty(\Gamma_\epsilon)}  \\
 &\le c(\kappa_{*,\textcolor{black}{\gamma^+}},c_\Gamma,\epsilon)\norm{v}_{C^{0,\gamma}(\T)} +c(\kappa_{*,\textcolor{black}{\gamma^+}},c_\Gamma)\,\epsilon^{2-\textcolor{black}{\gamma^+}}\norm{w_+}_{C^{0,\gamma}(\Gamma_\epsilon)} \\
 &\quad +c(\kappa_{*,\textcolor{black}{\gamma^+}},c_\Gamma)\,\epsilon^{1-\textcolor{black}{\gamma^+}}\norm{w}_{C^{0,\alpha}(\Gamma_\epsilon)} + c(\kappa_{*,\gamma},c_\Gamma)\,\epsilon^{2-2\gamma}\norm{v}_{C^{1,\alpha}(\T)} \\
 &\le c(\kappa_{*,\textcolor{black}{\gamma^+}},c_\Gamma,\epsilon) \norm{v}_{C^{1,\alpha}(\T)}\,.
\end{aligned}
\end{equation}
Thus we obtain the estimates \eqref{eq:thm_Leps_inv_bds} for the decomposition \eqref{eq:thm_Leps_inv} of $\mc{L}_\epsilon^{-1}$.

We next turn to the decomposition \eqref{eq:thm_Leps} for $\mc{L}_\epsilon$. For $f(s)$ satisfying $\int_\T f(s)\,ds=0$, we may use the $\mc{L}_\epsilon^{-1}$ decomposition \eqref{eq:thm_Leps_inv} to write 
\begin{align*}
v(s) &= ({\bf I}+\overline{\mc{L}}_\epsilon\P_0\mc{R}_{\rm d,\epsilon})^{-1}\overline{\mc{L}}_\epsilon[f(s)] - ({\bf I}+\overline{\mc{L}}_\epsilon\P_0\mc{R}_{\rm d,\epsilon})^{-1}\overline{\mc{L}}_\epsilon\P_0\mc{R}_{\rm d,+}[v(s)] \\
&= ({\bf I}+\overline{\mc{L}}_\epsilon\P_0\mc{R}_{\rm d,\epsilon})^{-1}\overline{\mc{L}}_\epsilon[f(s)] - ({\bf I}+\overline{\mc{L}}_\epsilon\P_0\mc{R}_{\rm d,\epsilon})^{-1}\overline{\mc{L}}_\epsilon\P_0\mc{R}_{\rm d,+}\mc{L}_\epsilon[f(s)]\,.
\end{align*}
Note that, using Lemma \ref{lem:mapping_Leps} and the estimate \eqref{eq:Rdeps_est} for $\mc{R}_{\rm d,\epsilon}$, for any $g\in C^{1,\alpha}(\T)$, we have
\begin{align*}
\norm{\overline{\mc{L}}_\epsilon\P_0\mc{R}_{\rm d,\epsilon}[g]}_{C^{1,\alpha}(\T)} &\le c\abs{\log\epsilon}^{1/2}\norm{\mc{R}_{\rm d,\epsilon}[g]}_{L^\infty(\T)}+c\epsilon^{-1}\abs{\log\epsilon}^{1/2}\norm{\mc{R}_{\rm d,\epsilon}[g]}_{\dot C^{0,\alpha}(\T)} \\
&\le \epsilon^{1-\textcolor{black}{\alpha^+}}\abs{\log\epsilon}^{1/2}\,c(\kappa_{*,\textcolor{black}{\alpha^+}},c_\Gamma)\norm{g}_{C^{1,\alpha}(\T)}\,.
\end{align*}
For $\epsilon$ sufficiently small, we may use a Neumann series to write 
\begin{equation}\label{eq:neumann_series}
({\bf I}+\overline{\mc{L}}_\epsilon\P_0\mc{R}_{\rm d,\epsilon})^{-1} = {\bf I}+\sum_{j=1}^\infty(-\overline{\mc{L}}_\epsilon\P_0\mc{R}_{\rm d,\epsilon})^j \;
=: \; {\bf I}  + \mc{R}_{\rm n,\epsilon} 
\end{equation}
where
\begin{equation}\label{eq:neumann_series_bd}
\norm{\mc{R}_{\rm n,\epsilon}[g]}_{C^{1,\alpha}(\T)}\le \epsilon^{1-\textcolor{black}{\alpha^+}}\abs{\log\epsilon}^{1/2}\,c(\kappa_{*,\textcolor{black}{\alpha^+}},c_\Gamma)\norm{g}_{C^{1,\alpha}(\T)}\,.
\end{equation}

Next, writing 
\begin{align*}
\mc{R}_{\rm n,+} = -({\bf I}+\overline{\mc{L}}_\epsilon\P_0\mc{R}_{\rm d,\epsilon})^{-1}\overline{\mc{L}}_\epsilon\P_0\mc{R}_{\rm d,+}\mc{L}_\epsilon\,,
\end{align*}
we may use \eqref{eq:neumann_series} and \eqref{eq:neumann_series_bd} along with Lemma \ref{lem:mapping_Leps}, estimate \eqref{eq:Rdplus_est}, and Lemma \ref{lem:SB_PDE_holder} to obtain 
\begin{align*}
\norm{\mc{R}_{\rm n,+}[f(s)]}_{C^{1,\gamma}(\T)} &= \norm{({\bf I}+ \mc{R}_{\rm n,\epsilon})\overline{\mc{L}}_\epsilon\P_0\mc{R}_{\rm d,+}\mc{L}_\epsilon[f(s)]}_{C^{1,\gamma}} 
\le c(\kappa_{*,\beta},c_\Gamma,\epsilon)\norm{\overline{\mc{L}}_\epsilon\P_0\mc{R}_{\rm d,+}\mc{L}_\epsilon[f(s)]}_{C^{1,\gamma}} \\
&\le c(\kappa_{*,\textcolor{black}{\gamma^+}},c_\Gamma,\epsilon)\norm{\mc{R}_{\rm d,+}\mc{L}_\epsilon[f(s)]}_{C^{0,\gamma}}
\le c(\kappa_{*,\textcolor{black}{\gamma^+}},c_\Gamma,\epsilon)\norm{\mc{L}_\epsilon[f(s)]}_{C^{1,\alpha}} \\
&\le c(\textcolor{black}{\kappa_{*,\gamma^+}},c_\Gamma,\epsilon)\norm{f}_{C^{0,\alpha}(\T)}\,.
\end{align*}
Altogether, we obtain the decomposition \eqref{eq:thm_Leps} and bounds \eqref{eq:thm_Leps_bds} for $\mc{L}_\epsilon$.
\hfill\qedsymbol \\

The remainder of the paper is structured as follows. Section \ref{sec:strt_periodic} contains the proofs of Lemmas \ref{lem:mapping_Leps}, \ref{lem:straight_components}, and \ref{lem:S_mapping} regarding mapping properties in the straight setting; section \ref{sec:RS_RD} contains the proof of the layer potential remainder estimates in Lemmas \ref{lem:RS_and_RD} and \ref{lem:mean_in_s}; section \ref{sec:w} gives the proof of Lemma \ref{lem:w} bounding the full Neumann boundary value $w(s,\theta)$; and section \ref{sec:SBNtD_Holder} has the proof of H\"older space estimates for $\mc{L}_\epsilon$ from Lemma \ref{lem:SB_PDE_holder}.


\section{Mapping properties in the straight setting}\label{sec:strt_periodic}
In this section we prove mapping properties for the slender body NtD and its components about the straight cylinder $\mc{C}_\epsilon$. 

\subsection{Symbols for single and double layer}
We begin with Lemma \ref{lem:straight_components} regarding explicit expressions for the single and double layer operators $\overline{\mc{S}}$ and $\overline{\mc{D}}$ on $\mc{C}_\epsilon$.

Note that, due to the form of the kernels \eqref{eq:str_kernels}, periodicity in $s$ will be enforced by considering only densities $\varphi$ which are 1-periodic in $s$. We may then write the operators $\overline{\mc{S}}$ and $\overline{\mc{D}}$ as 
\begin{equation}\label{eq:barS_barD}
\begin{aligned}
\overline{\mc{S}}[\varphi](s,\theta) &= \int_{-\infty}^\infty\int_0^{2\pi}\overline{\mc{G}}(s,\theta,s',\theta')\,\varphi(s',\theta') \, \epsilon\,d\theta'ds' \\
\overline{\mc{D}}[\varphi](s,\theta) &= \int_{-\infty}^\infty\int_0^{2\pi}\overline{K}_{\mc{D}}(s,\theta,s',\theta')\,\varphi(s',\theta') \, \epsilon\,d\theta'ds'
\end{aligned}
\end{equation}
where
\begin{equation}\label{eq:str_kernels}
\overline{\mc{G}}(s,\theta,s',\theta') = \frac{1}{\sqrt{(s-s')^2+4\epsilon^2\sin^2(\frac{\theta-\theta'}{2})}}\,, \;\;
\overline{K}_\mc{D}(s,\theta,s',\theta') = \frac{-2\epsilon \sin^2(\frac{\theta-\theta'}{2})}{\sqrt{(s-s')^2+4\epsilon^2\sin^2(\frac{\theta-\theta'}{2})}^{\,3}}\,.
\end{equation}
Note that in the straight setting, both $\overline{\mc{G}}$ and $\overline{K}_\mc{D}$ are convolution kernels in both $s$ and $\theta$.

We use the expressions \eqref{eq:barS_barD}, \eqref{eq:str_kernels} to derive the symbols for $\overline{\mc{S}}$ and $\overline{\mc{D}}$ given by Lemma \ref{lem:straight_components}.
\begin{proof}[Proof of Lemma \ref{lem:straight_components}]
For $k,\ell\in \Z$, we may calculate an exact expression for $\overline{\mc{S}}[e^{2\pi i k s}e^{i\ell\theta}]$. We begin with explicit integration in $s$:
\begin{align*}
\overline{\mc{S}}[e^{2\pi i k s}e^{i\ell\theta}] &= \frac{1}{4\pi}\int_{-\infty}^\infty\int_0^{2\pi} \frac{1}{\sqrt{(s')^2+4\epsilon^2\sin^2(\frac{\theta'}{2})}} e^{2\pi i k(s-s')} e^{i\ell(\theta-\theta')}\,\epsilon d\theta' ds' \\
&= \frac{1}{2\pi}\int_0^{2\pi}K_0(4\pi\epsilon \abs{k}\textstyle\sin(\frac{\theta'}{2}) )\, e^{-i\ell\theta'}\,\epsilon d\theta' \; e^{2\pi i k s}e^{i\ell\theta}\,.
\end{align*}
Now, due to symmetry in $\theta'$, we have that for $z\in \R_+$ and any integer $\ell\neq0$, 
\begin{equation}\label{eq:Bessel_sin}
\int_0^{2\pi}\sin(\ell \theta')K_0(z\textstyle\sin(\frac{\theta'}{2}) )\, d\theta'=0\,.
\end{equation}
Furthermore, using the definition of the zeroth order modified Bessel function $K_0$, we may write  
\begin{equation}\label{eq:Bessel_cos}
\begin{aligned}
\int_0^{2\pi}\cos(\ell \theta')K_0(z\textstyle\sin(\frac{\theta'}{2}) )\, d\theta' 
&= \int_0^{2\pi}\int_0^\infty \cos(\ell \theta')\frac{\cos(z t \,\sin(\frac{\theta'}{2}) )}{\sqrt{t^2+1}}\, dt\,d\theta' \\
&= 2\pi\int_0^\infty \frac{J_{2\abs{\ell}}(zt)}{\sqrt{t^2+1}}\, dt\\
&= 2\pi \textstyle I_{\abs{\ell}}(\frac{z}{2})K_{\abs{\ell}}(\frac{z}{2})\,,
\end{aligned}
\end{equation}
where $J_{2\abs{\ell}}$ and $I_{\abs{\ell}}$ are, respectively, Bessel functions and modified Bessel functions of the first kind.
Altogether, we find that $\overline{\mc{S}}[e^{2\pi i k s}e^{i\ell\theta}]$ may be explicitly evaluated to obtain
\begin{align*}
\overline{\mc{S}}[e^{2\pi i k s}e^{i\ell\theta}] &= \epsilon\, I_{\abs{\ell}}(2\pi\epsilon\abs{k})K_{\abs{\ell}}(2\pi\epsilon\abs{k})\; e^{2\pi i k s}\cos(\ell\theta)\,.
\end{align*}

For the double layer potential $\overline{\mc{D}}$, we may calculate 
\begin{align*}
\overline{\mc{D}}[e^{2\pi i k s}e^{i\ell\theta}] &= \frac{1}{4\pi}\int_{-\infty}^\infty\int_0^{2\pi}\frac{-2\epsilon\sin^2(\frac{\theta'}{2})}{\sqrt{(s')^2+4\epsilon^2\sin^2(\frac{\theta'}{2})}^{\,3}}\,e^{2\pi i k (s-s')}e^{i\ell(\theta-\theta')}\,\epsilon d\theta'ds' \\
&= -\int_0^{2\pi}\epsilon\abs{k}|\textstyle \sin(\frac{\theta'}{2})|\, K_1\big(4\pi\epsilon\abs{k}|\textstyle \sin(\frac{\theta'}{2})|\big)\,e^{-i\ell\theta'}\, d\theta'\; e^{2\pi i ks}e^{i\ell\theta}\,.
\end{align*}
For $z> 0$, we have
\begin{align*}
-\textstyle z|\sin(\frac{\theta'}{2})|K_1\big(z|\sin(\frac{\theta'}{2})|\big) = z\p_zK_0\big(z|\sin(\frac{\theta'}{2})|\big)\,,
\end{align*}
and, using \eqref{eq:Bessel_sin} and \eqref{eq:Bessel_cos} from above, we have
\begin{align*}
z\p_z\int_0^{2\pi}&\textstyle K_0\big(z|\sin(\frac{\theta'}{2})|\big)\,e^{-i\ell\theta'}\,d\theta' = \textstyle 2\pi z\p_z\big(I_{\abs{\ell}}(\frac{z}{2})K_{\abs{\ell}}(\frac{z}{2})\big) \\
&=\begin{cases}
\pi z \textstyle \big(I_1(\frac{z}{2})K_0(\frac{z}{2}) - I_0(\frac{z}{2})K_1(\frac{z}{2}) \big)\,, & \ell=0 \\
\pi z \textstyle \left[\big(I_{\abs{\ell}-1}(\frac{z}{2})+I_{\abs{\ell}+1}(\frac{z}{2})\big)K_{\abs{\ell}}(\frac{z}{2})   
- I_{\abs{\ell}}(\frac{z}{2})\big(K_{\abs{\ell}-1}(\frac{z}{2})+K_{\abs{\ell}+1}(\frac{z}{2})\big) \right]\,,  & \ell\neq 0
\end{cases}\\
&= \begin{cases}
\textstyle 2\pi - 2\pi zI_0(\frac{z}{2})K_1(\frac{z}{2})\,, & \ell=0 \\
\textstyle 2\pi - 2\pi zI_{\abs{\ell}}(\frac{z}{2})\big(K_{\abs{\ell}-1}(\frac{z}{2})+K_{\abs{\ell}+1}(\frac{z}{2})\big)\,, & \ell\neq 0\,,
\end{cases}
\end{align*}
where we have used the identity
\begin{align*}
I_{j+1}(z)K_j(z)+I_j(z)K_{j+1}(z) = \frac{1}{z}\,.
\end{align*}
Taking $z=4\pi\epsilon\abs{k}$, we then have
\begin{align*}
\overline{\mc{D}}[e^{2\pi i k s}e^{i\ell\theta}] 
&= \begin{cases}
\bigg(\frac{1}{2} - 2\pi\epsilon\abs{k}I_0(2\pi\epsilon\abs{k})K_1(\pi\epsilon\abs{k})\bigg)\, e^{2\pi i ks}\,, & \ell=0 \\
\bigg(\frac{1}{2}-2\pi\epsilon\abs{k}I_{\abs{\ell}}(2\pi\epsilon\abs{k})\big(K_{\abs{\ell}-1}(2\pi\epsilon\abs{k})+K_{\abs{\ell}+1}(2\pi\epsilon\abs{k})\big)\bigg)e^{2\pi i k s}\cos(\ell\theta)\,, & \ell\neq 0\,.
\end{cases}
\end{align*}
\end{proof}

\subsection{Mapping properties: preliminaries}
We next turn to the proofs of Lemmas \ref{lem:mapping_Leps} and \ref{lem:S_mapping} regarding mapping properties of the operators $\overline{\mc{L}}_\epsilon$ and $\overline{\mc{S}}$ in the straight setting. We will begin by building up some background tools.

We consider $C^{0,\alpha}(\T)$ as a subset of $C^{0,\alpha}(\R)$. The proof will rely on the characterization of $C^{0,\alpha}(\R)$ as the Besov space $B^\alpha_{\infty,\infty}(\R)$ (see \cite[Theorem 2.36]{bahouri2011fourier}). 
Consider a smooth cutoff function $\phi(\xi)$ supported on the annulus $\{\frac{3}{4}< \abs{\xi}< 2\}$ and satisfying $\sum_{j\in\Z}\phi(2^{-j}\xi)=1$. Let 
\begin{equation}\label{eq:dyadic_phi}
\phi_j(\xi) = \phi(2^{-j}\xi)
\end{equation}
and, given a function $g(s)$, $s\in\R$, let
\begin{align*}
 \mc{F}[g](\xi) = \int_{-\infty}^\infty g(s)\,e^{-2\pi i \xi s}\, ds 
\end{align*}
denote the Fourier transform of $g$. Define the Littlewood-Paley projection $P_jg$ onto frequencies supported in annulus $j$ by
\begin{equation}\label{eq:littlewood_p}
\mc{F}[P_j g](\xi)  = \phi_j(\xi)\mc{F}[g](\xi)\,.
\end{equation}
We also define
\begin{align*}
P_{\le j_*} := \sum_{j\le j_*} P_j \,.
\end{align*}
For $\sigma\in \R$, we may then define the homogeneous Besov seminorm
\begin{align*}
|g|_{\dot B^\sigma_{\infty,\infty}} = \sup_{j\in \Z} 2^{j\sigma}\norm{P_j g}_{L^\infty}
\end{align*}
as well as the inhomogeneous norm 
\begin{equation}\label{eq:besov}
\norm{g}_{B^\sigma_{\infty,\infty}} =\sup\big(\norm{P_{\le 0} \,g}_{L^\infty}, \sup_{j> 0} 2^{j\sigma}\norm{P_j g}_{L^\infty}\big)\,.
\end{equation}
 Letting $\lfloor\sigma\rfloor$ denote the integer part of $\sigma$, we have that the $B^\sigma_{\infty,\infty}$ norm \eqref{eq:besov} is equivalent to the $C^{\lfloor\sigma\rfloor,\sigma-\lfloor\sigma\rfloor}$ H\"older norm \eqref{eq:Ckalpha}.

Given a Fourier multiplier $m(\xi)$ and a function $g(s)$, $s\in \R$, we will use the notation
\begin{align*}
T_m g &= \mc{F}^{-1}[m\mc{F}[g]] \,, \qquad P_jT_m g = \mc{F}^{-1}[\phi_jm\mc{F}[g]] \,.
\end{align*}
To obtain the mapping properties of $T_m g$, we will need to measure the components $P_jT_m g$ of the norm $\norm{T_m g}_{B^\sigma_{\infty,\infty}}$ given by \eqref{eq:besov}. Let $M_j=\mc{F}^{-1}[\phi_jm],$ so $P_jT_m g =M_j*g$. Then by Young's inequality for convolutions we have 
\begin{align*}
\norm{P_jT_m g}_{L^\infty(\R)}=\norm{M_j*g}_{L^\infty(\R)}\le \norm{M_j}_{L^1(\R)}\norm{g}_{L^\infty(\R)}\,.
\end{align*}
The general strategy will thus be to obtain $L^1$ bounds for the functions $M_j$ in physical space. The following lemma will be useful.

\begin{lemma}[Physical space $L^1$ bounds for multipliers]\label{lem:multiplier}
For fixed $j\in \Z$, suppose that we have a Fourier multiplier $\wh M_j(\xi)\in C^\infty_0(\R)$ supported in the annulus $2^{j-1}<\abs{\xi}< 2^{j+1}$ and satisfying the bounds 
\begin{equation}\label{eq:Linfty_Mj}
\abs{\wh M_j}\le A\,, \qquad
\abs{\p_\xi^2\wh M_j}\le B
\end{equation}
for some numbers $A,B$.
Then on the physical side $M_j=\mc{F}^{-1}[\wh M_j]$ satisfies the $L^1$ estimate
\begin{equation}\label{eq:L1_Mj}
\norm{M_j}_{L^1(\R)} \lesssim 2^j \sqrt{AB}\,.
\end{equation}
\end{lemma}

\begin{proof}
Integrating the $L^\infty$ bounds \eqref{eq:Linfty_Mj} over the support of $\wh M_j$, we obtain the $L^1$ bounds 
\begin{align*}
\norm{\wh M_j}_{L^1(\R)}\lesssim 2^jA\,, \qquad 
\norm{\p_\xi^2\wh M_j}_{L^1(\R)}\lesssim 2^jB\,,
\end{align*}
which imply the physical space estimates 
\begin{align*}
\norm{M_j}_{L^\infty(\R)}\lesssim 2^jA\,, \qquad \norm{s^2 M_j}_{L^\infty(\R)}&\lesssim 2^jB\,.
\end{align*}
Using the above $L^\infty$ bounds, we have 
\begin{align*}
\int_\R \abs{M_j}\,ds = \int_{\abs{s}\le R}\abs{M_j}\,ds + \int_{\abs{s}> R}\abs{M_j}\,ds 
&\lesssim  2^j R A + 2^j \frac{B}{R}
\end{align*}
for some $R>0$. Equating $RA=\frac{B}{R}$, we find $R=\sqrt{\frac{B}{A}}$, from which we obtain \eqref{eq:L1_Mj}.
\end{proof}

The remainder of this section will be devoted to showing bounds of the form \eqref{eq:Linfty_Mj} for our multipliers $m_\epsilon$, $m_\epsilon^{-1}$, and $m_{\mc{S}}^{-1}$.

\subsection{Multiplier estimates}
Given the expressions \eqref{eq:symbol} and \eqref{eq:mS_inv} for our multipliers, we first state some general bounds for modified Bessel functions. 
\textcolor{black}{
\begin{proposition}[Modified Bessel function bounds]\label{prop:bessel_bds}
Let $K_j(z)$ denote the $j$-th order modified Bessel function of the second kind. For some constant $c>0$, we have 
\begin{align}
\abs{\frac{K_1(z)}{K_0(z)}-1-\frac{1}{2z}}&\le \frac{c}{z^2}\,, \quad
\abs{\frac{K_0(z)}{K_1(z)}-1+\frac{1}{2z}}\le \frac{c}{z^2}\,, \qquad z\ge 1\,, \label{eq:Kmod_bds} \\
\frac{c_1}{\abs{\log(z)}}&<z\frac{K_1(z)}{K_0(z)}<\frac{c_2}{\abs{\log(z)}} \,, \qquad 0<z<1 \label{eq:Kmod_bds_low} \,. 
\end{align}
Furthermore, for $I_j(z)$, the $j$-th order modified Bessel function of the first kind, we have
\begin{align}
\abs{\frac{I_1(z)}{I_0(z)}-1+\frac{1}{2z}}\le \frac{c}{z^2}\,, \qquad z\ge 1\,, \label{eq:Imod_bds} \\
\abs{\frac{I_1(z)}{I_0(z)}-\frac{z}{2}}\le c\,z^3\,, \qquad 0<z<1\,.
 \end{align} 
\end{proposition} 
The bounds of Proposition \ref{prop:bessel_bds} follow directly from well-documented large-$z$ and small-$z$ asymptotics of $K_0(z)$, $K_1(z)$, $I_0(z)$, and $I_1(z)$ (see \cite[Chapter 10]{NIST:DLMF}). 
}

Using Proposition \ref{prop:bessel_bds}, we may prove the following bounds for the multipliers $m_{\mc{S}}^{-1}$, $m_\epsilon$, and $m_\epsilon^{-1}$.
\begin{lemma}[Bounds for multipliers]\label{lem:multipliers}
The inverse single layer multiplier $m_{\mc{S}}^{-1}(\xi)$ given by \eqref{eq:mS_inv} satisfies the bounds
\begin{equation}\label{eq:S_multiplier_est}
\abs{\p_\xi^\ell m_{\mc{S}}^{-1}(\xi)} \le \begin{cases}
c\,\abs{\xi}^{1-\ell}\,, & \abs{\xi}\ge \frac{1}{2\pi\epsilon} \\
c\,\epsilon^{-1}\abs{\xi}^{-\ell}\,, & \abs{\xi} < \frac{1}{2\pi\epsilon}\,,
\end{cases} 
\ell = 0,1,2\,.
\end{equation}
Furthermore, for the DtN multiplier $m_\epsilon^{-1}(\xi)$ as in \eqref{eq:symbol}, we have that 
\begin{equation}\label{eq:meps_bd1}
\begin{aligned}
\abs{\p_\xi^\ell m_\epsilon^{-1}(\xi)}&\le c\,\epsilon\,\abs{\xi}^{1-\ell}\,,\; \ell=0,1,2\,, \qquad \abs{\xi}\ge \frac{1}{2\pi\epsilon}\,,\\
\abs{m_\epsilon^{-1}(\xi)}\le \frac{c}{\abs{\log\epsilon}}\,, \quad 
\abs{\p_\xi^\ell m_\epsilon^{-1}(\xi)}&\le \frac{c}{\abs{\xi}^\ell\abs{\log\epsilon}^2}\,, \; \ell=1,2\,, \qquad \abs{\xi}< \frac{1}{2\pi\epsilon}\,.
\end{aligned}
\end{equation}
For the NtD multiplier $m_\epsilon(\xi)=(m_\epsilon^{-1}(\xi))^{-1}$, we have
\begin{equation}\label{eq:meps_bd2}
\begin{aligned}
\abs{\p_\xi^\ell m_\epsilon(\xi)} &\le \frac{c}{\epsilon\abs{\xi}^{\ell+1}}\,, \;\ell=0,1,2\,, \qquad \abs{\xi}\ge \frac{1}{2\pi\epsilon} \\
\abs{m_\epsilon(\xi)}\le c\abs{\log\epsilon}\,, \quad 
\abs{\p_\xi^\ell m_\epsilon(\xi)} &\le \frac{c}{\abs{\xi}^\ell}\,,\qquad\ell=1,2\,, \qquad \abs{\xi}< \frac{1}{2\pi\epsilon}\,.
\end{aligned}
\end{equation}
\end{lemma}

\begin{proof}[Proof of Lemma \ref{lem:multipliers}]
We begin with the bounds \eqref{eq:S_multiplier_est} for $m_{\mc{S}}^{-1}$. Define
\begin{align*}
g(z) := \frac{1}{I_0(z)K_0(z)}
\end{align*}
and note that $g(z)>0$ for all $z>0$.
Using Lemma \ref{prop:bessel_bds}, we have that $g'(z)\ge 0$ satisfies
\begin{equation}\label{eq:Dgz_ineq0}
g'(z) = g(z) \bigg(\frac{K_1(z)}{K_0(z)}-\frac{I_1(z)}{I_0(z)} \bigg) 
\le g(z)\bigg(\frac{1}{z}+\textcolor{black}{\frac{c}{z^2}}\bigg)
\end{equation}
for $z\ge 1$. By Gr\"onwall's inequality, we then have 
\begin{equation}\label{eq:gz_ineq}
g(z) \le g(1) \,{\rm exp}\bigg(\int_1^z \bigg(\frac{1}{s}\textcolor{black}{+\frac{c}{s^2}}\bigg)\,ds\bigg) = c\,e^{\log(z)\textcolor{black}{+c(1-\frac{1}{z})}} \le c\,z\,, \qquad z\ge 1.
\end{equation}
For $0<z<1$, we simply note that since $g$ is monotone increasing, $g(z)\le g(1)$.

Using \eqref{eq:gz_ineq} in \eqref{eq:Dgz_ineq0}, we also obtain
\begin{equation}\label{eq:Dgz_ineq}
g'(z) \le \begin{cases}
c\,, & z\ge 1 \\
\frac{c}{z}\,, & z< 1\,.
\end{cases}
\end{equation}
In addition, we may calculate
\begin{align*}
\abs{g''(z)} &= \abs{\bigg(g'(z)-\frac{g(z)}{z} \bigg)\bigg(\frac{K_1(z)}{K_0(z)}-\frac{I_1(z)}{I_0(z)} \bigg) + g(z) \bigg(\frac{K_1^2(z)}{K_0^2(z)}+\frac{I_1^2(z)}{I_0^2(z)}-2 \bigg)}\\
&=\abs{\bigg(g'-\frac{g}{z} \bigg)\bigg(\frac{K_1}{K_0}-\frac{I_1}{I_0} \bigg) + g \bigg(\bigg(\frac{K_1}{K_0}-1\bigg)^2+\bigg(\frac{I_1}{I_0}-1\bigg)^2+2\bigg(\frac{K_1}{K_0}+\frac{I_1}{I_0}-2\bigg) \bigg)} \\
&\le \frac{c}{z}\abs{g'-\frac{g}{z}} + c\frac{\abs{g}}{z^2} \,,
\end{align*}
by Lemma \ref{prop:bessel_bds}. Using the bounds \eqref{eq:gz_ineq} and \eqref{eq:Dgz_ineq} for $g$ and $g'$, we then have
\begin{equation}\label{eq:D2gz_ineq}
\abs{g''(z)}\le \begin{cases}
\frac{c}{z}\,, & z \ge 1\\
\frac{c}{z^2}\,, & z<1\,.
\end{cases}
\end{equation}
Using that 
\begin{align*}
m_{\mc{S}}^{-1}(\xi) = \epsilon^{-1}g(2\pi\epsilon\abs{\xi})\,,  
\end{align*}
with each of \eqref{eq:gz_ineq}, \eqref{eq:Dgz_ineq}, and \eqref{eq:D2gz_ineq}, we obtain the bounds \eqref{eq:S_multiplier_est}.

We next turn to the bounds \eqref{eq:meps_bd1}. Let $H(z)= z\frac{K_1(z)}{K_0(z)}$. By Proposition \ref{prop:bessel_bds}, for $z\ge 1$ we have 
\begin{align*}
 \abs{H(z)-z-\frac{1}{2}} &\le \frac{c}{z}\\
 \abs{H'(z)-1} &= \abs{\frac{1}{z}(H+z)(H-z)-1} = \abs{2(H-z)+\frac{1}{z}(H-z)^2-1}\le \frac{c}{z}\\
 \abs{H''(z)} &= \abs{-\frac{1}{z^2}(H+z)(H-z)+\frac{1}{z}(H'+1)(H-z)+\frac{1}{z}(H+z)(H'-1)} \le \frac{c}{z}\,.
 \end{align*} 
Furthermore, for $z<1$ we have
\begin{align*}
\abs{H(z)}\le \frac{c}{\abs{\log(z)}}\,, \quad 
\abs{H'(z)}\le \frac{c}{z\log^2(z)}\,, \quad 
\abs{H''(z)}\le \frac{c}{z^2\log^2(z)}\,.
\end{align*}
Thus, using that  
\begin{align*}
m_\epsilon^{-1}(\xi) = 2\pi H(2\pi\epsilon\abs{\xi})\,,
\end{align*}
we obtain the bounds \eqref{eq:meps_bd1}.

We next consider $G(z) =\frac{1}{H(z)}= \frac{K_0(z)}{zK_1(z)}$. For $z\ge 1$, using Proposition \ref{prop:bessel_bds}, we have 
\begin{align*}
\abs{G(z)- \frac{1}{z}+\frac{1}{2z^2}}&\le \frac{c}{z^3} \\
\abs{G'(z)+\frac{1}{z^2}}
%
&=\abs{z\bigg(G(z)-\frac{1}{z}\bigg)^2 +
z\bigg(G(z)+\frac{1}{z}\bigg)\bigg(G(z)-\frac{1}{z}+\frac{1}{2z^2}\bigg)} \le \frac{c}{z^3}\\
\abs{G''(z)} &= \abs{\bigg(G+\frac{1}{z}\bigg)\bigg(G-\frac{1}{z} \bigg) + z\bigg(G'-\frac{1}{z^2}\bigg)\bigg(G-\frac{1}{z}\bigg) + z\bigg(G+\frac{1}{z} \bigg)\bigg(G'+\frac{1}{z^2} \bigg)} 
\le \frac{c}{z^3}\,.
\end{align*}
%
For $z<1$, we have 
\begin{align*}
 G(z) < c\abs{\log(z)}\,.
\end{align*}
By the above forms of $G'(z)$ and $G''(z)$, we may then obtain the bounds
\begin{align*}
\abs{G'(z)}\le \frac{c}{z}\,, \qquad \abs{G''(z)}\le \frac{c}{z^2}\,, \qquad 0<z<1\,.
\end{align*}
Using that 
\begin{align*}
m_\epsilon(\xi) = \frac{1}{2\pi} G(2\pi\epsilon\abs{\xi})\,,
\end{align*}
we obtain the estimates \eqref{eq:meps_bd2}.
\end{proof}

\subsection{Proof of Lemmas \ref{lem:mapping_Leps} and \ref{lem:S_mapping}}
Using the multiplier estimates of Lemma \ref{lem:multipliers}, we may finally prove Lemmas \ref{lem:mapping_Leps} and \ref{lem:S_mapping} regarding the mapping properties of $\overline{\mc{L}}_\epsilon$, $\overline{\mc{L}}_\epsilon^{-1}$, and $\overline{\mc{S}}^{-1}$.

We begin with the mapping properties of $\overline{\mc{L}}_\epsilon$ and $\overline{\mc{L}}_\epsilon^{-1}$.
\begin{proof}[Proof of Lemma \ref{lem:mapping_Leps}]
We first define the Fourier multipliers $\wh M^{(1)}_{\epsilon,j}(\xi)=\phi_j(\xi)m_\epsilon(\xi)$ and $\wh M^{(2)}_{\epsilon,j}(\xi)=\phi_j(\xi)m_\epsilon^{-1}(\xi)$ for $\phi_j$ as in \eqref{eq:dyadic_phi}. For $j_\epsilon=\frac{\abs{\log(2\pi\epsilon)}}{\log(2)}$, we may use Lemma \ref{lem:multipliers} to calculate 
\begin{align*}
\norm{\p_\xi^\ell\wh M^{(1)}_{\epsilon,j}}_{L^\infty}&\le 
\begin{cases}
c\,\epsilon^{-1}2^{-j(\ell+1)}\,,  & j\ge j_\epsilon\,, \;\ell=0,1,2\\
c\abs{\log\epsilon}\,, & j<j_\epsilon\,, \; \ell=0 \\
c\,2^{-j\ell}\,, & j<j_\epsilon\,, \;\ell=1,2\,.
\end{cases}
\end{align*}
as well as
\begin{align*}
\norm{\p_\xi^\ell\wh M^{(2)}_{\epsilon,j}}_{L^\infty}\le 
\begin{cases}
c\,\epsilon \,2^{j(1-\ell)}\,, & j\ge j_\epsilon \\
c\abs{\log\epsilon}^{-1} 2^{-j\ell}\,, & j<j_\epsilon\,, 
\end{cases}
\quad \ell=0,1,2\,.
\end{align*}
By Lemma \ref{lem:multiplier}, we thus obtain the following physical space estimates for $M^{(1)}_{\epsilon,j}$ and $M^{(2)}_{\epsilon,j}$:
\begin{align*}
\norm{M^{(1)}_{\epsilon,j}}_{L^1}\le \begin{cases}
c\, \epsilon^{-1}2^{-j}\,, & j\ge j_\epsilon\\
c\abs{\log\epsilon}^{1/2}\,, & j<j_\epsilon\,;
\end{cases}
\qquad
\norm{M^{(2)}_{\epsilon,j}}_{L^1}\le \begin{cases}
c\,\epsilon \, 2^j\,, & j\ge j_\epsilon\\
c\abs{\log\epsilon}^{-1}\,, & j<j_\epsilon\,.
\end{cases}
\end{align*}

Now, given $h\in C^{0,\alpha}(\T)$ with $\int_{\T}h(s)\,ds=0$ and $g\in C^{1,\alpha}(\T)$, we use the dyadic partition of unity \eqref{eq:dyadic_phi} to write 
\begin{align*}
\overline{\mc{L}}_\epsilon[h]= T_{m_\epsilon}h = \sum_j P_jT_{m_\epsilon}h\,; \qquad
\overline{\mc{L}}_\epsilon^{-1}[g]= T_{m_\epsilon^{-1}}g = \sum_j P_jT_{m_\epsilon^{-1}}g\,.
\end{align*}
Noting that on $\T$, 
$P_{\le 0}T_{m_\epsilon} h=P_0T_{m_\epsilon} h$ and $P_{\le 0}T_{m_\epsilon^{-1}} g=P_0T_{m_\epsilon^{-1}} g$, using the above $L^1$ bounds we may estimate 
\begin{equation}\label{eq:mult_est_linfty0}
\begin{aligned}
\norm{P_{\le 0}T_{m_\epsilon} h}_{L^\infty} &=\norm{M^{(1)}_{\epsilon,0}*h}_{L^\infty} \le \norm{M^{(1)}_{\epsilon,0}}_{L^1}\norm{h}_{L^\infty}\le c\abs{\log\epsilon}^{1/2}\norm{h}_{L^\infty}\,;\\
\norm{P_{\le 0}T_{m_\epsilon^{-1}}g}_{L^\infty} &=\norm{M^{(2)}_{\epsilon,0}*g}_{L^\infty} \le \norm{M^{(2)}_{\epsilon,0}}_{L^1}\norm{g}_{L^\infty}
\le c\abs{\log\epsilon}^{-1}\norm{g}_{L^\infty}\,.
\end{aligned}
\end{equation}
For $j> 0$, we will introduce the notation $\wt P_j =P_{j-1}+P_j+P_{j+1}$. Note that, due to the support of $\mc{F}[P_j\,\cdot]$ \eqref{eq:littlewood_p}, we have $P_jT_{m_\epsilon}h=P_jT_{m_\epsilon}\wt P_jh$ and $P_jT_{m_\epsilon^{-1}}g=P_jT_{m_\epsilon^{-1}}\wt P_jg$.
We estimate the low frequencies $0<j< j_\epsilon$ of $T_{m_\epsilon} h$ and $T_{m_\epsilon^{-1}}g$ via 
\begin{equation}\label{eq:mult_est_linfty}
\begin{aligned}
\sup_{0<j< j_\epsilon}2^{j(1+\alpha)}\norm{P_jT_{m_\epsilon} h}_{L^\infty} &= \sup_{0<j< j_\epsilon}2^{j(1+\alpha)}\norm{P_jT_{m_\epsilon}\wt P_j h}_{L^\infty}\\
&= \sup_{0<j< j_\epsilon}2^{j(1+\alpha)}\norm{M^{(1)}_{\epsilon,j}*(\wt P_j h)}_{L^\infty}\\ 
&\le 2^{j_\epsilon}\sup_{0<j< j_\epsilon}\norm{M^{(1)}_{\epsilon,j}}_{L^1}2^{j\alpha}\norm{\wt P_j h}_{L^\infty}
\le c\,\epsilon^{-1}\abs{\log\epsilon}^{1/2}\abs{h}_{\dot B^\alpha_{\infty,\infty}}\,;\\
\sup_{0<j< j_\epsilon}2^{j\alpha}\norm{P_jT_{m_\epsilon^{-1}}g}_{L^\infty} &=\sup_{0<j< j_\epsilon}2^{j\alpha}\norm{P_jT_{m_\epsilon^{-1}}\wt P_j g}_{L^\infty} 
= \sup_{0<j< j_\epsilon}2^{j\alpha}\norm{M^{(2)}_{\epsilon,j}*(\wt P_j g)}_{L^\infty}\\ 
&\le \sup_{0<j< j_\epsilon}\norm{M^{(2)}_{\epsilon,j}}_{L^1}2^{j\alpha}\norm{\wt P_j g}_{L^\infty}
\le c\,\abs{\log\epsilon}^{-1}\abs{g}_{\dot B^\alpha_{\infty,\infty}}\,.
\end{aligned}
\end{equation}
For high frequencies $j\ge j_\epsilon$, we may estimate
\begin{equation}\label{eq:mult_est_calpha}
\begin{aligned}
\sup_{j\ge j_\epsilon} 2^{j(1+\alpha)}\norm{P_jT_{m_\epsilon} h}_{L^\infty} 
&\le \sup_{j\ge j_\epsilon} 2^{j(1+\alpha)}\norm{M^{(1)}_{\epsilon,j}}_{L^1}\norm{\wt P_jh}_{L^\infty}
\le c\,\epsilon^{-1}\sup_{j\ge j_\epsilon}\,2^{j\alpha}\|\wt P_jh\|_{L^\infty}\\
&\le c\,\epsilon^{-1}\abs{h}_{\dot B^\alpha_{\infty,\infty}}\,; \\
\sup_{j\ge j_\epsilon} 2^{j\alpha}\norm{P_jT_{m_\epsilon^{-1}}g}_{L^\infty} 
&\le \sup_{j\ge j_\epsilon} 2^{j\alpha}\norm{M^{(2)}_{\epsilon,j}}_{L^1}\norm{\wt P_jg}_{L^\infty} 
\le c\,\epsilon \sup_{j\ge j_\epsilon}\,2^{j(1+\alpha)}\|\wt P_jg\|_{L^\infty}\\
&\le c\,\epsilon \abs{g}_{\dot B^{1+\alpha}_{\infty,\infty}}\,.
\end{aligned}
\end{equation}
Combining the estimates \eqref{eq:mult_est_linfty0}, \eqref{eq:mult_est_linfty}, and \eqref{eq:mult_est_calpha} and using the Besov characterization \eqref{eq:besov} of $C^{k,\alpha}(\R)$, we obtain the mapping properties \eqref{eq:est_barLeps} and \eqref{eq:est_barLepsinv} of Lemma \ref{lem:mapping_Leps}.
\end{proof}

Finally, we show Lemma \ref{lem:S_mapping} regarding the operator $\overline{\mc{S}}^{-1}$.
\begin{proof}[Proof of Lemma \ref{lem:S_mapping}]
Define the Fourier multiplier $\wh M_{\mc{S},j}(\xi)=\phi_j(\xi) m_{\mc{S}}^{-1}(\xi)$ where $\phi_j$ is given by \eqref{eq:dyadic_phi}. Using the estimates \eqref{eq:S_multiplier_est} for $m_{\mc{S}}^{-1}$ from Lemma \ref{lem:multipliers} and taking $j_\epsilon = \frac{\abs{\log(2\pi\epsilon)}}{\log(2)}$, we see that
\begin{align*}
\norm{\p_\xi^\ell\wh M_{\mc{S},j}}_{L^\infty} &\le  \begin{cases}
c\,2^{j(1-\ell)} \,, & j\ge j_\epsilon\\
c\,\epsilon^{-1}2^{-j\ell} \,, & j< j_\epsilon\,,
\end{cases} 
\qquad \ell=0,1,2\,. 
\end{align*}
Using Lemma \ref{lem:multiplier}, we then have the following physical space estimate for $M_{\mc{S},j}$:
\begin{align*}
\norm{M_{\mc{S},j}}_{L^1} \le  \begin{cases}
c\,2^j \,, & j\ge j_\epsilon\\
c\,\epsilon^{-1} \,, & j< j_\epsilon\,.
\end{cases}
\end{align*}

For $g\in C^{1,\alpha}(\T)$, we use the dyadic partition of unity \eqref{eq:dyadic_phi} to write 
\begin{align*}
\overline{\mc{S}}^{-1}[g] = T_{m_{\mc{S}}^{-1}}g = \sum_j P_j T_{m_{\mc{S}}^{-1}}g\,.
\end{align*}
Noting again that $P_{\le 0}T_{m_{\mc{S}}^{-1}}g=P_0T_{m_{\mc{S}}^{-1}}g$, we have
\begin{align*}
\norm{P_{\le 0}T_{m_{\mc{S}}^{-1}}g}_{L^\infty} &=\norm{M_{\mc{S},0}*g}_{L^\infty}
\le c\norm{M_{\mc{S},0}}_{L^1} \norm{g}_{L^\infty} 
\le c\,\epsilon^{-1}\norm{g}_{L^\infty}\,.
\end{align*}
Again letting $\wt P_jg =(P_{j-1}+P_j+P_{j+1})g$ and noting that $P_jT_{m_{\mc{S}}^{-1}}g=P_jT_{m_{\mc{S}}^{-1}}\wt P_jg$, for low frequencies $0<j< j_\epsilon$, we may estimate
\begin{align*}
\sup_{0<j< j_\epsilon} 2^{j\alpha}\norm{P_jT_{m_{\mc{S}}^{-1}}g}_{L^\infty}
&\le \sup_{0<j< j_\epsilon} 2^{j\alpha}\norm{P_jT_{m_{\mc{S}}^{-1}}\wt P_jg}_{L^\infty}
= \sup_{0<j< j_\epsilon} 2^{j\alpha}\norm{M_{\mc{S},j}*(\wt P_jg)}_{L^\infty}\\
&\le \sup_{0<j< j_\epsilon} 2^{j\alpha}\norm{M_{\mc{S},j}}_{L^1} \norm{\wt P_jg}_{L^\infty} 
\le c\,\epsilon^{-1}\sup_{0<j< j_\epsilon}\,2^{j\alpha}\|\wt P_jg\|_{L^\infty}\\
&\le c\,\epsilon^{-1}\abs{g}_{\dot B^{\alpha}_{\infty,\infty}}\,.
\end{align*}
For high frequencies $j\ge j_\epsilon$ we obtain the bound
\begin{align*}
\sup_{j\ge j_\epsilon} 2^{j\alpha}\norm{P_jT_{m_{\mc{S}}^{-1}}g}_{L^\infty}
&\le \sup_{j\ge j_\epsilon} 2^{j\alpha}\norm{P_jT_{m_{\mc{S}}^{-1}}\wt P_jg}_{L^\infty}
= \sup_{j\ge j_\epsilon} 2^{j\alpha}\norm{M_{\mc{S},j}*(\wt P_jg)}_{L^\infty}\\
&\le \sup_{j\ge j_\epsilon} 2^{j\alpha}\norm{M_{\mc{S},j}}_{L^1} \norm{\wt P_jg}_{L^\infty} 
\le c\,\sup_{j\ge j_\epsilon}\,2^{j(1+\alpha)}\|\wt P_jg\|_{L^\infty}\\
&\le c\,\abs{g}_{\dot B^{1+\alpha}_{\infty,\infty}}\,.
\end{align*}
Combining the above three estimates with the Besov characterization \eqref{eq:besov} of $C^{1,\alpha}(\R)$, we obtain Lemma \ref{lem:S_mapping}.
\end{proof}


\section{Mapping properties of remainder terms}\label{sec:RS_RD}
Here we prove Lemma \ref{lem:RS_and_RD} regarding mapping properties of the single and double layer potential remainder terms $\mc{R}_\mc{S}$ and $\mc{R}_\mc{D}$. We additionally show Lemma \ref{lem:mean_in_s} rgarding the angle-averaged single layer operator applied to constant-in-$s$ functions.

\subsection{Setup and tools}\label{subsec:tools}
The proof of Lemma \ref{lem:RS_and_RD} relies on a more careful characterization of the layer potential kernels \eqref{eq:G} and \eqref{eq:KD} along $\Gamma_\epsilon$ as functions of $s$ and $\theta$.

We begin in the straight setting. For $\overline{\bx},\overline{\bx}'\in \mc{C}_\epsilon$, the difference $\overline{\bx}-\overline{\bx}'$ may be written in terms of $s$ and $\theta$ as 
\begin{equation}\label{eq:overlinebx}
\barR:=\overline{\bx}-\overline{\bx}' = (s-s')\be_z + \epsilon(\be_r(\theta) - \be_r(\theta')) = \textstyle (s-s')\be_z + 2\epsilon\sin(\frac{\theta-\theta'}{2})\be_\theta(\frac{\theta+\theta'}{2})\,.
\end{equation}
It will be convenient to work in terms of $\bars=s-s'$ and $\bartheta=\theta-\theta'$ instead of $s'$ and $\theta'$. We may then write
\begin{equation}\label{eq:overlineR}
\overline{\bm{R}}(\bars,\theta,\bartheta) = \textstyle \bars \be_z + 2\epsilon\sin(\frac{\bartheta}{2})\be_\theta(\theta-\frac{\bartheta}{2})\,, \quad \bars = s-s'\,,\; \bartheta=\theta-\theta'\,,
\end{equation}
so, in particular, using that $\bm{n}_{x}=\be_r(\theta)$ and $\bm{n}_{x'}=\be_r(\theta-\bartheta)$, we have 
\begin{equation}\label{eq:barR_ids}
\begin{aligned}
\abs{\overline{\bm{R}}(\bars,\bartheta)}^2 &= \textstyle \bars^2+4\epsilon^2\sin^2(\frac{\bartheta}{2}) \\
\barR\cdot\overline{\bm{n}}_{x'} &= -\barR\cdot\overline{\bm{n}}_x= -\textstyle 2\epsilon\sin^2(\frac{\bartheta}{2}) \,.
\end{aligned}
\end{equation}

In the curved setting, it will again be convenient to consider
\begin{equation}\label{eq:R}
\bR := \bx-\bx'
\end{equation}
 as a function of $(s,\theta)$ and $(\bars,\bartheta)$.
 Using that $\X(s)\in C^{2,\beta}$ and the orthonormal frame ODEs \eqref{eq:frame}, we note the following expansions in $s$: 
\begin{equation}\label{eq:s_expand}
\begin{aligned}
\X(s)-\X(s-\bars) &= \bars\be_{\rm t}(s) - \bars^2\bm{Q}_{\rm t}(s,\bars) \\
\be_{\rm t}(s) - \be_{\rm t}(s-\bars) &= \bars\bm{Q}_{\rm t}(s,\bars)\\
\be_{\rm n_1}(s)-\be_{\rm n_1}(s-\bars) &= \bars(-\kappa_1(s)\be_{\rm t}(s) + \kappa_3\be_{\rm n_2}(s)) + \bars^2\bm{Q}_{\rm n_1}(s,\bars) \\
\be_{\rm n_2}(s)-\be_{\rm n_2}(s-\bars) &= \bars(-\kappa_2(s)\be_{\rm t}(s) - \kappa_3\be_{\rm n_1}(s)) + \bars^2\bm{Q}_{\rm n_2}(s,\bars) \,.
\end{aligned}
\end{equation}
%
Throughout, for any function $Q(s,\theta,\bars,\bartheta)$ of both $(s,\theta)$ and $(\bars,\bartheta)\in \T\times 2\pi\T$, we will consider the following norms with respect to both pairs of arguments: 
\begin{equation}\label{eq:Q_Cbeta_12}
\norm{Q}_{C^{0,\beta}_1}:=\sup_{\bars,\bartheta}\norm{Q(\cdot,\cdot,\bars,\bartheta)}_{C^{0,\beta}}\,, \quad \norm{Q}_{C^{0,\beta}_2}:=\sup_{s,\theta}\norm{Q(s,\theta,\cdot,\cdot)}_{C^{0,\beta}}\,.
\end{equation}
Recalling the definition \eqref{eq:kappastar} of $\kappa_{*,\beta}$, the remainder terms $\bm{Q}_{\rm t}(s,\bars)$, $\bm{Q}_{\rm n_1}(s,\bars)$, and $\bm{Q}_{\rm n_2}(s,\bars)$ in \eqref{eq:s_expand} may be shown to satisfy
%
\begin{equation}\label{eq:Q_bds2}
\begin{aligned}
&\norm{\bm{Q}_{\rm t}}_{C^{0,\beta}_\mu} \le c(\kappa_{*,\beta})\,; \qquad \bm{Q}_{\rm t}\cdot\be_{\rm t}=\bars Q_{\rm t}(s,\bars)\,, \; \norm{Q_{\rm t}}_{C^{0,\beta}_\mu} \le c(\kappa_{*,\beta})\,; \\
&\bm{Q}_{{\rm n}_j}(s,\bars)\cdot\be_{\rm t}(s) + \bm{Q}_{\rm t}(s,\bars)\cdot\be_{{\rm n}_j}(s) = \bars Q_{{\rm tn}_j}(s,\bars)\,,\; \norm{Q_{{\rm tn}_j}}_{C^{0,\beta}_\mu} \le c(\kappa_{*,\beta})\,; \\
&\norm{\bm{Q}_{{\rm n}_j}\cdot\be_{{\rm n}_i}}_{C^{0,\beta}_\mu} \le c(\kappa_{*,\beta})\,, \qquad \mu,i,j\in\{1,2\}\,.
\end{aligned}
\end{equation}

Using \eqref{eq:s_expand}-\eqref{eq:Q_bds2} and recalling the forms \eqref{eq:er_etheta} of $\be_r(s,\theta)$ and $\be_\theta(s,\theta)$, we also note the following identities and expansions: 
\begin{equation}\label{eq:QrQtheta}
\begin{aligned}
\be_r(s,\theta)-\be_r(s,\theta-\bartheta) &= \textstyle 2\sin(\frac{\bartheta}{2})\be_\theta(s,\theta-\frac{\bartheta}{2}) \,, \qquad
\textstyle\be_\theta(s,\theta-\frac{\bartheta}{2})\cdot\be_r(s,\theta) = \textstyle\sin(\frac{\bartheta}{2}) \\
\be_r(s,\theta)-\be_r(s-\bars,\theta) &= \bars\bm{Q}_r(s,\bars,\theta)\,, \qquad
\be_\theta(s,\theta)-\be_\theta(s-\bars,\theta) = \bars\bm{Q}_\theta(s,\bars,\theta) 
\end{aligned}
\end{equation}
where $\bm{Q}_r$ and $\bm{Q}_\theta$ satisfy
\begin{equation}\label{eq:Qexpand}
\begin{aligned}
\be_{\rm t}(s)\cdot\bm{Q}_r(s,\bars,\theta) &= -\wh\kappa(s,\theta) + \bars Q_{0,1}(s,\theta,\bars)  \\
\be_{\rm t}(s)\cdot\bm{Q}_\theta(s,\theta,\bars)+\bm{Q}_{\rm t}(s,\bars)\cdot\be_\theta(s,\theta) &= \bars Q_{0,2}(s,\theta,\bars)\\
\bm{Q}_r(s,\theta,\bars)-\bm{Q}_r(s,\theta-\bartheta,\bars) 
&= \textstyle 2\sin(\frac{\bartheta}{2})\bm{Q}_\theta(s,\theta-\frac{\bartheta}{2},\bars)\\
\textstyle \be_\theta(s,\theta-\frac{\bartheta}{2})\cdot\bm{Q}_r(s,\bars,\theta-\bartheta)&= \textstyle \kappa_3\cos(\frac{\bartheta}{2}) + \bars Q_{0,3}(s,\theta,\bars,\bartheta) \\
\abs{\bm{Q}_r(s,\theta,\bars)}^2&= \wh\kappa(s,\theta)^2+\kappa_3^2 +\bars Q_{0,4}(s,\theta,\bars) \\
\be_r(s,\theta)\cdot\bm{Q}_r(s,\theta,\bars) &= \bars Q_{0,5}(s,\theta,\bars)\,.
%
\end{aligned}
\end{equation}
Here, using the notation of \eqref{eq:Q_Cbeta_12}, each $Q_{0,j}$ satisfies $\norm{Q_{0,j}}_{C^{0,\beta}_\mu}\le \epsilon^{-\beta}c(\kappa_{*,\beta})$, $\mu=1,2$. The inverse $\epsilon$-dependence arises from the $\theta$-dependence in these terms, in contrast to \eqref{eq:Q_bds2}.

Using \eqref{eq:s_expand} and \eqref{eq:QrQtheta}, we may expand $\bR$ \eqref{eq:R} as
\begin{equation}\label{eq:curvedR}
\begin{aligned}
\bR(s,\theta,\bars,\bartheta) &= \X(s) - \X(s-\bars) + \epsilon\big(\be_r(s,\theta)-\be_r(s-\bars,\theta-\bartheta) \big) \\
&= \textstyle \bars\be_{\rm t}(s) + 2\epsilon\sin(\frac{\bartheta}{2})\be_\theta(s,\theta-\frac{\bartheta}{2})-\bars^2\bm{Q}_{\rm t}(s,\bars)  + \epsilon\bars\bm{Q}_r(s,\bars,\theta-\bartheta)\,.
\end{aligned}
\end{equation}
Furthermore, using the identities in \eqref{eq:Qexpand}, we may write
\begin{equation}\label{eq:Rsq}
\begin{aligned}
\abs{\bm{R}}^2 &= \textstyle
\abs{\overline{\bm{R}}}^2 +2\bars^3\be_{\rm t}\cdot\bm{Q}_{\rm t} +2\epsilon \bars^2\be_{\rm t}(s)\cdot\big(\bm{Q}_r(s,\bars,\theta)-2\sin(\frac{\bartheta}{2})\bm{Q}_\theta(s,\bars,\theta-\frac{\bartheta}{2})\big) \\
&\quad \textstyle - 4\epsilon \bars^2\sin(\frac{\bartheta}{2})\bm{Q}_{\rm t}(s,\bars)\cdot\be_\theta(s,\theta-\frac{\bartheta}{2}) + 4\epsilon^2\bars\sin(\frac{\bartheta}{2})\be_{\theta}(s,\theta-\frac{\bartheta}{2})\cdot\bm{Q}_r(s,\bars,\theta-\bartheta)  \\
&\quad \textstyle +2\epsilon \bars^3\bm{Q}_{\rm t}(s,\bars)\cdot\bm{Q}_r(s,\bars,\theta-\bartheta) +\bars^4\abs{\bm{Q}_{\rm t}}^2 +\epsilon^2\bars^2\abs{\bm{Q}_r(s,\bars,\theta)-2\sin(\frac{\bartheta}{2})\bm{Q}_\theta}^2\\
%
%
%
%
&=\textstyle
\abs{\overline{\bm{R}}}^2 + \epsilon\bars^2Q_{R,0}(s,\theta) + 2\kappa_3\epsilon^2\bars\sin(\bartheta) +\bars^4Q_{R,1}(s,\bars) \\
&\hspace{2cm} \textstyle + \epsilon \bars^3Q_{R,2}(s,\theta,\bars,\bartheta)+ \epsilon^2\bars^2\sin(\frac{\bartheta}{2})Q_{R,3}(s,\theta,\bars,\bartheta) \,.
\end{aligned}
\end{equation}
Here the functions $Q_{R,j}$ are given by
\begin{equation}\label{eq:Q_Rj}
\begin{aligned}
Q_{R,0}(s,\theta) &= -2\wh\kappa(s,\theta)+\epsilon\wh\kappa(s,\theta)^2+\epsilon\kappa_3^2\\
Q_{R,1}(s,\bars) &= 3Q_{\rm t} + \abs{\bm{Q}_{\rm t}}^2 \\
Q_{R,2}(s,\theta,\bars,\bartheta) &= \textstyle 2Q_{0,1}(s,\theta,\bars) +4\sin(\frac{\bartheta}{2})Q_{0,2}(s,\theta-\frac{\bartheta}{2},\bars) \\
&\hspace{2cm}+ 2\bm{Q}_{\rm t}(s,\bars)\cdot\bm{Q}_r(s,\theta-\bartheta,\bars)  +\epsilon Q_{0,4}(s,\theta,\bars) \\
Q_{R,3}(s,\theta,\bars,\bartheta) &= \textstyle 4Q_{0,3}(s,\theta,\bars,\bartheta)  -4\bm{Q}_\theta(s,\theta-\frac{\bartheta}{2},\bars)\cdot\bm{Q}_r(s,\bars,\theta) \\
&\hspace{4cm}\textstyle +4\sin(\frac{\bartheta}{2})|\bm{Q}_\theta(s,\bars,\theta-\frac{\bartheta}{2})|^2 \,.
\end{aligned}
\end{equation}
Note in particular that $Q_{R,0}(s,\theta)$ does not depend on $\bars$ and $\bartheta$ while the remaining $Q_{R,j}$ satisfy 
\begin{align*}
\norm{Q_{R,1}}_{C^{0,\beta}_\mu}&\le c(\kappa_{*,\beta}) \\
\norm{Q_{R,j}}_{C^{0,\beta}_\mu}&\le \epsilon^{-\beta} c(\kappa_{*,\beta})\,, \qquad \mu=1,2; \quad j=2,3\,.
\end{align*}

Using \eqref{eq:Rsq}, the difference $\frac{1}{\abs{\bm{R}}}-\frac{1}{|\overline{\bm{R}}|}$ may be rewritten as
\begin{equation}\label{eq:Rdiff}
\begin{aligned}
\frac{1}{\abs{\bm{R}}}-\frac{1}{\abs{\overline{\bm{R}}}} = \frac{\epsilon\bars^2Q_{R,0} +\kappa_3\epsilon^2\bars\sin(\bartheta) + \bars^4Q_{R,1} +\epsilon\bars^3 Q_{R,2}+ \epsilon^2 \bars^2\sin(\frac{\bartheta}{2})Q_{R,3}}{\abs{\overline{\bm{R}}}\abs{\bm{R}}(\abs{\bm{R}}+\abs{\overline{\bm{R}}})}\,.
\end{aligned}
\end{equation}
Furthermore, for any integer $k\ge 1$, we may expand
\begin{equation}\label{eq:Rdiffk}
\begin{aligned}
&\frac{1}{\abs{\bR}^k}-\frac{1}{|\barR|^k} = \bigg(\frac{1}{\abs{\bR}}-\frac{1}{|\barR|}\bigg)\sum_{\ell=0}^{k-1}\frac{1}{\abs{\bR}^\ell|\barR|^{k-1-\ell}}\\
&\quad= \frac{\epsilon\bars^2Q_{R,0} +\kappa_3\epsilon^2\bars\sin(\bartheta) + \bars^4Q_{R,1} +\epsilon\bars^3 Q_{R,2}+ \epsilon^2 \bars^2\sin(\frac{\bartheta}{2})Q_{R,3}}{\abs{\overline{\bm{R}}}\abs{\bm{R}}(\abs{\bm{R}}+\abs{\overline{\bm{R}}})}\sum_{\ell=0}^{k-1}\frac{1}{\abs{\bR}^\ell|\barR|^{k-1-\ell}}\,.
\end{aligned}
\end{equation}

We also note that, using that $\bm{n}_{x'}=\be_r(s-\bars,\theta-\bartheta)$ along with the properties \eqref{eq:Qexpand} of $\bm{Q}_r$, we have
\begin{equation}\label{eq:Rnxprime}
\begin{aligned}
\bR\cdot\bm{n}_{x'} &= -\textstyle 2\epsilon\sin^2(\frac{\bartheta}{2})+ \epsilon\bars\bm{Q}_r(s,\bars,\theta-\bartheta)\cdot\be_r(s,\theta-\bartheta) -\textstyle \bars^2(\bm{Q}_{\rm t}\cdot\be_r +  \bars\bm{Q}_{\rm t}\cdot\bm{Q}_r + \epsilon\abs{\bm{Q}_r}^2)\\
&= \barR\cdot\overline{\bm{n}}_{x'} + \bars^2Q_{\rm n'}(s,\bars,\theta-\bartheta)\,, \qquad
\norm{Q_{\rm n'}}_{C^{0,\beta}_\mu} \le \epsilon^{-\beta}c(\kappa_{*,\beta})\,, \; \mu=1,2\,.
\end{aligned}
\end{equation}
Similarly, for $\bm{n}_{x}=\be_r(s,\theta)$, we have 
\begin{equation}\label{eq:Rnx}
\bR\cdot\bm{n}_{x} = \textstyle 2\epsilon\sin^2(\frac{\bartheta}{2})+ \bars^2Q_{\rm n}(s,\bars,\theta-\bartheta) 
= \barR\cdot\overline{\bm{n}}_{x} + \bars^2Q_{\rm n}\,, \quad 
\norm{Q_{\rm n}}_{C^{0,\beta}_\mu} \le \epsilon^{-\beta}c(\kappa_{*,\beta})\,,\; \mu=1,2\,.
\end{equation}

Using the definition \eqref{eq:overlineR} of $\barR$ and the expansion \eqref{eq:curvedR} for $\bR$, we may show the following.
\begin{lemma}[Relating $\bR$, $\barR$, and ``flat"]\label{lem:Rests}
For $\barR$, $\bR$ as in \eqref{eq:overlineR}, \eqref{eq:curvedR}, the following bounds hold for $\epsilon$ sufficiently small:
\begin{align}
\abs{\abs{\bR}-|\barR|} &\le \frac{\kappa_{*}}{2}\bars^2 + c(\kappa_{*})\,\epsilon\abs{\bars} \label{eq:xest1}\\
\abs{\bR} &\ge c(\kappa_{*},c_\Gamma)|\barR| \,. \label{eq:xest2}
\end{align}
Furthermore, $\barR$ is close to the ``flat" expression $\sqrt{\bars^2+\epsilon^2\bartheta^2}$ in the following sense:
\begin{align}
\abs{\abs{\overline{\bm{R}}}-\sqrt{\bars^2+\epsilon^2\bartheta^2}} &\le (\sinh(\pi)-\pi)\epsilon|\bartheta|^3 \label{eq:flat2cyl1}\\
\abs{\overline{\bm{R}}} &\ge c\sqrt{\bars^2+\epsilon^2\bartheta^2} \,. \label{eq:flat2cyl2}
\end{align}
\end{lemma}
\begin{proof}
First, noting that 
\begin{align*}
\textstyle\abs{\bars\be_{\rm t}(s)+2\epsilon\sin(\frac{\bartheta}{2})\be_\theta(s,\theta-\frac{\bartheta}{2})} &= \textstyle \sqrt{\bars^2+4\epsilon^2\sin^2(\frac{\bartheta}{2})} = |\barR|\,,
\end{align*}
the first inequality \eqref{eq:xest1} follows from the form \eqref{eq:curvedR} of $\bR$ and the triangle inequality.

Second, using \eqref{eq:xest1} and the form \eqref{eq:barR_ids} of $|\barR|$, if $\abs{\bars}\le\frac{1}{\kappa_{*}}$, then
\begin{align*}
\abs{\bR} &\ge |\barR|-\frac{\kappa_{*}}{2}\bars^2 - c(\kappa_*)\,\epsilon\abs{\bars}
\ge \frac{1}{6}|\barR| + \frac{\abs{\bars}}{2} -\frac{\kappa_{*}}{2}\bars^2 + \frac{\abs{\bars}}{3} - c(\kappa_*)\,\epsilon\abs{\bars}
\ge \frac{1}{6}|\barR| 
\end{align*}
for $\epsilon\le \frac{1}{3c(\kappa_*)}$. If $\kappa_{*}\le 2$ we are done, since $\max\abs{\bars}=\frac{1}{2}$. If not, suppose $\frac{1}{\kappa_{*}}\le \abs{\bars}\le \frac{1}{2}$. Then, using \eqref{eq:s_expand}, we have  
\begin{align*}
\abs{\bR} &\ge \textstyle\abs{\X(s)-\X(s-\bars)}-2\epsilon\abs{\sin(\frac{\bartheta}{2})}
\ge c_\Gamma\abs{\bars}-2\epsilon
\ge \frac{c_\Gamma}{\kappa_{*}} -2\epsilon
\ge \frac{c_\Gamma}{2\kappa_{*}}
\end{align*}
as long as $\epsilon\le \frac{c_\Gamma}{4\kappa_{*}}$. Since $|\barR|\le \abs{\bars}+2\epsilon\abs{\sin(\frac{\bartheta}{2})}\le \frac{1}{2}+\frac{c_\Gamma}{2\kappa_{*}}$, we obtain \eqref{eq:xest2}.

The bound \eqref{eq:flat2cyl1} follows immediately from the triangle inequality and the fact that $|\bartheta|\le \pi$ and 
\begin{align*}
\textstyle \abs{2\sin(\frac{\bartheta}{2})-\bartheta} &\le (\sinh(\pi)-\pi)|\bartheta|^3\,.
\end{align*}
The final bound \eqref{eq:flat2cyl2} may be seen using \eqref{eq:flat2cyl1} as follows. If $(\sinh(\pi)-\pi)\bartheta^2\le \frac{1}{2}$, then
\begin{align*}
\abs{\overline{\bm{R}}}&\ge \sqrt{\bars^2+\epsilon^2\bartheta^2} - (\sinh(\pi)-\pi)\epsilon|\bartheta|^3
\ge \sqrt{\bars^2+\epsilon^2\bartheta^2} - \frac{\epsilon}{2}|\bartheta| \ge \frac{1}{2}\sqrt{\bars^2+\epsilon^2\bartheta^2}\,.
\end{align*}
If instead $\frac{1}{\sqrt{2(\sinh(\pi)-\pi)}} \le |\bartheta|\le \pi$, then 
\begin{align*}
\abs{\overline{\bm{R}}}&\ge \textstyle\sqrt{\bars^2+4\epsilon^2\sin^2\big(\frac{1}{2\sqrt{2(\sinh(\pi)-\pi)}}\big)} \ge c\sqrt{\bars^2+\epsilon^2\pi^2}\,.
\end{align*}
The above two bounds together yield \eqref{eq:flat2cyl2}.
\end{proof}

With the aid of Lemma \ref{lem:Rests}, we may obtain bounds on the following integrals. 
\begin{lemma}[Integral bounds for $\bR$ and $\barR$]\label{lem:basic_est}
Let $\bR$, $\overline{\bm{R}}$ be as in \eqref{eq:curvedR}, \eqref{eq:overlineR}, respectively. Given $0\le\alpha<1$ and an integer $k$ satisfying $k-\alpha< 2$, we have 
\begin{equation}
\int_{-1/2}^{1/2}\int_{-\pi}^{\pi}\frac{1}{\abs{\bR}^{k-\alpha}}\,\epsilon d\bartheta d\bars\, \le c(\kappa_*,c_\Gamma)\int_{-1/2}^{1/2}\int_{-\pi}^{\pi}\frac{1}{\abs{\overline{\bm{R}}}^{k-\alpha}}\,\epsilon d\bartheta d\bars \; \le \begin{cases}
c\,\epsilon^{2-k+\alpha} &\text{for } 1<k-\alpha<2\\
c\, \epsilon &\text{for } k-\alpha\le 1 \,.
\end{cases}
\end{equation}
\end{lemma}
\begin{proof}
We start by noting that the following ``flat'' estimate holds:
\begin{equation}\label{eq:flatest}
\int_{-1/2}^{1/2}\int_{-\pi}^{\pi}\frac{1}{\sqrt{\bars^2+(\epsilon\bartheta)^2}^{\,k-\alpha}} \, \epsilon \,d\bartheta d\bars \le \begin{cases}
c\,\epsilon^{2-k+\alpha} &\text{for } 1<k-\alpha<2\\
c\, \epsilon &\text{for } k\le 1 \,.
\end{cases}
\end{equation}
This may be seen by rescaling $\bartheta$ by $\epsilon^{-1}$ and writing 
\begin{align*}
  \int_{-1/2}^{1/2}\int_{-\pi\epsilon}^{\pi\epsilon}\frac{1}{\sqrt{\bars^2+\bartheta^2}^{\,k-\alpha}} \, d\bartheta d\bars = J_1+J_2\,,
\end{align*}
where
\begin{align*}
J_1 &:= \int_{-2\pi\epsilon}^{2\pi\epsilon}\int_{-\pi\epsilon}^{\pi\epsilon}\frac{1}{\sqrt{\bars^2+\bartheta^2}^{\,k-\alpha}} \, d\bartheta d\bars 
\le \int_0^{2\pi}\int_{\rho\le3\pi\epsilon} \frac{1}{\rho^{k-\alpha}} \, \rho d\rho d\phi 
\le c\,\epsilon^{2-k+\alpha}
\end{align*}
and 
\begin{align*}
J_2 &:= 2\int_{2\pi\epsilon}^{1/2}\int_{-\pi\epsilon}^{\pi\epsilon}\frac{1}{\sqrt{\bars^2+\bartheta^2}^{\,k-\alpha}} \, d\bartheta d\bars 
\end{align*}
can be directly integrated in $\bartheta$ to yield
\begin{align*}
J_2 &= 4\pi\epsilon\int_{2\pi\epsilon}^{1/2}\frac{(\bars^2+\pi^2\epsilon^2)^{\frac{2-k+\alpha}{2}}}{\bars^2} \;_2F_1\big(1,\frac{3-k+\alpha}{2},\frac{3}{2},-\frac{\pi^2\epsilon^2}{\bars^2}\big) \, d\bars\,.
\end{align*}
Here $\,_2F_1$ is the hypergeometric function which, since $\frac{\pi^2\epsilon^2}{\bars^2}\le \frac{1}{2}$, is defined by the power series
\begin{align*}
\;_2F_1\big(1,\frac{3-k+\alpha}{2},\frac{3}{2},-\frac{\pi^2\epsilon^2}{\bars^2}\big) = \sum_{n=0}^\infty a_n(k,\alpha)\bigg(-\frac{\pi^2\epsilon^2}{\bars^2} \bigg)^n\,,
\end{align*}
where $a_n(k,\alpha)$ depends on the first three arguments.
Using this expression for $J_2$, we have
\begin{align*}
J_2 &\le 4\pi\epsilon\int_{2\pi\epsilon}^{1/2}\sum_{n=0}^\infty \abs{a_n}(\pi\epsilon)^{2n}\big(\bars^{\,-2n-k+\alpha}+(\pi\epsilon)^{2-k+\alpha}\bars^{\,-2n-2} \big) \, d\bars\\
&\le c\,\epsilon \sum_{n=0}^\infty c_n(\pi\epsilon)^{2n}\big(\bars^{\,-2n-k+1+\alpha}+(\pi\epsilon)^{2-k+\alpha}\bars^{\,-2n-1} \big)\bigg|_{\bars=2\pi\epsilon}^{\bars=\frac{1}{2}}\\
&\le c\,\epsilon(1+\epsilon^{1-k+\alpha})\,.
\end{align*}
Together, the estimates for $J_1$ and $J_2$ yield \eqref{eq:flatest}.
Combining Lemma \ref{lem:Rests} and the flat estimate \eqref{eq:flatest}, we then obtain Lemma \ref{lem:basic_est}. 
\end{proof}

Note that the difference \eqref{eq:Rsq} between $\abs{\bR}^2$ and $|\barR|^2$ involves two quadratic terms, $\epsilon\bars^2Q_{R,0}(s,\theta)$ and $\kappa_3\epsilon^2\bars\sin(\frac{\bartheta}{2})$, leading to two terms in the expansion \eqref{eq:Rdiff} with a $\frac{1}{\abs{\bR}}$ singularity. In particular, this difference is small in $\epsilon$ but does not gain regularity. We thus define an intermediary term $\bR_{\rm even}$ satisfying 
\begin{equation}\label{eq:Reven}
\abs{\bR_{\rm even}}^2 := \abs{\overline{\bm{R}}}^2 +\epsilon\bars^2Q_{R,0}(s,\theta) +\kappa_3\epsilon^2\bars\sin(\bartheta)\,.
\end{equation}
From the form \eqref{eq:barR_ids} of $|\barR|$, we note that $|\barR|$ is even about both $\bars=0$ and $\bartheta=0$. Since the coefficients $Q_{R,0}(s,\theta)$ and $\kappa_3$ in \eqref{eq:Reven} are independent of $\bars$ and $\bartheta$, we have that, for nonnegative integers $n$ and $m$, the expression
\begin{align*}
\frac{\bars^n(\epsilon\sin(\frac{\bartheta}{2}))^m}{\abs{\bR_{\rm even}}^{n+m+2}}\,, \quad n+m \text{ odd},
\end{align*}
is globally odd about $\bars=\bartheta=0$. In particular, we have
\begin{equation}\label{eq:oddness}
{\rm p.v.}\int_{-1/2}^{1/2}\int_{-\pi}^{\pi}\frac{\bars^n(\epsilon\sin(\frac{\bartheta}{2}))^m }{\abs{\bR_{\rm even}}^{n+m+2}}\, \epsilon d\bartheta d\bars = 0\,,
\end{equation}
where ${\rm p.v.}$ denotes the Cauchy principal value, i.e.
\begin{align*}
{\rm p.v.}\int_{-1/2}^{1/2}\int_{-\pi}^{\pi} g(\bars,\bartheta)\,d\bartheta d\bars = \lim_{\delta\to 0^+}\bigg(\int_{-1/2}^{-\delta}+\int_{\delta}^{1/2} \bigg)\bigg(\int_{-\pi}^{-2\pi\delta}+\int_{2\pi\delta}^{\pi}\bigg)g(\bars,\bartheta)\,d\bartheta d\bars\,.
\end{align*}
In addition, we have
\begin{equation}\label{eq:Rsq2}
\begin{aligned}
\abs{\bm{R}}^2 &=\textstyle \abs{\bR_{\rm even}}^2 + \bars^4Q_{R,1} +\epsilon\bars^3 Q_{R,2}+ \epsilon^2 \bars^2\sin(\frac{\bartheta}{2})Q_{R,3}\,,\\
\frac{1}{\abs{\bm{R}}}-\frac{1}{\abs{\bR_{\rm even}}} &= \frac{\bars^4Q_{R,1} +\epsilon\bars^3 Q_{R,2}+ \epsilon^2 \bars^2\sin(\frac{\bartheta}{2})Q_{R,3}}{\abs{\bR_{\rm even}}\abs{\bm{R}}(\abs{\bm{R}}+\abs{\bR_{\rm even}})}\,,
\end{aligned}
\end{equation}
and, by a similar argument to Lemma \ref{lem:Rests}, 
\begin{equation}\label{eq:Reven_low}
\abs{\bR_{\rm even}} \ge c(\kappa_*,c_\Gamma)|\barR|\,.
\end{equation}

The expansion \eqref{eq:Rsq2} may be used to show the following bound which relies on near-cancellation in $\bars$ and $\bartheta$ upon integration. 
\begin{lemma}[Bound for integrands with odd powers]\label{lem:odd_nm}
Let $\barR$, $\bR$ be as in \eqref{eq:overlineR}, \eqref{eq:curvedR} and consider nonnegative integers $\ell,k,n,m$ such that $\ell+k+2=n+m$ and $\ell+k$ is odd. Given $g,\varphi\in C^{0,\alpha}(\Gamma_\epsilon)$, we have 
\begin{equation}\label{eq:oddlem}
\abs{{\rm p.v.}\int_{-1/2}^{1/2}\int_{-\pi}^{\pi}\frac{\bars^\ell(\epsilon\sin(\frac{\bartheta}{2}))^k g(\bars,\bartheta) }{|\barR|^m\abs{\bR}^n}\varphi(s-\bars,\theta-\bartheta)\, \epsilon d\bartheta d\bars} \le 
c(\kappa_*,c_\Gamma)\,\epsilon^\alpha\norm{g}_{C^{0,\alpha}}\norm{\varphi}_{C^{0,\alpha}} 
\end{equation}
\end{lemma}

\begin{proof}
Recalling the definition \eqref{eq:Reven} of $\abs{\bR_{\rm even}}$ and using the cancellation property \eqref{eq:oddness}, we may write
\begin{align*}
{\rm p.v.}\int_{-1/2}^{1/2}&\int_{-\pi}^{\pi}\frac{\bars^\ell(\epsilon\sin(\frac{\bartheta}{2}))^k g(\bars,\bartheta)}{|\barR|^m\abs{\bR}^n}\varphi(s-\bars,\theta-\bartheta)\, \epsilon d\bartheta d\bars = J_1 + J_2 + J_3\,,\\
J_1&= {\rm p.v.}\int_{-1/2}^{1/2}\int_{-\pi}^{\pi}\frac{ \bars^\ell(\epsilon\sin(\frac{\bartheta}{2}))^k}{|\barR|^m\abs{\bR}^n}\big(g(\bars,\bartheta)-g(0,0) \big)\varphi(s-\bars,\theta-\bartheta)\, \epsilon d\bartheta d\bars\\
J_2&= {\rm p.v.}\int_{-1/2}^{1/2}\int_{-\pi}^{\pi}\frac{\bars^\ell(\epsilon\sin(\frac{\bartheta}{2}))^k g(0,0)}{|\barR|^m} \bigg(\frac{1}{\abs{\bR}^n}-\frac{1}{\abs{\bR_{\rm even}}^n}\bigg)\varphi(s-\bars,\theta-\bartheta)\, \epsilon d\bartheta d\bars\\
J_3&= {\rm p.v.}\int_{-1/2}^{1/2}\int_{-\pi}^{\pi}\frac{\bars^\ell(\epsilon\sin(\frac{\bartheta}{2}))^k g(0,0)}{|\barR|^m\abs{\bR_{\rm even}}^n}\big(\varphi(s,\theta)-\varphi(s-\bars,\theta-\bartheta)\big)\, \epsilon d\bartheta d\bars\,.
\end{align*}

Since $g\in C^{0,\alpha}(\Gamma_\epsilon)$, we have
\begin{align*}
\abs{J_1}&\le \norm{g}_{C^{0,\alpha}}\norm{\varphi}_{L^\infty}\int_{-1/2}^{1/2}\int_{-\pi}^{\pi}\frac{ \abs{\bars}^\ell|\epsilon\sin(\frac{\bartheta}{2})|^k}{|\barR|^{\ell+k+2-\alpha}}\, \epsilon d\bartheta d\bars \le c\,\epsilon^\alpha \norm{g}_{C^{0,\alpha}}\norm{\varphi}_{L^\infty}
\end{align*}
by Lemma \ref{lem:basic_est}. To bound $J_2$, we use \eqref{eq:Rsq2} to write
\begin{equation}\label{eq:Reven_Ndiff}
\begin{aligned}
\frac{1}{\abs{\bR}^n}-\frac{1}{\abs{\bR_{\rm even}}^n}
 &= \bigg(\frac{1}{\abs{\bR}}-\frac{1}{\abs{\bR_{\rm even}}}\bigg)\sum_{j=0}^{n-1}\frac{1}{\abs{\bR}^j\abs{\bR_{\rm even}}^{n-1-j}}\\
&= \frac{\bars^4Q_{R,1} +\epsilon\bars^3 Q_{R,2}+ \epsilon^2 \bars^2\sin(\frac{\bartheta}{2})Q_{R,3} }{\abs{\bR_{\rm even}}\abs{\bm{R}}(\abs{\bm{R}}+\abs{\bR_{\rm even}})}\sum_{j=0}^{n-1}\frac{1}{\abs{\bR}^j\abs{\bR_{\rm even}}^{n-1-j}}\,.
\end{aligned}
\end{equation}
Using \eqref{eq:Reven_low}, we may also estimate
\begin{align*}
\abs{J_2}&\le c\norm{g}_{L^\infty}\norm{\varphi}_{L^\infty}\int_{-1/2}^{1/2}\int_{-\pi}^{\pi} \frac{\bars^4 +\epsilon\abs{\bars}^3 + \epsilon\bars^2|\epsilon\sin(\frac{\bartheta}{2})| }{|\barR|^m\abs{\bR_{\rm even}}\abs{\bm{R}}(\abs{\bm{R}}+\abs{\bR_{\rm even}})}\sum_{j=0}^{n-1}\frac{\abs{\bars}^\ell|\epsilon\sin(\frac{\bartheta}{2})|^k }{\abs{\bR}^j\abs{\bR_{\rm even}}^{n-1-j}}\, \epsilon d\bartheta d\bars \\
&\le c(\kappa_*,c_\Gamma)\norm{g}_{L^\infty}\norm{\varphi}_{L^\infty}\int_{-1/2}^{1/2}\int_{-\pi}^{\pi}\bigg(1+ \frac{\epsilon}{|\barR|}\bigg)\, \epsilon d\bartheta d\bars 
\le c(\kappa_*,c_\Gamma)\,\epsilon\norm{g}_{L^\infty}\norm{\varphi}_{L^\infty}\,.
\end{align*}
Finally, using that $\varphi\in C^{0,\alpha}(\Gamma_\epsilon)$, we have
\begin{align*}
\abs{J_3} &\le \norm{g}_{L^\infty}\norm{\varphi}_{C^{0,\alpha}}\int_{-1/2}^{1/2}\int_{-\pi}^{\pi}\frac{\abs{\bars}^\ell|\epsilon\sin(\frac{\bartheta}{2})|^k }{|\barR|^{\ell+k+2-\alpha}}\, \epsilon d\bartheta d\bars
\le c(\kappa_*,c_\Gamma)\, \epsilon^\alpha \norm{g}_{L^\infty}\norm{\varphi}_{C^{0,\alpha}}\,.
\end{align*}
Combining the estimates for $J_1$, $J_2$, and $J_3$, we obtain Lemma \ref{lem:odd_nm}.
\end{proof}

We will next need a lemma for estimating the $\abs{\cdot}_{\dot C^{0,\alpha}}$ seminorm of common terms appearing in the expanded boundary integral expressions.  

\begin{lemma}[Estimating $\abs{\cdot}_{\dot C^{0,\alpha}}$ seminorms]\label{lem:alpha_est}
Let $\ell,k,m,n$ be nonnegative integers satisfying either
\begin{enumerate}
\item $\ell+k=m+n-1$ or
\item $\ell+k=m+n-2$ and $\ell+k$ is odd.
\end{enumerate} 
Given $g(s,\theta,s-\bars,\theta-\bartheta)\in C^{0,\textcolor{black}{\alpha^+}}(\Gamma_\epsilon\times\Gamma_\epsilon)$ for any $\textcolor{black}{\alpha^+}>\alpha>0$, for $\barR$ and $\bR$ as in \eqref{eq:overlineR}, \eqref{eq:curvedR}, let $K_{\ell kmn}(s,\theta,\bars,\bartheta)$ denote
\begin{equation}\label{eq:K_lkmn}
 K_{\ell kmn}(s,\theta,\bars,\bartheta):= \frac{\bars^\ell(\epsilon\sin(\frac{\bartheta}{2}))^k g(s,\theta,s-\bars,\theta-\bartheta)}{\abs{\barR(s,\theta,\bars,\bartheta)}^{m}\abs{\bR(s,\theta,\bars,\bartheta)}^n}\,.
\end{equation} 
Suppose $\varphi(s,\theta)\in C^{0,\alpha}(\Gamma_\epsilon)$. Then for any $(s_0,\theta_0)\in[-1/2,1/2]\times[-\pi,\pi]$ with $s_0^2+\theta_0^2\neq 0$ we have 
\begin{equation}\label{eq:alphalem}
\begin{aligned}
&\bigg|{\rm p.v.}\int_{-s_0-1/2}^{-s_0+1/2}\int_{-\theta_0-\pi}^{-\theta_0+\pi} K_{\ell kmn}(s_0+s,\theta_0+\theta,s_0+\bars,\theta_0+\bartheta)\varphi(s-\bars,\theta-\bartheta) \, \epsilon d\bartheta d\bars \\
&\qquad\quad -  {\rm p.v.}\int_{-1/2}^{1/2}\int_{-\pi}^{\pi}K_{\ell kmn}(s,\theta,\bars,\bartheta)\varphi(s-\bars,\theta-\bartheta) \, \epsilon d\bartheta d\bars\bigg| \\
&\hspace{2cm} \le \begin{cases}
c(\kappa_*,c_\Gamma)\,\epsilon^{1-\alpha}\norm{g}_{C^{0,\alpha}_1}\norm{\varphi}_{L^\infty}\sqrt{s_0^2+\epsilon^2\theta_0^2}^{\,\alpha} & \text{in case (1)} \\
c(\kappa_*,c_\Gamma)\big(\norm{g}_{C^{0,\textcolor{black}{\alpha^+}}_1}+\norm{g}_{C^{0,\alpha}_2}\big)\norm{\varphi}_{C^{0,\alpha}}\sqrt{s_0^2+\epsilon^2\theta_0^2}^{\,\alpha} & \text{in case (2)}\,, 
\end{cases}
\end{aligned}
\end{equation}
where $\norm{\cdot}_{C^{0,\alpha}_1}$ and $\norm{\cdot}_{C^{0,\alpha}_2}$ are as defined in \eqref{eq:Q_Cbeta_12}.
\end{lemma}
Note that Lemma \ref{lem:alpha_est} gives us a bound for the $\abs{\cdot}_{\dot C^\alpha}$ seminorm for boundary integral expressions with kernels of the form \eqref{eq:K_lkmn}.

\begin{proof}
Throughout, we will use that the following sequence of inequalities holds due to Lemma \ref{lem:Rests}: 
\begin{equation}\label{eq:RtoFlat}
\abs{\bR}\ge c(\kappa_*,c_\Gamma) |\barR| \ge c(\kappa_*,c_\Gamma)\sqrt{\bars^2+\epsilon^2\bartheta^2} \,.
\end{equation}

We begin by considering the case $\sqrt{s_0^2+\epsilon^2\theta_0^2}\ge \epsilon$. Then, using Lemma \ref{lem:basic_est} in case (1) and Lemma \ref{lem:odd_nm} in case (2), we have
\begin{equation}\label{eq:alphaest1}
\begin{aligned}
&\bigg|{\rm p.v.}\int_{-s_0-1/2}^{-s_0+1/2}\int_{-\theta_0-\pi}^{-\theta_0+\pi} K_{\ell kmn}(s_0+s,\theta_0+\theta,s_0+\bars,\theta_0+\bartheta)\varphi(s-\bars,\theta-\bartheta) \, \epsilon d\bartheta d\bars \\
&\qquad\quad -  {\rm p.v.}\int_{-1/2}^{1/2}\int_{-\pi}^{\pi}K_{\ell kmn}(s,\theta,\bars,\bartheta)\varphi(s-\bars,\theta-\bartheta) \, \epsilon d\bartheta d\bars\bigg| \\
&\quad\le \abs{{\rm p.v.}\int_{-s_0-1/2}^{-s_0+1/2}\int_{-\theta_0-\pi}^{-\theta_0+\pi}K_{\ell kmn}(s_0+s,\theta_0+\theta,s_0+\bars,\theta_0+\bartheta)\varphi(s-\bars,\theta-\bartheta)\, \epsilon d\bartheta d\bars}\\
&\qquad\quad
+ \abs{{\rm p.v.}\int_{-1/2}^{1/2}\int_{-\pi}^{\pi}K_{\ell kmn}(s,\theta,\bars,\bartheta)\varphi(s-\bars,\theta-\bartheta) \, \epsilon d\bartheta d\bars} \\
&\quad \le \begin{cases}
c\,\epsilon\norm{g}_{L^\infty}\norm{\varphi}_{L^\infty}\\
c\,\epsilon^\alpha \norm{g}_{C^{0,\alpha}_2}\norm{\varphi}_{C^{0,\alpha}}
\end{cases}
\le \begin{cases}
c\,\epsilon^{1-\alpha}\norm{g}_{L^\infty}\norm{\varphi}_{L^\infty}\sqrt{s_0^2+\epsilon^2\theta_0^2}^{\,\alpha} & \text{in case (1)}\\
c\norm{g}_{C^{0,\alpha}_2}\norm{\varphi}_{C^{0,\alpha}}\sqrt{s_0^2+\epsilon^2\theta_0^2}^{\,\alpha} & \text{in case (2)}.
\end{cases}
\end{aligned}
\end{equation}

When $\sqrt{s_0^2+\epsilon^2\theta_0^2}< \epsilon$, we split the integral \eqref{eq:alphalem} into two regions. First, consider the region $I_1$ where $\sqrt{(s_0+\bars)^2+\epsilon^2(\theta_0+\bartheta)^2}\le 4\sqrt{s_0^2+\epsilon^2\theta_0^2}$. Then $\sqrt{\bars^2+\epsilon^2\bartheta^2}$ belongs to the region $I_1'$ where $\sqrt{\bars^2+\epsilon^2\bartheta^2}\le 5\sqrt{s_0^2+\epsilon^2\theta_0^2}$. In case (1), using the notation 
\begin{equation}\label{eq:barR0}
\barR_0:=\barR(s_0+s,\theta_0+\theta,s_0+\bars,\theta_0+\bartheta)\,,
\end{equation}
 we may use \eqref{eq:RtoFlat} to estimate
\begin{equation}\label{eq:alphaest2_1}
\begin{aligned}
&\bigg|\iint_{I_1}K_{\ell kmn}(s_0+s,\theta_0+\theta,s_0+\bars,\theta_0+\bartheta)\varphi(s-\bars,\theta-\bartheta) \, \epsilon d\bartheta d\bars \\
&\qquad\qquad- \iint_{I_1'}K_{\ell kmn}(s,\theta,\bars,\bartheta)\varphi(s-\bars,\theta-\bartheta) \, \epsilon d\bartheta d\bars \bigg|\\
&\qquad \le c\norm{g}_{L^\infty}\norm{\varphi}_{L^\infty}\iint_{I_1} \frac{1}{\abs{\barR_0}} \, \epsilon d\bartheta d\bars + c\norm{g}_{L^\infty}\norm{\varphi}_{L^\infty}\iint_{I_1'}\frac{1}{|\barR|} \, \epsilon d\bartheta d\bars \\
&\qquad \le c\norm{g}_{L^\infty}\norm{\varphi}_{L^\infty}\bigg(\iint_{I_1} \frac{1}{\sqrt{(s_0+\bars)^2+\epsilon^2(\theta_0+\bartheta)^2}} \, \epsilon d\bartheta d\bars + \iint_{I_1'}\frac{1}{\sqrt{\bars^2+\epsilon^2\bartheta^2}} \, \epsilon d\bartheta d\bars\bigg) \\
&\qquad \le c\norm{g}_{L^\infty}\norm{\varphi}_{L^\infty}\bigg(\iint_{\rho\le 4\sqrt{s_0^2+\epsilon^2\theta_0^2}} \frac{1}{\rho} \, \rho d\rho d\phi + \iint_{\rho\le 5\sqrt{s_0^2+\epsilon^2\theta_0^2}}\frac{1}{\rho} \, \rho d\rho d\phi\bigg) \\
&\qquad \le c\norm{g}_{L^\infty}\norm{\varphi}_{L^\infty}\sqrt{s_0^2+\epsilon^2\theta_0^2} \le c\,\epsilon^{1-\alpha}\norm{g}_{L^\infty}\norm{\varphi}_{L^\infty}\sqrt{s_0^2+\epsilon^2\theta_0^2}^{\,\alpha}\,.
\end{aligned}
\end{equation}
In the third inequality we have switched to polar coordinates $(\rho,\phi)$ as in the proof of Lemma \ref{lem:basic_est}.

In case (2), as in the proof of Lemma \ref{lem:odd_nm}, we make use of the decomposition 
\begin{equation}\label{eq:case2_decomp}
\begin{aligned}
&{\rm p.v.}\int_{-1/2}^{1/2}\int_{-\pi}^\pi K_{\ell kmn}(s,\theta,\bars,\bartheta)\varphi(s-\bars,\theta-\bartheta) \, \epsilon d\bartheta d\bars = {\rm p.v.}\int_{-1/2}^{1/2}\int_{-\pi}^\pi \big(K_1+K_2+K_3\big)\, \epsilon d\bartheta d\bars\,,\\
&\qquad K_1(s,\theta,\bars,\bartheta)= \frac{ \bars^\ell(\epsilon\sin(\frac{\bartheta}{2}))^k}{|\barR|^m\abs{\bR}^n}\big(g(s,\theta,s-\bars,\theta-\bartheta)-g(s,\theta,s,\theta) \big)\varphi(s-\bars,\theta-\bartheta) \\
&\qquad K_2(s,\theta,\bars,\bartheta)= \frac{\bars^\ell(\epsilon\sin(\frac{\bartheta}{2}))^k g(s,\theta,s,\theta)}{|\barR|^m} \bigg(\frac{1}{\abs{\bR}^n}-\frac{1}{\abs{\bR_{\rm even}}^n}\bigg)\varphi(s-\bars,\theta-\bartheta) \\
&\qquad K_3(s,\theta,\bars,\bartheta)= \frac{\bars^\ell(\epsilon\sin(\frac{\bartheta}{2}))^k g(s,\theta,s,\theta)}{|\barR|^m\abs{\bR_{\rm even}}^n}\big(\varphi(s,\theta)-\varphi(s-\bars,\theta-\bartheta)\big)\,.
\end{aligned}
\end{equation}

Over the region $I_1'$, using that 
\begin{align*}
{\rm p.v.}\iint_{I_1'}\big(\abs{K_1}+\abs{K_2}+\abs{K_3}\big)\,\epsilon d\bartheta d\bars &\le c\norm{g}_{C^{0,\alpha}_2}\norm{\varphi}_{C^{0,\alpha}}\iint_{I_1'}\frac{1}{|\barR|^{2-\alpha}}\, \epsilon d\bartheta d\bars\,,
\end{align*}
we may follow the same steps \eqref{eq:alphaest2_1} as in case (1) to obtain 
\begin{equation}\label{eq:alphaest2_2}
\begin{aligned}
&\bigg|\iint_{I_1}K_{\ell kmn}(s_0+s,\theta_0+\theta,s_0+\bars,\theta_0+\bartheta)\varphi(s-\bars,\theta-\bartheta) \, \epsilon d\bartheta d\bars \\
&\qquad- \iint_{I_1'}K_{\ell kmn}(s,\theta,\bars,\bartheta)\varphi(s-\bars,\theta-\bartheta) \, \epsilon d\bartheta d\bars \bigg|
\le c\norm{g}_{C^{0,\alpha}_2}\norm{\varphi}_{C^{0,\alpha}}\sqrt{s_0^2+\epsilon^2\theta_0^2}^{\,\alpha} \,.
\end{aligned}
\end{equation}

Finally, we consider the region $I_2$ where $\sqrt{(s_0+\bars)^2+\epsilon^2(\theta_0+\bartheta)^2}> 4\sqrt{s_0^2+\epsilon^2\theta_0^2}$, so $\sqrt{\bars^2+\epsilon^2\bartheta^2}> 3\sqrt{s_0^2+\epsilon^2\theta_0^2}$. 
In case (1), making use of \eqref{eq:RtoFlat} along with the fact that 
\begin{align*}
\sqrt{(s_0+\bars)^2+\epsilon^2(\theta_0+\bartheta)^2} \le \frac{4}{3}\sqrt{\bars^2+\epsilon^2\bartheta^2}\,,
\end{align*}
 it may be shown that  
\begin{align*}
|\wt K|&:=
\abs{K_{\ell kmn}(s_0+s,\theta_0+\theta,s_0+\bars,\theta_0+\bartheta)- K_{\ell kmn}(s,\theta,\bars,\bartheta)} \\
&\le c\norm{g}_{L^\infty}\bigg(\abs{\frac{1}{\abs{\barR_0}^{m}}-\frac{1}{|\barR|^{m}}}\frac{\abs{s_0+\bars}^\ell|\epsilon\sin(\frac{\theta_0-\bartheta}{2})|^k }{\abs{\bR_0}^n} + \abs{(s_0+\bars)^\ell-\bars^\ell}\frac{|\epsilon\sin(\frac{\theta_0-\bartheta}{2})|^k }{|\barR|^{m}\abs{\bR}^n} \\
&\quad +\abs{\frac{1}{\abs{\bR_0}^n}-\frac{1}{\abs{\bR}^n}}\frac{\abs{s_0+\bars}^\ell|\epsilon\sin(\frac{\theta_0-\bartheta}{2})|^k }{|\barR|^m}  +\abs{\textstyle(\epsilon\sin(\frac{\theta_0+\bartheta}{2}))^k - (\epsilon\sin(\frac{\bartheta}{2}))^k}\frac{\abs{\bars}^\ell }{|\barR|^{m}\abs{\bR}^n}\bigg) \\
&\quad +\abs{g(s_0+s,\theta_0+\theta,s-\bars,\theta-\bartheta)-g(s,\theta,s-\bars,\theta-\bartheta)}\frac{\abs{\bars}^\ell|\epsilon\sin(\frac{\bartheta}{2})|^k }{|\barR|^{m}\abs{\bR}^n}\\
&\le  c\norm{g}_{C^{0,\alpha}_1}\frac{\sqrt{s_0^2+\epsilon^2\theta_0^2}^\alpha }{\abs{\barR_0}} + c \norm{g}_{L^\infty}\frac{\sqrt{s_0^2+\epsilon^2\theta_0^2}^{\,\alpha}}{\abs{\barR_0}^{1+\alpha}} \,,
\end{align*}
where we again use the notation \eqref{eq:barR0}. Using Lemma \ref{lem:basic_est}, we then have
\begin{equation}\label{eq:alphaest3_1}
\begin{aligned}
&\iint_{I_2}|\wt K|\abs{\varphi(s-\bars,\theta-\bartheta)}\,\epsilon d\bartheta d\bars \\
&\qquad \le c\norm{\varphi}_{L^\infty}\sqrt{s_0^2+\epsilon^2\theta_0^2}^{\,\alpha}\bigg(\norm{g}_{C^{0,\alpha}_1}\iint_{I_2}\frac{1}{\abs{\barR_0}} \,\epsilon d\bartheta d\bars 
 + \norm{g}_{L^\infty}\iint_{I_2}\frac{1}{\abs{\barR_0}^{1+\alpha}}\,\epsilon d\bartheta d\bars\bigg) \\
&\qquad \le c\norm{\varphi}_{L^\infty}\sqrt{s_0^2+\epsilon^2\theta_0^2}^{\,\alpha}\big(\epsilon\norm{g}_{C^{0,\alpha}_1} + \epsilon^{1-\alpha} \norm{g}_{L^\infty}\big)\,.
\end{aligned}
\end{equation}

In case (2), using the decomposition \eqref{eq:case2_decomp}, for each $i\in\{1,2,3\}$ we may define 
\begin{align*}
|\wt K_i|&:= \abs{K_i(s_0+s,\theta_0+\theta,s_0+\bars,\theta_0+\bartheta)- K_i(s,\theta,\bars,\bartheta)}\,.
\end{align*}
Similar to case (1), within the region $I_2$ it may be shown that 
\begin{align*}
|\wt K_1| &\le c\norm{g}_{C^{0,\textcolor{black}{\alpha^+}}_1}\norm{\varphi}_{L^\infty}\frac{\sqrt{s_0^2+\epsilon^2\theta_0^2}^{\,\textcolor{black}{\alpha^+}}}{\abs{\barR_0}^2}\\
|\wt K_2|&\le c\norm{\varphi}_{L^\infty}\bigg(\norm{g}_{C^{0,\alpha}_1}\frac{\sqrt{s_0^2+\epsilon^2\theta_0^2}^{\,\alpha}}{\abs{\barR_0}}+\norm{g}_{L^\infty}\frac{\sqrt{s_0^2+\epsilon^2\theta_0^2}^{\,\alpha}}{\abs{\barR_0}^{1+\alpha}}\bigg)\\
|\wt K_3|&\le c\norm{\varphi}_{C^{0,\alpha}}\bigg(\norm{g}_{C^{0,\alpha}_1}\frac{\sqrt{s_0^2+\epsilon^2\theta_0^2}^{\,\alpha}}{\abs{\barR_0}^{2-\alpha}}+\norm{g}_{L^\infty}\frac{\sqrt{s_0^2+\epsilon^2\theta_0^2}}{\abs{\barR_0}^{3-\alpha}}\bigg)\,.
\end{align*}
We may bound $|\wt K_2|$ exactly as in \eqref{eq:alphaest3_1} to obtain
\begin{equation}\label{eq:alphaest3_2}
\iint_{I_2}|\wt K_2|\,\epsilon d\bartheta d\bars 
\le c\,\epsilon\norm{g}_{C^{0,\alpha}_1}\norm{\varphi}_{L^\infty}\sqrt{s_0^2+\epsilon^2\theta_0^2}^{\,\alpha} + c\,\epsilon^{1-\alpha} \norm{g}_{L^\infty}\norm{\varphi}_{L^\infty}\sqrt{s_0^2+\epsilon^2\theta_0^2}^{\,\alpha}\,.
\end{equation}
For $|\wt K_1|$ and $|\wt K_3|$, using the inequality \eqref{eq:RtoFlat} along with a switch to polar coordinates, we obtain
\begin{equation}\label{eq:alphaest3_3}
\begin{aligned}
&\iint_{I_2}\big(|\wt K_1|+|\wt K_3|\big)\,\epsilon d\bartheta d\bars \le c\norm{g}_{C^{0,\textcolor{black}{\alpha^+}}_1}\norm{\varphi}_{L^\infty}\iint_{I_2}\frac{\sqrt{s_0^2+\epsilon^2\theta_0^2}^{\,\textcolor{black}{\alpha^+}}}{\abs{\barR_0}^2} \,\epsilon d\bartheta d\bars  \\
&\quad + c\norm{\varphi}_{C^{0,\alpha}}\bigg(\norm{g}_{C^{0,\alpha}_1}\iint_{I_2}\frac{\sqrt{s_0^2+\epsilon^2\theta_0^2}^{\,\alpha}}{\abs{\barR_0}^{\,2-\alpha}} \,\epsilon d\bartheta d\bars 
 + \norm{g}_{L^\infty}\iint_{I_2}\frac{\sqrt{s_0^2+\epsilon^2\theta_0^2}}{\abs{\barR_0}^{3-\alpha}}\,\epsilon d\bartheta d\bars\bigg) \\
%
%
%
&\le c\norm{g}_{C^{0,\beta}_1}\norm{\varphi}_{C^{0,\alpha}}\bigg(\iint_{\rho>4\sqrt{s_0^2+\epsilon^2\theta_0^2}}\bigg(\frac{\sqrt{s_0^2+\epsilon^2\theta_0^2}^{\,\textcolor{black}{\alpha^+}}}{\rho^2} +\frac{\sqrt{s_0^2+\epsilon^2\theta_0^2}^{\,\alpha}}{\rho^{2-\alpha}} + \frac{\sqrt{s_0^2+\epsilon^2\theta_0^2}}{\rho^{3-\alpha}}\bigg) \rho d\rho d\phi \\
&\le c\norm{g}_{C^{0,\textcolor{black}{\alpha^+}}_1}\norm{\varphi}_{C^{0,\alpha}}\bigg(\sqrt{s_0^2+\epsilon^2\theta_0^2}^{\,\textcolor{black}{\alpha^+}}\big(1+|\log(s_0^2+\epsilon^2\theta_0^2)|\big) + \sqrt{s_0^2+\epsilon^2\theta_0^2}^{\,\alpha}\big(1+\sqrt{s_0^2+\epsilon^2\theta_0^2}^{\,\alpha}\big)  \\
&\hspace{3cm} + \sqrt{s_0^2+\epsilon^2\theta_0^2}\big(1+\sqrt{s_0^2+\epsilon^2\theta_0^2}^{\,\alpha-1}\big)\bigg)\\
&\le c\norm{g}_{C^{0,\textcolor{black}{\alpha^+}}_1}\norm{\varphi}_{C^{0,\alpha}}\sqrt{s_0^2+\epsilon^2\theta_0^2}^{\,\alpha}\,,
\end{aligned}
\end{equation}
where we have used that $\textcolor{black}{\alpha^+}>\alpha$ to absorb the logarithmic term. 

Combining the estimates \eqref{eq:alphaest1}, \eqref{eq:alphaest2_1}, \eqref{eq:alphaest2_2}, \eqref{eq:alphaest3_1}, \eqref{eq:alphaest3_2}, and \eqref{eq:alphaest3_3}, we obtain Lemma \ref{lem:alpha_est}.
\end{proof}

\subsection{Mapping properties of single layer remainder terms}
Equipped with the lemmas of section \ref{subsec:tools}, we may now prove Lemma \ref{lem:RS_and_RD} regarding mapping properties of the single and double layer remainder terms $\mc{R}_\mc{S}$ and $\mc{R}_\mc{D}$. We begin with the single layer remainder $\mc{R}_\mc{S}$, defined in \eqref{eq:RS_def}. Recalling the mapping properties of $\overline{\mc{S}}$ (see Lemma \ref{lem:straight_components}), we show that $\mc{R}_\mc{S}$ may be decomposed into two operators: one which gains more regularity than $\overline{\mc{S}}$, and one which is smaller in $\epsilon$. We restate the result \eqref{eq:RS_mapping} here for convenience. 
\begin{lemma}[Mapping properties of $\mc{R}_{\mc{S}}$, restated]\label{lem:sing_layer_remainder}
For $0<\alpha<\gamma\le\beta<1$, let $\X\in C^{2,\beta}$ be as in section \ref{subsec:geom} and consider $\mc{R}_\mc{S}$ as in \eqref{eq:RS_def}. Given $\varphi\in C^{0,\alpha}(\Gamma_\epsilon)$ with $\int_{\T}\varphi(s,\theta)\,ds=0$, we have that $\mc{R}_\mc{S}$ may be decomposed as 
\begin{align*}
\mc{R}_\mc{S}[\varphi] &= \mc{R}_{\mc{S},\epsilon}[\varphi] + \mc{R}_{\mc{S},+}[\varphi] 
\end{align*}
where
\begin{equation}\label{eq:RS_lem}
\begin{aligned}
\norm{\mc{R}_{\mc{S},\epsilon}[\varphi]}_{C^{0,\alpha}} &\le c(\kappa_{*,\alpha},c_\Gamma)\,\epsilon^{2-\alpha}\norm{\varphi}_{L^\infty}\,, \qquad 
\abs{\mc{R}_{\mc{S},\epsilon}[\varphi]}_{\dot C^{1,\alpha}} \le c(\kappa_{*,\textcolor{black}{\alpha^+}},c_\Gamma)\,\epsilon^{1-\textcolor{black}{\alpha^+}}\norm{\varphi}_{C^{0,\alpha}}  \\
\norm{\mc{R}_{\mc{S},+}[\varphi]}_{C^{0,\alpha}} &\le c(\kappa_{*,\alpha},c_\Gamma)\,\epsilon^{1-\alpha}\norm{\varphi}_{L^\infty} \,, \qquad
\abs{\mc{R}_{\mc{S},+}[\varphi]}_{\dot C^{1,\textcolor{black}{\gamma}}} \le c(\kappa_{*,\textcolor{black}{\gamma}},c_\Gamma)\,\epsilon^{1-2\textcolor{black}{\gamma}}\norm{\varphi}_{C^{0,\alpha}}\,.
\end{aligned}
\end{equation}
\end{lemma}

\begin{proof}
In order to separate $\mc{R}_{\mc{S}}$ into a piece that is very small in $\epsilon$ and a piece that gives additional smoothing, we will need to define an intermediate kernel similar to what was done with $\bR_{\rm even}$ in \eqref{eq:Reven}. We define
\begin{align*}
\bR_{\rm t}(s,\bars,\theta,\bartheta) = \bars\be_{\rm t}(s) + \epsilon(\be_r(s,\theta)-\be_r(s-\bars,\theta-\bartheta))\,, 
\end{align*}
and note that
\begin{equation}\label{eq:Rtsq}
\begin{aligned}
\abs{\bR_{\rm t}}^2 
%
%
%
&= \textstyle
\abs{\overline{\bm{R}}}^2 + \epsilon\bars^2Q_{S,0}(s,\theta,\bars,\bartheta) + 2\kappa_3\epsilon^2\bars\sin(\bartheta) + \epsilon\bars^3 Q_{S,1}(s,\theta,\bars,\bartheta) + \epsilon^2\bars^2\sin(\frac{\bartheta}{2})Q_{S,2}(s,\theta,\bars,\bartheta) \\
\abs{\bm{R}}^2 
%
%
&= \textstyle
\abs{\bR_{\rm t}}^2 + \bars^3 Q_{S,3}(s,\theta,\bars,\bartheta) + \epsilon \bars^2\sin(\frac{\bartheta}{2}) Q_{S,4}(s,\theta,\bars,\bartheta)\,,
\end{aligned}
\end{equation}
where, using the notation of section \ref{subsec:tools}, we have
\begin{align*}
Q_{S,0} &= \textstyle -2\wh\kappa(s,\theta)+\epsilon\wh\kappa(s,\theta)^2+\epsilon\kappa_3^2 -4\sin(\frac{\bartheta}{2})\be_{\rm t}(s)\cdot\bm{Q}_\theta(s,\bars,\theta-\frac{\bartheta}{2}) \\
Q_{S,1} &=  \textstyle 2\epsilon \bars^3 Q_{0,1}(s,\bars,\theta)   +\epsilon^2\bars^3 Q_{0,4}(s,\bars,\theta) \\
Q_{S,2} &=\textstyle 4Q_{0,3}(s,\theta,\bars,\bartheta) -4\bm{Q}_\theta\cdot\bm{Q}_r(s,\bars,\theta)+4\sin(\frac{\bartheta}{2})\abs{\bm{Q}_\theta}^2\\
Q_{S,3} &= \textstyle 2\be_{\rm t}\cdot\bm{Q}_{\rm t} +2\epsilon \bm{Q}_{\rm t}(s,\bars)\cdot\bm{Q}_r(s,\bars,\theta-\bartheta) +\bars\abs{\bm{Q}_{\rm t}}^2 \\
Q_{S,4} &= \textstyle - 4\bm{Q}_{\rm t}(s,\bars)\cdot\be_\theta(s,\theta-\frac{\bartheta}{2})\,.
\end{align*}
Here, using the notation \eqref{eq:Q_Cbeta_12}, each $Q_{S,j}$ satisfies the $C^{0,\beta}$ estimates
\begin{equation}\label{eq:QSj_cbeta}
\norm{Q_{S,j}}_{C^{0,\beta}_\mu} \le \epsilon^{-\beta}\,c(\kappa_{*,\beta})\,, \quad \mu=1,2\,.
\end{equation}
Note, in particular, that the difference between $\bR_{\rm t}$ and $\bR$ is higher order in $\bars$, while the difference between $\bR_{\rm t}$ and $\barR$ is higher order in $\epsilon$. As in \eqref{eq:Rdiff} and \eqref{eq:Rdiffk}, we have that the inverse differences $\frac{1}{\abs{\bR}^k}-\frac{1}{\abs{\bR_{\rm t}}^k}$ and $\frac{1}{\abs{\bR_{\rm t}}^k}-\frac{1}{|\barR|^k}$ may be written
\begin{equation}\label{eq:Rt_inv_diff}
\begin{aligned}
\frac{1}{\abs{\bR}^k}-\frac{1}{\abs{\bR_{\rm t}}^k} &= \frac{\bars^3 Q_{S,3} + \epsilon \bars^2\sin(\frac{\bartheta}{2})Q_{S,4}}{\abs{\bR}\abs{\bR_{\rm t}}(\abs{\bR}+\abs{\bR_{\rm t}})}\sum_{\ell=0}^{k-1}\frac{1}{\abs{\bR}^\ell\abs{\bR_{\rm t}}^{k-1-\ell}}\\
\frac{1}{\abs{\bR_{\rm t}}^k}-\frac{1}{|\barR|^k} &= \frac{\epsilon\bars^2Q_{S,0} + 2\kappa_3\epsilon^2\bars\sin(\bartheta) + \epsilon\bars^3 Q_{S,1} + \epsilon^2\bars^2\sin(\frac{\bartheta}{2})Q_{S,2}}{\abs{\bR_{\rm t}}|\barR|(\abs{\bR_{\rm t}}+|\barR|)}\sum_{\ell=0}^{k-1}\frac{1}{\abs{\bR_{\rm t}}^\ell|\barR|^{k-1-\ell}}\,.
\end{aligned}
\end{equation}
Furthermore, we have that $\bR_{\rm t}$ satisfies the same estimates as $\bR$ in each of Lemmas \ref{lem:Rests}, \ref{lem:basic_est}, \ref{lem:odd_nm}, and \ref{lem:alpha_est} by analogous arguments.

Using the representation \eqref{eq:barS_barD} of $\overline{\mc{S}}$ and recalling the form \eqref{eq:jacfac} of the Jacobian factor $\mc{J}_\epsilon(s-\bars,\theta-\bartheta)$, we proceed by considering the single layer remainder $\mc{R}_\mc{S}$ in four parts: 
\begin{align*}
\mc{R}_{\mc{S}}[\varphi](s,\theta) &= \mc{R}_{\mc{S},0}[\varphi](s,\theta)+\mc{R}_{\mc{S},1}[\varphi](s,\theta)+\mc{R}_{\mc{S},2}[\varphi](s,\theta)+\mc{R}_{\mc{S},3}[\varphi](s,\theta)\,,\\
\mc{R}_{\mc{S},0} &:= 
-\frac{1}{4\pi}\bigg(\int_{-\infty}^{-1/2}+\int_{1/2}^{\infty}\bigg)\int_{-\pi}^{\pi}\frac{1}{|\barR|}\,\varphi(s-\bars,\theta-\bartheta) \,\epsilon\, d\bartheta d\bars \\
\mc{R}_{\mc{S},1} &:= 
\frac{1}{4\pi}\int_{-1/2}^{1/2}\int_{-\pi}^{\pi}\bigg(\frac{1}{\abs{\bR}}-\frac{1}{\abs{\bR_{\rm t}}} \bigg)\,\varphi(s-\bars,\theta-\bartheta) \,\epsilon\, d\bartheta d\bars \\
\mc{R}_{\mc{S},2} &:= 
\frac{1}{4\pi}\int_{-1/2}^{1/2}\int_{-\pi}^{\pi}\bigg(\frac{1}{\abs{\bR_{\rm t}}}-\frac{1}{|\barR|} \bigg)\,\varphi(s-\bars,\theta-\bartheta) \,\epsilon\, d\bartheta d\bars \\
\mc{R}_{\mc{S},3} &:= -\frac{1}{4\pi}\int_{-1/2}^{1/2}\int_{-\pi}^{\pi}\frac{1}{\abs{\bR}}\,\varphi(s-\bars,\theta-\bartheta) \,\epsilon^2\wh\kappa(s-\bars,\theta-\bartheta)\, d\bartheta d\bars \,.
\end{align*}
We begin by bounding the smooth remainder $\mc{R}_{\mc{S},0}$ away from the singularity at $\bars=\bartheta=0$. Since $\int_0^1\varphi(s,\theta)\,ds=0$, we may write $\varphi(s,\theta)=\p_s\Phi(s,\theta)$ for some $\Phi$ with $\Phi(0,\theta)=0$. Then, integrating by parts in $\bars$, we have
\begin{align*}
\mc{R}_{\mc{S},0} &= 
\frac{1}{4\pi}\bigg(\int_{-\infty}^{-1/2}+\int_{1/2}^{\infty}\bigg)\int_{-\pi}^{\pi}\p_{\bars}\bigg(\frac{1}{|\barR|}\bigg)\,\Phi(s-\bars,\theta-\bartheta) \,\epsilon\, d\bartheta d\bars \\
&\qquad - \frac{1}{4\pi}\int_{-\pi}^{\pi}\frac{1}{|\barR|}\,\Phi(s-\bars,\theta-\bartheta) \,\epsilon\, d\bartheta \bigg|_{\bars=-1/2}^{\bars=1/2}\,.
\end{align*}
In particular, 
\begin{align*}
\abs{\mc{R}_{\mc{S},0}} &\le c\,\epsilon\norm{\Phi}_{L^\infty}\int_{1/2}^{\infty}\frac{1}{\bars^2} d\bars +c\,\epsilon\norm{\Phi}_{L^\infty} \le c\,\epsilon\norm{\Phi}_{L^\infty}= c\,\epsilon\norm{\int_0^s\varphi\,ds}_{L^\infty}\le c\,\epsilon\norm{\varphi}_{L^\infty}\,.
\end{align*}
Furthermore, we have
\begin{align*}
\abs{\mc{R}_{\mc{S},0}}_{\dot C^2} &\le c\,\epsilon\norm{\varphi}_{L^\infty}\int_{1/2}^{\infty}\bigg(\frac{1}{\bars^3}+\frac{\epsilon}{\bars^4}+\frac{\epsilon^2}{\bars^5}\bigg) d\bars \le c\,\epsilon\norm{\varphi}_{L^\infty}\,.
\end{align*}

Moving on to the remainders $\mc{R}_{\mc{S},1}$, $\mc{R}_{\mc{S},2}$, and $\mc{R}_{\mc{S},3}$, we may use \eqref{eq:Rt_inv_diff} and Lemma \ref{lem:basic_est} to estimate
\begin{align*}
\abs{\mc{R}_{\mc{S},1}} &\le 
 c\norm{\varphi}_{L^\infty}\int_{-1/2}^{1/2}\int_{-\pi}^{\pi}\frac{\abs{\bars}^3\abs{Q_{S,3}}+\bars^2|\epsilon\sin(\frac{\bartheta}{2})|\abs{Q_{S,4}}}{\abs{\bR}\abs{\bR_{\rm t}}(\abs{\bR}+\abs{\bR_{\rm t}})}\,\epsilon\, d\bartheta d\bars 
\le c(\kappa_*,c_\Gamma)\,\epsilon\norm{\varphi}_{L^\infty} \\
\abs{\mc{R}_{\mc{S},2}} &\le  
c(\kappa_*)\norm{\varphi}_{L^\infty}\int_{-1/2}^{1/2}\int_{-\pi}^{\pi}\frac{\epsilon\bars^2 + \epsilon^2\abs{\bars}|\sin(\bartheta)| }{\abs{\bR_{\rm t}}|\barR|(\abs{\bR_{\rm t}}+|\barR|)} \,\epsilon\, d\bartheta d\bars 
\le c(\kappa_*,c_\Gamma)\,\epsilon^2\norm{\varphi}_{L^\infty} \\
\abs{\mc{R}_{\mc{S},3}} &\le c(\kappa_*)\norm{\varphi}_{L^\infty}\int_{-1/2}^{1/2}\int_{-\pi}^{\pi}\frac{1}{\abs{\bR}} \,\epsilon^2\, d\bartheta d\bars
\le c(\kappa_*,c_\Gamma)\,\epsilon^2\norm{\varphi}_{L^\infty}\,.
\end{align*}
Using Lemma \ref{lem:alpha_est}, case (1), we may similarly estimate
\begin{align*}
\abs{\mc{R}_{\mc{S},1}}_{\dot C^{0,\alpha}} &\le c(\kappa_{*,\alpha},c_\Gamma)\,\epsilon^{1-\alpha}\norm{\varphi}_{L^\infty} \\
\abs{\mc{R}_{\mc{S},2}}_{\dot C^{0,\alpha}} &\le c(\kappa_{*,\alpha},c_\Gamma)\,\epsilon^{2-\alpha}\norm{\varphi}_{L^\infty} \\
\abs{\mc{R}_{\mc{S},3}}_{\dot C^{0,\alpha}} &\le c(\kappa_{*,\alpha},c_\Gamma)\,\epsilon^{2-\alpha}\norm{\varphi}_{L^\infty}\,.
\end{align*}
We next bound derivatives of $\mc{R}_{\mc{S},j}$, $j=1,2,3$, in $s$ and $\theta$. We begin with $\p_s\mc{R}_{\mc{S}}$ and note that
\begin{align*}
\p_s\bR_{\rm t} &= (1-\epsilon\wh\kappa(s,\theta))\be_{\rm t}(s) + \bars(\kappa_1\be_{\rm n_1}(s)+\kappa_2\be_{\rm n_2}(s))+\epsilon \kappa_3\be_\theta(s,\theta) 
\end{align*}
while
\begin{align*}
\p_s\bR&= (1-\epsilon\wh\kappa(s,\theta))\be_{\rm t}(s)+\epsilon\kappa_3\be_\theta(s,\theta)\,.
\end{align*}
We may then write

\begin{equation}\label{eq:psRt_1}
\begin{aligned}
\p_s\bR_{\rm t}\cdot\bR_{\rm t} &= \p_s\barR\cdot\barR + \epsilon\bars Q_{S,5}(s,\theta,\bars,\bartheta)-\epsilon^2 \kappa_3\sin(\bartheta) + \epsilon\bars^2 Q_{S,6}(s,\theta,\bars,\bartheta)\,, \\
Q_{S,5} &= -\wh\kappa(s,\theta-\bartheta)-(1-\epsilon\wh\kappa)\big(\wh\kappa(s,\theta)-\bars Q_{0,1}(s,\bars,\theta)\big)+\epsilon\kappa_3^2\cos(\bartheta) \\
Q_{S,6} &= (\kappa_1\be_{\rm n_1}(s)+\kappa_2\be_{\rm n_2}(s))\cdot\bm{Q}_r(s,\bars,\theta-\bartheta)
\end{aligned}
\end{equation}
where $Q_{0,1}$ and $\bm{Q}_r$ are as in section \ref{subsec:tools}. In addition, we have 
\begin{equation}\label{eq:psRt_2}
\begin{aligned}
\p_s\bR_{\rm t}\cdot\bR_{\rm t} &= \p_s\bR\cdot\bR + \epsilon\bars^2Q_{S,6}(s,\theta,\bars,\bartheta) + \bars^2Q_{S,7}(s,\theta,\bars) \\
Q_{S,7} &= -\big((1-\epsilon\wh\kappa(s,\theta))\be_{\rm t}(s)+\epsilon\kappa_3\be_\theta(s,\theta)\big)\cdot\bm{Q}_{\rm t}(s,\bars)
\end{aligned}
\end{equation}
where $Q_{S,6}$ is as in \eqref{eq:psRt_1} and $\bm{Q}_{\rm t}$ is as in section \ref{subsec:tools}. Each of $Q_{S,5}$, $Q_{S,6}$, $Q_{S,7}$ additionally satisfy the bound \eqref{eq:QSj_cbeta}.
Using \eqref{eq:psRt_1} and \eqref{eq:psRt_2}, we may write
\begin{align*}
\p_s\bigg(\frac{1}{\abs{\bR}}-\frac{1}{\abs{\bR_{\rm t}}}\bigg) &= \frac{\p_s\bR_{\rm t}\cdot\bR_{\rm t}}{\abs{\bR_{\rm t}}^3}-\frac{\p_s\bR\cdot\bR}{\abs{\bR}^3}\\
&= \p_s\bR_{\rm t}\cdot\bR_{\rm t}\bigg(\frac{1}{\abs{\bR_{\rm t}}^3}-\frac{1}{\abs{\bR}^3}\bigg)+ \frac{\p_s\bR_{\rm t}\cdot\bR_{\rm t}-\p_s\bR\cdot\bR}{\abs{\bR}^3}\\
&= \big(\bars + \epsilon\bars Q_{S,5}-\epsilon^2 \kappa_3\sin(\bartheta) + \epsilon\bars^2 Q_{S,6}\big)\bigg(\frac{1}{\abs{\bR_{\rm t}}^3}-\frac{1}{\abs{\bR}^3}\bigg)+ \frac{\epsilon\bars^2Q_{S,6} + \bars^2Q_{S,7}}{\abs{\bR}^3}\,,
\end{align*}
where we note that $\p_s\barR\cdot\barR=\bars$.

Using \eqref{eq:Rt_inv_diff} to expand $\frac{1}{\abs{\bR_{\rm t}}^3}-\frac{1}{\abs{\bR}^3}$ and noting the additional power of $\bars$ in the difference, we may apply Lemma \ref{lem:basic_est} to estimate 
\begin{align*}
\abs{\p_s\mc{R}_{\mc{S},1}} &= \abs{ \frac{1}{4\pi}\int_{-1/2}^{1/2}\int_{-\pi}^{\pi}\p_s\bigg(\frac{1}{\abs{\bR}}-\frac{1}{\abs{\bR_{\rm t}}} \bigg)\,\varphi(s-\bars,\theta-\bartheta) \,\epsilon\, d\bartheta d\bars } \\
&\le c(\kappa_*)\norm{\varphi}_{L^\infty}\int_{-1/2}^{1/2}\int_{-\pi}^{\pi} \frac{1}{\abs{\bR}}\,\epsilon\, d\bartheta d\bars
\le c(\kappa_*,c_\Gamma)\,\epsilon\norm{\varphi}_{L^\infty}\,.
\end{align*}
Furthermore, for any \textcolor{black}{$\gamma$} satisfying $\alpha<\textcolor{black}{\gamma}<1$, we may use case (1) of Lemma \ref{lem:alpha_est} to estimate
\begin{align*}
\abs{\p_s\mc{R}_{\mc{S},1}}_{\dot C^{0,\textcolor{black}{\gamma}}} &= \abs{ \frac{1}{4\pi}\int_{-1/2}^{1/2}\int_{-\pi}^{\pi}\p_s\bigg(\frac{1}{\abs{\bR}}-\frac{1}{\abs{\bR_{\rm t}}} \bigg)\,\varphi(s-\bars,\theta-\bartheta) \,\epsilon\, d\bartheta d\bars }_{\dot C^{0,\textcolor{black}{\gamma}}} \\
&\le c(\kappa_*,c_\Gamma)\,\epsilon^{1-\textcolor{black}{\gamma}}\bigg(\max_j\norm{Q_{S,j}}_{C^{0,\textcolor{black}{\gamma}}_1} \bigg)\norm{\varphi}_{L^\infty}
\le c(\kappa_{*,\textcolor{black}{\gamma}},c_\Gamma)\,\epsilon^{1-2\textcolor{black}{\gamma}}\norm{\varphi}_{L^\infty}\,.
\end{align*}

Similarly, for $\mc{R}_{\mc{S},2}$ we may calculate
\begin{align*}
\p_s\bigg(\frac{1}{\abs{\bR_{\rm t}}}-\frac{1}{|\barR|}\bigg) 
&= \bars\bigg(\frac{1}{|\barR|^3}-\frac{1}{\abs{\bR_{\rm t}}^3}\bigg)- \frac{\epsilon\bars Q_{S,5}-\epsilon^2 \kappa_3\sin(\bartheta) + \epsilon\bars^2 Q_{S,6}}{\abs{\bR_{\rm t}}^3}\,.
\end{align*}
Using \eqref{eq:Rt_inv_diff} and noting the additional $\epsilon$, we may use Lemmas \ref{lem:odd_nm} and \ref{lem:basic_est} to obtain 
\begin{align*}
\abs{\p_s\mc{R}_{\mc{S},2}} &= 
\abs{\frac{1}{4\pi}\int_{-1/2}^{1/2}\int_0^{2\pi}\p_s\bigg(\frac{1}{\abs{\bR_{\rm t}}}-\frac{1}{|\barR|} \bigg)\,\varphi(s-\bars,\theta-\bartheta) \,\epsilon\, d\bartheta d\bars} \\
&\le c(\kappa_*,c_\Gamma)\,\epsilon^{1+\alpha}\,\big(\max_{j}\norm{Q_{S,j}}_{C^{0,\alpha}_2}\big)\norm{\varphi}_{C^{0,\alpha}}
\le c(\kappa_{*,\alpha},c_\Gamma)\,\epsilon\norm{\varphi}_{C^{0,\alpha}}\,.
\end{align*}
Furthermore, using Lemma \ref{lem:alpha_est}, case (2), we may estimate
\begin{align*}
\abs{\p_s\mc{R}_{\mc{S},2}}_{\dot C^{0,\alpha}} &= 
\abs{\frac{1}{4\pi}\int_{-1/2}^{1/2}\int_{-\pi}^{\pi}\p_s\bigg(\frac{1}{\abs{\bR_{\rm t}}}-\frac{1}{|\barR|} \bigg)\,\varphi(s-\bars,\theta-\bartheta) \,\epsilon\, d\bartheta d\bars}_{\dot C^{0,\alpha}} \\
&\le c(\kappa_*,c_\Gamma)\,\epsilon\,\max_j\big(\norm{Q_{S,j}}_{C^{0,\textcolor{black}{\alpha^+}}_1}+\norm{Q_{S,j}}_{C^{0,\alpha}_2} \big)\norm{\varphi}_{C^{0,\alpha}}
\le c(\kappa_{*,\textcolor{black}{\alpha^+}},c_\Gamma)\,\epsilon^{1-\textcolor{black}{\alpha^+}} \norm{\varphi}_{C^{0,\alpha}}\,.
\end{align*}

Finally, using \eqref{eq:psRt_1} and \eqref{eq:psRt_2}, by Lemmas \ref{lem:odd_nm} and \ref{lem:basic_est} we may estimate 
\begin{align*}
\abs{\p_s\mc{R}_{\mc{S},3}} &= \abs{\frac{1}{4\pi}\int_{-1/2}^{1/2}\int_{-\pi}^{\pi}\frac{\p_s\bR\cdot\bR}{\abs{\bR}^3}\,\varphi(s-\bars,\theta-\bartheta) \,\epsilon^2\wh\kappa(s-\bars,\theta-\bartheta)\, d\bartheta d\bars}\\
&\le c(\kappa_*,c_\Gamma)\,\epsilon^{1+\alpha}\,\big(\max_{j}\norm{Q_{S,j}}_{C^{0,\alpha}_2}\big)\norm{\varphi\wh\kappa}_{C^{0,\alpha}}
\le c(\kappa_{*,\alpha},c_\Gamma)\,\epsilon\norm{\varphi}_{C^{0,\alpha}}\,.
\end{align*}
Similarly, by Lemma \ref{lem:alpha_est}, case (2), we have
\begin{align*}
\abs{\p_s\mc{R}_{\mc{S},3}}_{\dot C^{0,\alpha}} &= \abs{\frac{1}{4\pi}\int_{-1/2}^{1/2}\int_{-\pi}^{\pi}\frac{\p_s\bR\cdot\bR}{\abs{\bR}^3}\,\varphi(s-\bars,\theta-\bartheta) \,\epsilon^2\wh\kappa(s-\bars,\theta-\bartheta)\, d\bartheta d\bars}_{\dot C^{0,\alpha}}\\
&\le c(\kappa_*,c_\Gamma)\,\epsilon\,\max_j\big(\norm{Q_{S,j}}_{C^{0,\textcolor{black}{\alpha^+}}_1}+\norm{Q_{S,j}}_{C^{0,\alpha}_2} \big)\norm{\varphi\wh\kappa}_{C^{0,\alpha}}
\le c(\kappa_{*,\textcolor{black}{\alpha^+}},c_\Gamma)\,\epsilon^{1-\textcolor{black}{\alpha^+}} \norm{\varphi}_{C^{0,\alpha}}\,.
\end{align*}

Next we estimate $\frac{1}{\epsilon}\p_\theta\mc{R}_{\mc{S}}$. Note that $\frac{1}{\epsilon}\p_\theta\bR=\frac{1}{\epsilon}\p_\theta\bR_{\rm t}=\be_\theta(s,\theta)$. Using the identities of section \ref{subsec:tools}, we may write 
\begin{equation}\label{eq:pthetaRt}
\begin{aligned}
\textstyle\frac{1}{\epsilon}\p_\theta\bR_{\rm t}\cdot\bR_{\rm t} &= 
\textstyle 2\epsilon\sin(\frac{\bartheta}{2})\cos(\frac{\bartheta}{2}) +\epsilon \bars\kappa_3\cos(\bartheta)+\epsilon\bars^2 Q_{0,3}(s,\theta,\bars,\bartheta) \\
&= \textstyle\frac{1}{\epsilon}\p_\theta\bR\cdot\bR \\
&= \textstyle\frac{1}{\epsilon}\p_\theta\barR\cdot\barR + \epsilon\bars Q_{S,8}(s,\theta,\bars,\bartheta)\,,
\end{aligned}
\end{equation}
where $Q_{0,3}$ is as in \eqref{eq:Qexpand} and $Q_{S,8}$ satisfies the bound \eqref{eq:QSj_cbeta}.
We may then write
\begin{align*}
\textstyle \frac{1}{\epsilon}\p_\theta\displaystyle\bigg(\frac{1}{\abs{\bR}}-\frac{1}{\abs{\bR_{\rm t}}}\bigg) &= \frac{\frac{1}{\epsilon}\p_\theta\bR_{\rm t}\cdot\bR_{\rm t}}{\abs{\bR_{\rm t}}^3}-\frac{\frac{1}{\epsilon}\p_\theta\bR\cdot\bR}{\abs{\bR}^3}
= \textstyle\frac{1}{\epsilon}\p_\theta\bR\cdot\bR\displaystyle\bigg(\frac{1}{\abs{\bR_{\rm t}}^3}-\frac{1}{\abs{\bR}^3}\bigg)\\
&= \big(\textstyle 2\epsilon\sin(\frac{\bartheta}{2})\cos(\frac{\bartheta}{2}) +\epsilon \bars\kappa_3\cos(\bartheta)+\epsilon\bars^2 Q_{0,3}\big)\displaystyle\bigg(\frac{1}{\abs{\bR_{\rm t}}^3}-\frac{1}{\abs{\bR}^3}\bigg)\,.
\end{align*}
Again using \eqref{eq:Rt_inv_diff} to expand $\frac{1}{\abs{\bR_{\rm t}}^3}-\frac{1}{\abs{\bR}^3}$, we may use Lemma \ref{lem:basic_est} to obtain
\begin{align*}
\textstyle\abs{\frac{1}{\epsilon}\p_\theta\mc{R}_{\mc{S},1}} &= 
\abs{\frac{1}{4\pi}\int_{-1/2}^{1/2}\int_{-\pi}^{\pi}\textstyle\frac{1}{\epsilon}\p_\theta\displaystyle\bigg(\frac{1}{\abs{\bR}}-\frac{1}{\abs{\bR_{\rm t}}} \bigg)\,\varphi(s-\bars,\theta-\bartheta) \,\epsilon\, d\bartheta d\bars}\\
&\le c(\kappa_*)\norm{\varphi}_{L^\infty}\int_{-1/2}^{1/2}\int_{-\pi}^{\pi} \frac{1}{\abs{\bR}}\,\epsilon\, d\bartheta d\bars
\le c(\kappa_*,c_\Gamma)\,\epsilon\norm{\varphi}_{L^\infty}\,.
\end{align*}
Furthermore, again choosing $\textcolor{black}{\gamma}$ to satisfy $\alpha<\textcolor{black}{\gamma}\le\beta$, we may use case (1) of Lemma \ref{lem:alpha_est} to estimate
\begin{align*}
\textstyle\abs{\frac{1}{\epsilon}\p_\theta\mc{R}_{\mc{S},1}}_{\dot C^{0,\textcolor{black}{\gamma}}} &= 
\abs{\frac{1}{4\pi}\int_{-1/2}^{1/2}\int_{-\pi}^{\pi}\textstyle\frac{1}{\epsilon}\p_\theta\displaystyle\bigg(\frac{1}{\abs{\bR}}-\frac{1}{\abs{\bR_{\rm t}}} \bigg)\,\varphi(s-\bars,\theta-\bartheta) \,\epsilon\, d\bartheta d\bars}_{\dot C^{0,\textcolor{black}{\gamma}}}\\
&\le c(\kappa_*,c_\Gamma)\,\epsilon^{1-\textcolor{black}{\gamma}}\,\max_j\norm{Q_{S,j}}_{C^{0,\textcolor{black}{\gamma}}_1}\norm{\varphi}_{L^\infty}
\le c(\kappa_{*,\textcolor{black}{\gamma}},c_\Gamma)\,\epsilon^{1-2\textcolor{black}{\gamma}}\norm{\varphi}_{L^\infty}\,.
\end{align*}

Next, using \eqref{eq:pthetaRt}, we may write
\begin{align*}
\textstyle\frac{1}{\epsilon}\p_\theta\displaystyle\bigg(\frac{1}{\abs{\bR_{\rm t}}}-\frac{1}{|\barR|}\bigg) 
&= \textstyle\frac{1}{\epsilon}\p_\theta\barR\cdot\barR\displaystyle\bigg(\frac{1}{|\barR|^3}-\frac{1}{\abs{\bR_{\rm t}}^3}\bigg) + \frac{\frac{1}{\epsilon}\p_\theta\barR\cdot\barR-\frac{1}{\epsilon}\p_\theta\bR_{\rm t}\cdot\bR_{\rm t}}{\abs{\bR_{\rm t}}^3}\\
&= \textstyle 2\epsilon\sin(\frac{\bartheta}{2})\cos(\frac{\bartheta}{2})
\displaystyle\bigg(\frac{1}{|\barR|^3}-\frac{1}{\abs{\bR_{\rm t}}^3}\bigg) - \frac{\epsilon\bars Q_{S,8}}{\abs{\bR_{\rm t}}^3}\,.
\end{align*}
We may then use the expansion \eqref{eq:Rt_inv_diff} of $\frac{1}{|\barR|^3}-\frac{1}{\abs{\bR_{\rm t}}^3}$ along with Lemmas \ref{lem:odd_nm} and \ref{lem:basic_est} to estimate 
\begin{align*}
\textstyle\abs{\frac{1}{\epsilon}\p_\theta\mc{R}_{\mc{S},2}} &= 
\abs{\frac{1}{4\pi}\int_{-1/2}^{1/2}\int_{-\pi}^{\pi}\bigg(\frac{1}{\abs{\bR_{\rm t}}}-\frac{1}{|\barR|} \bigg)\,\varphi(s-\bars,\theta-\bartheta) \,\epsilon\, d\bartheta d\bars} \\
&\le c(\kappa_*,c_\Gamma)\,\epsilon^{1+\alpha}\,\big(\max_{j}\norm{Q_{S,j}}_{C^{0,\alpha}_2}\big)\norm{\varphi}_{C^{0,\alpha}}
\le c(\kappa_{*,\alpha},c_\Gamma)\,\epsilon\norm{\varphi}_{C^{0,\alpha}}\,.
\end{align*}
As was the case for $\p_s\mc{R}_{\mc{S},2}$, we may next use Lemma \ref{lem:alpha_est}, case (2) to obtain the $\dot C^{0,\alpha}$ estimate
\begin{align*}
\textstyle\abs{\frac{1}{\epsilon}\p_\theta\mc{R}_{\mc{S},2}}_{\dot C^{0,\alpha}} &= 
\abs{\frac{1}{4\pi}\int_{-1/2}^{1/2}\int_{-\pi}^{\pi}\bigg(\frac{1}{\abs{\bR_{\rm t}}}-\frac{1}{|\barR|} \bigg)\,\varphi(s-\bars,\theta-\bartheta) \,\epsilon\, d\bartheta d\bars} \\
&\le c(\kappa_*,c_\Gamma)\,\epsilon\,\max_j\big(\norm{Q_{S,j}}_{C^{0,\textcolor{black}{\alpha^+}}_1}+\norm{Q_{S,j}}_{C^{0,\alpha}_2} \big)\norm{\varphi}_{C^{0,\alpha}}
\le c(\kappa_{*,\textcolor{black}{\alpha^+}},c_\Gamma)\,\epsilon^{1-\textcolor{black}{\alpha^+}} \norm{\varphi}_{C^{0,\alpha}}\,.
\end{align*}
Finally, using \eqref{eq:pthetaRt}, we may estimate $\frac{1}{\epsilon}\p_\theta\mc{R}_{\mc{S},3}$ similarly to $\p_s\mc{R}_{\mc{S},3}$ to obtain
\begin{align*}
\textstyle\abs{\frac{1}{\epsilon}\p_\theta\mc{R}_{\mc{S},3}} &= \abs{\frac{1}{4\pi}\int_{-1/2}^{1/2}\int_{-\pi}^{\pi}\frac{\frac{1}{\epsilon}\p_\theta\bR\cdot\bR}{\abs{\bR}}\,\varphi(s-\bars,\theta-\bartheta) \,\epsilon^2\wh\kappa(s-\bars,\theta-\bartheta)\, d\bartheta d\bars} \\
&\le c(\kappa_*,c_\Gamma)\,\epsilon^{1+\alpha}\,\big(\max_{j}\norm{Q_{S,j}}_{C^{0,\alpha}_2}\big)\norm{\varphi\wh\kappa}_{C^{0,\alpha}}
\le c(\kappa_{*,\alpha},c_\Gamma)\,\epsilon\norm{\varphi}_{C^{0,\alpha}}\,.
\end{align*}
Again by case (2) of Lemma \ref{lem:alpha_est} we have 
\begin{align*}
\textstyle\abs{\frac{1}{\epsilon}\p_\theta\mc{R}_{\mc{S},3}}_{\dot C^{0,\alpha}} &= \abs{\frac{1}{4\pi}\int_{-1/2}^{1/2}\int_{-\pi}^{\pi}\frac{\frac{1}{\epsilon}\p_\theta\bR\cdot\bR}{\abs{\bR}}\,\varphi(s-\bars,\theta-\bartheta) \,\epsilon^2\wh\kappa(s-\bars,\theta-\bartheta)\, d\bartheta d\bars}_{\dot C^{0,\alpha}} \\
&\le c(\kappa_*,c_\Gamma)\,\epsilon\,\max_j\big(\norm{Q_{S,j}}_{C^{0,\textcolor{black}{\alpha^+}}_1}+\norm{Q_{S,j}}_{C^{0,\alpha}_2} \big)\norm{\varphi\wh\kappa}_{C^{0,\alpha}}
\le c(\kappa_{*,\textcolor{black}{\alpha^+}},c_\Gamma)\,\epsilon^{1-\textcolor{black}{\alpha^+}} \norm{\varphi}_{C^{0,\alpha}}\,.
\end{align*}
In total, we obtain Lemma \ref{lem:sing_layer_remainder}.
\end{proof}

\subsection{Mapping properties of double layer remainder terms}
We next show estimate \eqref{eq:RD_mapping} of Lemma \ref{lem:RS_and_RD} regarding the double layer remainder $\mc{R}_\mc{D}$ defined in \eqref{eq:RD_def}. In particular, we show that $\mc{R}_\mc{D}$ is smoother than $(\frac{1}{2}{\bf I}-\overline{{\mc D}})$ and obtain explicit $\epsilon$ dependence for the bound. We reiterate the result here for convenience. 
\begin{lemma}[Mapping properties of $\mc{R}_{\mc{D}}$, restated]\label{lem:doub_layer_remainder}
For $0<\gamma<\beta<1$, let $\X\in C^{2,\beta}$ be as in section \ref{subsec:geom} and consider $\mc{R}_\mc{D}$ as in \eqref{eq:RD_def}. Given $\psi\in C^{0,\gamma}(\Gamma_\epsilon)$, we have
\begin{equation}\label{eq:RD_lem}
\norm{\mc{R}_\mc{D}[\psi]}_{C^{1,\gamma}} \le c(\kappa_{*,\textcolor{black}{\gamma^+}},c_\Gamma)\,\epsilon^{-\textcolor{black}{\gamma^+}} \norm{\psi}_{C^{0,\gamma}}
\end{equation}
\textcolor{black}{for any $\gamma^+\in(\gamma, \beta]$.}
\end{lemma}

\begin{proof}
Due to the form \eqref{eq:barS_barD} of $\overline{\mc{D}}$ used in section \ref{sec:strt_periodic} and the Jacobian factor \eqref{eq:jacfac} appearing in \eqref{eq:RD_def}, we consider the double layer remainder $\mc{R}_\mc{D}$ in three parts: 
\begin{align*}
\mc{R}_{\mc{D}}[\psi](s,\theta) &= \mc{R}_{\mc{D},0}[\psi](s,\theta)+\mc{R}_{\mc{D},1}[\psi](s,\theta)+\mc{R}_{\mc{D},2}[\psi](s,\theta)\,,\\
\mc{R}_{\mc{D},0} &:= 
-\frac{1}{4\pi}\bigg(\int_{-\infty}^{-1/2}+\int_{1/2}^{\infty}\bigg)\int_{-\pi}^{\pi}\overline{K}_\mc{D}\,\psi(s-\bars,\theta-\bartheta) \,\epsilon\, d\bartheta d\bars \\
\mc{R}_{\mc{D},1} &:= 
\frac{1}{4\pi}\int_{-1/2}^{1/2}\int_{-\pi}^{\pi}(K_\mc{D}-\overline{K}_\mc{D})\,\psi(s-\bars,\theta-\bartheta) \,\epsilon\, d\bartheta d\bars \\
\mc{R}_{\mc{D},2} &:= -\frac{1}{4\pi}\int_{-1/2}^{1/2}\int_{-\pi}^{\pi}K_\mc{D}\,\psi(s-\bars,\theta-\bartheta) \,\epsilon^2\wh\kappa(s-\bars,\theta-\bartheta)\, d\bartheta d\bars \,.
\end{align*}

We begin with the smooth remainder $\mc{R}_{\mc{D},0}$ away from the singularity at $\bars=\bartheta=0$. The straight kernel $\overline K_{\mc{D}}$ may be parameterized as 
\begin{equation}\label{eq:barKD}
\overline K_{\mc{D}} = \frac{-2\epsilon \sin^2(\frac{\bartheta}{2})}{|\barR|^3}\,,
\end{equation}
and thus, using the form \eqref{eq:overlinebx} of $\barR$, we have
\begin{align*}
\norm{\mc{R}_{\mc{D},0}}_{C^2} &\le 
c\,\epsilon^2\norm{\psi}_{L^\infty}\int_{1/2}^{\infty}\bigg(\frac{1}{\abs{\bars}^3}+\frac{1}{\bars^4}+\frac{1}{\abs{\bars}^5}\bigg) d\bars \le c\,\epsilon^2\norm{\psi}_{L^\infty}\,.
\end{align*}

We next turn to $L^\infty$ bounds for the remainders $\mc{R}_{\mc{D},1}$ and $\mc{R}_{\mc{D},2}$. Using the expansion \eqref{eq:Rnxprime}, we have that the curved kernel $K_{\mc D}$ may be expressed as  
\begin{align*}
K_{\mc D}-\overline K_{\mc D} &= \frac{\bR\cdot\bm{n}_{x'}-\barR\cdot\overline{\bm{n}}_{x'}}{\abs{\bR}^3}+ \barR\cdot\overline{\bm{n}}_{x'}\bigg(\frac{1}{\abs{\bR}^3}-\frac{1}{|\barR|^3}\bigg)
\\
&= \frac{\bars^2Q_{\rm n'}}{\abs{\bR}^3} - \textstyle 2\epsilon\sin^2(\frac{\bartheta}{2})\displaystyle \bigg(\frac{1}{\abs{\bR}^3}-\frac{1}{|\barR|^3}\bigg) \,.
\end{align*}
We may then use equation \eqref{eq:Rdiffk} and Lemma \ref{lem:basic_est} to estimate 
\begin{align*}
\abs{\mc{R}_{\mc{D},1}} &\le c\norm{\psi}_{L^\infty}\int_{-1/2}^{1/2}\int_{-\pi}^{\pi}\abs{K_\mc{D}-\overline{K}_\mc{D}} \,\epsilon\, d\bartheta d\bars \\
&\le c\norm{\psi}_{L^\infty}\int_{-1/2}^{1/2}\int_{-\pi}^{\pi}\bigg(\frac{\bars^2}{\abs{\bR}^3}+\frac{\epsilon\bars^2 +\epsilon^2|\sin(\frac{\bartheta}{2})|\abs{\bars} + \bars^4+\epsilon\abs{\bars}^3 + \epsilon^2 \bars^2|\sin(\frac{\bartheta}{2})|}{\abs{\overline{\bm{R}}}\abs{\bm{R}}(\abs{\bm{R}}+\abs{\overline{\bm{R}}})}\sum_{\ell=0}^{2}\frac{\epsilon\sin^2(\frac{\bartheta}{2})}{\abs{\bR}^\ell|\barR|^{2-\ell}} \bigg) \,\epsilon\, d\bartheta d\bars\\
&\le c\norm{\psi}_{L^\infty}\int_{-1/2}^{1/2}\int_{-\pi}^{\pi}\frac{1}{\abs{\bR}} \,\epsilon\, d\bartheta d\bars 
\le c\,\epsilon \norm{\psi}_{L^\infty}\,.
\end{align*}
Similarly, using the form \eqref{eq:barKD} of $\overline K_{\mc{D}}$, we may estimate
\begin{align*}
\abs{\mc{R}_{\mc{D},2}} &\le c(\kappa_*)\norm{\psi}_{L^\infty}\int_{-1/2}^{1/2}\int_{-\pi}^{\pi}\abs{K_\mc{D}} \,\epsilon^2 d\bartheta d\bars \\
&\le c(\kappa_*)\norm{\psi}_{L^\infty}\int_{-1/2}^{1/2}\int_{-\pi}^{\pi}\bigg(\frac{1}{|\barR|}+\frac{\epsilon}{\abs{\bR}} \bigg) \,\epsilon d\bartheta d\bars 
\le c(\kappa_*,c_\Gamma)\,\epsilon\norm{\psi}_{L^\infty}\,.
\end{align*}

We next show $C^{0,\gamma}$ bounds for $\p_s\mc{R}_{\mc D}$. We begin by noting that the straight kernel $\overline K_{\mc D}$ satisfies
\begin{equation}\label{eq:ds_barKD}
\p_s \overline K_{\mc D} = -3\frac{\bars\,\barR\cdot\overline{\bm{n}}_{x'}}{|\barR|^5} = 6\frac{\epsilon\sin^2(\frac{\bartheta}{2})\bars}{|\barR|^5}\,.
\end{equation}
Using that 
\begin{align*}
\p_s\bm{R} = (1-\epsilon\wh\kappa(s,\theta))\be_{\rm t}(s)+\epsilon\kappa_3\be_\theta(s,\theta)
\end{align*}
along with the form \eqref{eq:curvedR} of $\bR$ and the remainder expressions \eqref{eq:Qexpand}, we have that 
\begin{align*}
\bm{R}\cdot\p_s\bm{R} &= \textstyle \bars  + \bars^3Q_{\rm s,1}(s,\bars) + \epsilon\bars Q_{\rm s,2}(s,\theta,\bars,\bartheta) +\epsilon^2\kappa_3\sin(\bartheta) \\ 
\p_s\bm{R}\cdot\bm{n}_{x'} &= \p_s\bm{R}\cdot\be_r(s-\bars,\theta-\bartheta)=\bars\wh\kappa(s,\theta)-\kappa_3\epsilon\sin(\bartheta)+\epsilon\bars Q_{\rm s,3}(s,\theta,\bars,\bartheta) \,.
\end{align*}
Here, using the notation \eqref{eq:Q_Cbeta_12}, the functions $Q_{{\rm s},j}$ satisfy 
\begin{equation}\label{eq:Qsj_ests}
\begin{aligned}
 \norm{Q_{{\rm s},1}}_{C^{0,\beta}_\mu}&\le c(\kappa_{*,\beta})\,,\quad \mu=1,2 \\
 \norm{Q_{{\rm s},j}}_{C^{0,\beta}_\mu}&\le \epsilon^{-\beta}c(\kappa_{*,\beta})\,, \quad \mu=1,2\,;\; j=2,3\,.
 \end{aligned} 
 \end{equation}
Altogether, we have that the curved kernel $K_{\mc D}$ may be written 
\begin{equation}\label{eq:ps_KD}
\begin{aligned}
\p_s K_{\mc D} &= -3\frac{(\bm{R}\cdot\bm{n}_{x'})(\bm{R}\cdot\p_s\bm{R})}{\abs{\bm{R}}^5} + \frac{(\p_s\bm{R})\cdot\bm{n}_{x'}}{\abs{\bm{R}}^3}\\
&=\p_s\overline K_{\mc D} + 6\epsilon\textstyle\sin^2(\frac{\bartheta}{2})\bars \displaystyle \bigg(\frac{1}{\abs{\bR}^5} -\frac{1}{|\barR|^5}\bigg) + \frac{\bars\,\wh\kappa-\kappa_3\epsilon\sin(\bartheta)+\epsilon\bars Q_{\rm s,3}}{\abs{\bR}^3}+\frac{6\epsilon\sin^2(\frac{\bartheta}{2})\bars^3Q_{\rm s,1} }{\abs{\bR}^5} \\
 &\quad +\frac{6\epsilon^2\sin^2(\frac{\bartheta}{2})\bars Q_{\rm s,2} + \bars^3 Q_{\rm s,4} +\epsilon^2\sin(\frac{\bartheta}{2})\bars^2 Q_{\rm s,5} +12\kappa_3\epsilon^3\sin^3(\frac{\bartheta}{2})\cos(\frac{\bartheta}{2}) }{\abs{\bR}^5}
\end{aligned}
\end{equation}
where $Q_{\rm s,4} = -3Q_{\rm n'}(1+\bars^2 Q_{\rm s,1}+\epsilon Q_{\rm s,2})$ and $Q_{\rm s,5} = -6Q_{\rm n'}\kappa_3\cos(\frac{\bartheta}{2})$ for $Q_{\rm n'}(s,\bars,\theta-\bartheta)$ as in \eqref{eq:Rnxprime}.

Using equation \eqref{eq:Rdiffk} to rewrite $\frac{1}{\abs{\bR}^5}-\frac{1}{|\barR|^5}$, we may apply Lemmas \ref{lem:odd_nm} and \ref{lem:basic_est} along with the estimates \eqref{eq:Qsj_ests} to obtain
\begin{align*}
\abs{\p_s\mc{R}_{\mc{D},1}} &=
\abs{\frac{1}{4\pi}{\rm p.v.}\int_{-1/2}^{1/2}\int_{-\pi}^{\pi}(\p_sK_\mc{D}-\p_s\overline{K}_\mc{D})\psi(s-\bars,\theta-\bartheta) \,\epsilon\, d\bartheta d\bars } \\
&\le c(\kappa_{*},c_\Gamma)\,\epsilon^\gamma \big(\max_{j}\norm{Q_{{\rm s},j}}_{C^{0,\gamma}_2} \big) \norm{\psi}_{C^{0,\gamma}} 
\le c(\kappa_{*,\gamma},c_\Gamma) \norm{\psi}_{C^{0,\gamma}}\,.
\end{align*}
We may similarly apply Lemma \ref{lem:alpha_est} to obtain a $\dot C^{0,\gamma}$ estimate for $\p_s\mc{R}_{\mc{D},1}$:
\begin{align*}
\abs{\p_s\mc{R}_{\mc{D},1}}_{\dot C^{0,\gamma}} &=
\abs{\frac{1}{4\pi}{\rm p.v.}\int_{-1/2}^{1/2}\int_{-\pi}^{\pi}(\p_sK_\mc{D}-\p_s\overline{K}_\mc{D})\psi(s-\bars,\theta-\bartheta) \,\epsilon\, d\bartheta d\bars }_{\dot C^{0,\gamma}}\\
&\le c(\kappa_*,c_\Gamma)\,\max_j\big(\norm{Q_{{\rm s},j}}_{C^{0,\textcolor{black}{\gamma^+}}_1}+\norm{Q_{{\rm s},j}}_{C^{0,\gamma}_2} \big)\norm{\psi}_{C^{0,\gamma}} 
\le c(\kappa_{*,\textcolor{black}{\gamma^+}},c_\Gamma)\,\epsilon^{-\textcolor{black}{\gamma^+}} \norm{\psi}_{C^{0,\gamma}}\,.
\end{align*}
For $\p_s\mc{R}_{\mc{D},2}$, we use the form \eqref{eq:ds_barKD} of $\p_s\overline K_\mc{D}$ along with Lemmas \ref{lem:odd_nm} and \ref{lem:basic_est} to estimate
\begin{align*}
\abs{\p_s\mc{R}_{\mc{D},2}} &= \abs{\frac{1}{4\pi}{\rm p.v.}\int_{-1/2}^{1/2}\int_{-\pi}^{\pi}\p_s K_\mc{D}\,\psi(s-\bars,\theta-\bartheta) \,\epsilon^2\wh\kappa\, d\bartheta d\bars} \\
&\le c(\kappa_{*},c_\Gamma)\,\epsilon^\gamma \big(\max_{j}\norm{Q_{{\rm s},j}}_{C^{0,\gamma}_2} \big) \norm{\psi}_{C^{0,\gamma}} 
\le c(\kappa_{*,\gamma},c_\Gamma) \norm{\psi}_{C^{0,\gamma}}\,.
\end{align*}
Similarly, using Lemma \ref{lem:alpha_est}, we have 
\begin{align*}
\abs{\p_s\mc{R}_{\mc{D},2}}_{\dot C^{0,\gamma}} &=
\abs{\frac{1}{4\pi}{\rm p.v.}\int_{-1/2}^{1/2}\int_{-\pi}^{\pi}\p_sK_\mc{D}\,\psi(s-\bars,\theta-\bartheta) \,\epsilon^2\wh\kappa\, d\bartheta d\bars }_{\dot C^{0,\gamma}} \\
 &\le c(\kappa_*,c_\Gamma)\,\max_j\big(\norm{Q_{{\rm s},j}}_{C^{0,\beta}_1}+\norm{Q_{{\rm s},j}}_{C^{0,\gamma}_2} \big)\norm{\psi}_{C^{0,\gamma}} 
\le c(\kappa_{*,\beta},c_\Gamma)\,\epsilon^{-\beta} \norm{\psi}_{C^{0,\gamma}}\,.
\end{align*}

Finally, we prove $C^{0,\gamma}$ bounds for $\frac{1}{\epsilon}\p_\theta\mc{R}_{\mc D}$. We again begin by calculating the straight kernel $\frac{1}{\epsilon}\p_\theta\overline K_\mc{D}$, given by 
\begin{equation}\label{eq:dtheta_barKD}
 \textstyle \frac{1}{\epsilon}\p_\theta \overline K_\mc{D} = \displaystyle-\frac{3}{\epsilon}\frac{(\barR\cdot\overline{\bm{n}}_{x'})(\barR\cdot\p_\theta\barR)}{|\barR|^5} + \frac{\frac{1}{\epsilon}(\p_\theta\barR)\cdot\overline{\bm{n}}_{x'}}{|\barR|^3} =\frac{12\epsilon^2\sin^3(\frac{\bartheta}{2})\cos(\frac{\bartheta}{2})}{|\barR|^5} - \frac{2\sin(\frac{\bartheta}{2})\cos(\frac{\bartheta}{2})}{|\barR|^3}\,.
\end{equation}
Using that 
\begin{align*}
\textstyle \frac{1}{\epsilon}\p_\theta\bm{R} = \be_\theta(s,\theta) 
\end{align*}
along with the form \eqref{eq:curvedR} of $\bR$, we may write 
\begin{align*}
\textstyle \frac{1}{\epsilon}\bm{R}\cdot\p_\theta\bm{R} &= \textstyle 2\epsilon\sin(\frac{\bartheta}{2})\cos(\frac{\bartheta}{2}) +\bars^2Q_{\theta,1} + \epsilon\bars Q_{\theta,2} \,, \quad 
\textstyle\frac{1}{\epsilon}\p_\theta\bm{R}\cdot\bm{n}_{x'} = \textstyle -2\sin(\frac{\bartheta}{2})\cos(\frac{\bartheta}{2}) + \bars Q_{\theta,3}\,,
\end{align*}
where again, using the notation \eqref{eq:Q_Cbeta_12}, each $Q_{\theta,j}(s,\theta,\bars,\bartheta)$ satisfies 
\begin{equation}\label{eq:Qthetaj}
\norm{Q_{\theta,j}}_{C^{0,\beta}_\mu}\le \epsilon^{-\beta}c(\kappa_{*,\beta})\,, \quad \mu=1,2\,.
\end{equation}
The curved kernel $\frac{1}{\epsilon}\p_\theta K_{\mc{D}}$ may then be written
\begin{equation}\label{eq:ptheta_KD}
\begin{aligned}
\textstyle \frac{1}{\epsilon}\p_\theta K_{\mc{D}} &= -3\frac{\frac{1}{\epsilon}(\bR\cdot\bm{n}_{x'})(\bR\cdot\p_\theta\bR)}{\abs{\bR}^5} + \frac{\frac{1}{\epsilon}(\p_\theta\bR)\cdot\bm{n}_{x'}}{\abs{\bR}^3}\\
 &= \textstyle \frac{1}{\epsilon}\p_\theta\overline K_{\mc D} +\textstyle 12\epsilon^2\sin^3(\frac{\bartheta}{2})\cos(\frac{\bartheta}{2})\displaystyle \bigg(\frac{1}{\abs{\bR}^5}- \frac{1}{|\barR|^5}\bigg) 
 -\textstyle 2\sin(\frac{\bartheta}{2})\cos(\frac{\bartheta}{2})\displaystyle\bigg( \frac{1}{\abs{\bR}^3}-\frac{1}{|\barR|^3}\bigg) \\
 &\qquad  + \frac{\bars Q_{\theta,3}}{\abs{\bR}^3} -3\frac{ 2\epsilon^2\sin^2(\frac{\bartheta}{2})\bars Q_{\theta,2} + \epsilon\sin(\frac{\bartheta}{2})\bars^2Q_{\theta,4} +\bars^3Q_{\theta,5} }{\abs{\bR}^5} 
\end{aligned}
\end{equation}
where $Q_{\theta,4}= 2(\sin(\frac{\bartheta}{2})Q_{\theta,1} +\cos(\frac{\bartheta}{2})Q_{\rm n'})$ and $Q_{\theta,5}= Q_{\rm n'}(\epsilon Q_{\theta,2}+\bars Q_{\theta,1})$ for $Q_{\rm n'}(s,\bars,\theta-\bartheta)$ as in \eqref{eq:Rnxprime}.

Again using equation \eqref{eq:Rdiffk} to rewrite $\frac{1}{\abs{\bR}^5}-\frac{1}{|\barR|^5}$ and $\frac{1}{\abs{\bR}^3}-\frac{1}{|\barR|^3}$, by Lemmas \ref{lem:odd_nm} and \ref{lem:basic_est} we may estimate 
\begin{align*}
\textstyle \abs{\frac{1}{\epsilon}\p_\theta\mc{R}_{\mc{D},1}} &= \abs{\frac{1}{4\pi}{\rm p.v.}\int_{-1/2}^{1/2}\int_{-\pi}^{\pi}\frac{1}{\epsilon}(\p_\theta K_\mc{D}-\p_\theta\overline{K}_\mc{D})\,\psi(s-\bars,\theta-\bartheta) \,\epsilon\, d\bartheta d\bars } \\
&\le c(\kappa_{*},c_\Gamma)\,\epsilon^\gamma \big(\max_{j}\norm{Q_{\theta,j}}_{C^{0,\gamma}_2} \big) \norm{\psi}_{C^{0,\gamma}}
\le c(\kappa_{*,\gamma},c_\Gamma) \norm{\psi}_{C^{0,\gamma}} \,.
\end{align*}
Furthermore, using Lemma \ref{lem:alpha_est} and the form \eqref{eq:ptheta_KD} of $\frac{1}{\epsilon}(\p_\theta K_{\mc{D}}-\p_\theta\overline{K}_\mc{D})$, we may obtain the $\dot C^{0,\gamma}$ estimate 
\begin{align*}
\textstyle \abs{\frac{1}{\epsilon}\p_\theta\mc{R}_{\mc{D},1}}_{\dot C^{0,\gamma}} &= \abs{\frac{1}{4\pi}{\rm p.v.}\int_{-1/2}^{1/2}\int_{-\pi}^{\pi}\frac{1}{\epsilon}(\p_\theta K_\mc{D}-\p_\theta\overline{K}_\mc{D})\,\psi(s-\bars,\theta-\bartheta) \,\epsilon\, d\bartheta d\bars }_{\dot C^{0,\gamma}}\\
&\le c(\kappa_*,c_\Gamma)\,\max_j\big(\norm{Q_{\theta,j}}_{C^{0,\textcolor{black}{\gamma^+}}_1}+\norm{Q_{\theta,j}}_{C^{0,\gamma}_2} \big)\norm{\psi}_{C^{0,\gamma}} 
\le c(\kappa_{*,\textcolor{black}{\gamma^+}},c_\Gamma)\,\epsilon^{-\textcolor{black}{\gamma^+}} \norm{\psi}_{C^{0,\gamma}}\,.
\end{align*}
Next, for $\frac{1}{\epsilon}\p_\theta\mc{R}_{\mc{D},2}$, using \eqref{eq:ptheta_KD} and the form \eqref{eq:dtheta_barKD} of $\p_\theta\overline{K}_\mc{D}$, by Lemmas \ref{lem:odd_nm} and \ref{lem:basic_est} we obtain 
\begin{align*}
\textstyle \abs{\frac{1}{\epsilon}\p_\theta\mc{R}_{\mc{D},2}} &= \abs{\frac{1}{4\pi}{\rm p.v.}\int_{-1/2}^{1/2}\int_{-\pi}^{\pi}\frac{1}{\epsilon}\p_\theta K_\mc{D}\,\psi(s-\bars,\theta-\bartheta) \,\epsilon^2\wh\kappa\, d\bartheta d\bars }\\
&\le c(\kappa_{*},c_\Gamma)\,\epsilon^\gamma \big(\max_{j}\norm{Q_{\theta,j}}_{C^{0,\gamma}_2} \big) \norm{\psi}_{C^{0,\gamma}}
\le c(\kappa_{*,\gamma},c_\Gamma) \norm{\psi}_{C^{0,\gamma}} \,.
\end{align*}
Finally, using Lemma \ref{lem:alpha_est}, we may show 
\begin{align*}
\textstyle \abs{\frac{1}{\epsilon}\p_\theta\mc{R}_{\mc{D},2}}_{\dot C^{0,\gamma}} &= \abs{\frac{1}{4\pi}{\rm p.v.}\int_{-1/2}^{1/2}\int_{-\pi}^{\pi}\frac{1}{\epsilon}\p_\theta K_\mc{D}\,\psi(s-\bars,\theta-\bartheta) \,\epsilon^2\wh\kappa\, d\bartheta d\bars }_{\dot C^{0,\gamma}}\\
&\le c(\kappa_*,c_\Gamma)\,\max_j\big(\norm{Q_{\theta,j}}_{C^{0,\textcolor{black}{\gamma^+}}_1}+\norm{Q_{\theta,j}}_{C^{0,\gamma}_2} \big)\norm{\psi}_{C^{0,\gamma}} 
\le c(\kappa_{*,\textcolor{black}{\gamma^+}},c_\Gamma)\,\epsilon^{-\textcolor{black}{\gamma^+}} \norm{\psi}_{C^{0,\gamma}}\,.
\end{align*}
Combining each of the estimates for $\mc{R}_{\mc{D},0}$, $\mc{R}_{\mc{D},1}$, and $\mc{R}_{\mc{D},2}$, we obtain Lemma \ref{lem:doub_layer_remainder}.
\end{proof}

\subsection{Proof of Lemma \ref{lem:mean_in_s}}\label{subsec:mean_s_pf}
\textcolor{black}{
Here we show the decomposition of the $\theta$-averaged single layer operator presented in Lemma \ref{lem:mean_in_s}; in particular, for constant-in-$s$ function $h(\theta)$, we show that
\begin{align*}
\overline{\mc{S}}^{-1}\int_0^{2\pi}\mc{S}[h(\theta)]\,\epsilon\,d\theta = \mc{H}_\epsilon[h(\theta)]+\mc{H}_+[h(\theta)]
\end{align*}
for $\mc{H}_\epsilon$ which is small in $\epsilon$ and $\mc{H}_+$ which is smoother. We begin by letting
\begin{align*}
\int_0^{2\pi}\mc{S}[h(\theta)]\,\epsilon\,d\theta &= H_1(s)+H_2(s)\,,\\
H_1(s) &= \int_0^{2\pi}\int_\T\int_0^{2\pi}\mc{G}\,h(\theta')\,\epsilon^2\,d\theta'ds'd\theta \\
H_2(s) &= \int_0^{2\pi}\int_\T\int_0^{2\pi}\mc{G}\,h(\theta')\,\epsilon^3\wh\kappa\,d\theta'ds'd\theta 
\end{align*}
where $\mc{G}=\frac{1}{4\pi}\frac{1}{\abs{\bR}}$. We may use Lemmas \ref{lem:basic_est} and \ref{lem:alpha_est} to obtain
\begin{align*}
\norm{H_1}_{C^{0,\alpha}}\le c(\kappa_*,c_\Gamma)\,\epsilon^{2-\alpha}\norm{h}_{L^\infty}\,, \qquad
\norm{H_2}_{C^{0,\alpha}}\le c(\kappa_*,c_\Gamma)\,\epsilon^{3-\alpha}\norm{h}_{L^\infty}\,.
\end{align*}
}

\textcolor{black}{
We next consider $\p_sH_1$ and $\p_sH_2$. We note that for a curved filament, $\mc{G}$ is not a convolution kernel with respect to $s$, but it nearly is in the following sense. Using the form \eqref{eq:curvedR} of $\bR$ (and noting that we are using $s'$ rather than $\bars=s-s'$), we may show
\begin{align*}
\p_s\mc{G} = -\p_{s'}\mc{G} +\frac{1}{4\pi}\frac{\p_s\bR\cdot\bR-\p_{s'}\bR\cdot\bR}{\abs{\bR}^3} = -\p_{s'}\mc{G} +\frac{\epsilon(s-s')Q_{H1}+\epsilon^2\sin(\frac{\theta-\theta'}{2})Q_{H2}}{\abs{\bR}^3}
\end{align*}
where $Q_{H1}$ and $Q_{H2}$ both satisfy bounds of the form \eqref{eq:QSj_cbeta}. Using that $h(\theta')$ is independent of $s'$, we then have
\begin{align*}
\p_sH_1 &= \int_0^{2\pi}\int_\T\int_0^{2\pi}\p_s\mc{G}(s,s',\theta,\theta')\,h(\theta')\,\epsilon^2\,d\theta'ds'd\theta \\
&= \int_0^{2\pi}\int_\T\int_0^{2\pi}\frac{\epsilon(s-s')Q_{H1}+\epsilon^2\sin(\frac{\theta-\theta'}{2})Q_{H2}}{\abs{\bR}^3}\,h(\theta')\,\epsilon^2\,d\theta'ds'd\theta \,.
\end{align*}
By Lemma \ref{lem:alpha_est}, we then have
\begin{align*}
\abs{\p_sH_1}_{\dot C^{0,\alpha}} &\le c(\kappa_{*,\alpha^+},c_\Gamma)\,\epsilon^{2-\alpha^+} \norm{h}_{C^{0,\alpha}}\,.
\end{align*}
Similarly, we have 
\begin{align*}
\abs{\p_sH_1}_{\dot C^{0,\alpha}} &= \abs{\int_0^{2\pi}\int_\T\int_0^{2\pi}\p_s\mc{G}\,h(\theta')\,\epsilon^3\wh\kappa\,d\theta'ds'd\theta}_{\dot C^{0,\alpha}}
\le c(\kappa_{*,\alpha^+},c_\Gamma)\,\epsilon^{2-\alpha^+} \norm{h}_{C^{0,\alpha}}\,.
\end{align*}
}

\textcolor{black}{
Recall the notation $P_j$, $P_{<j}$, $P_{\ge j}$ associated with the Littlewood-Paley projection \eqref{eq:littlewood_p} and let $j_\epsilon = \frac{\abs{\log(2\pi\epsilon)}}{\log(2)}$ as in the proof of Lemma \ref{lem:S_mapping}. 
We define
\begin{align*}
\mc{H}_\epsilon = \overline{\mc{S}}^{-1}(P_{\ge j_\epsilon}H_1 + H_2)\,, \qquad 
\mc{H}_+ = \overline{\mc{S}}^{-1}(P_{<j_\epsilon}H_1)\,,
\end{align*}
and will use that the high frequencies here are small in $\epsilon$ and the low frequencies are smoother. 
Using Lemma \ref{lem:S_mapping}, we have
\begin{align*}
\norm{\mc{H}_\epsilon}_{C^{0,\alpha}}\le c\abs{H_1}_{\dot C^{1,\alpha}}+c\,\epsilon^{-1}\norm{H_2}_{C^{0,\alpha}} + c\,\abs{H_2}_{\dot C^{1,\alpha}}
\le c(\kappa_{*,\alpha^+},c_\Gamma)\,\epsilon^{2-\alpha^+} \norm{h}_{C^{0,\alpha}}\,.
\end{align*}
Furthermore, again using Lemma \ref{lem:S_mapping}, for $\alpha<\gamma\le \beta$ we have 
\begin{align*}
\norm{\mc{H}_+}_{C^{0,\gamma}} &\le c\,\epsilon^{-1}\norm{P_{<j_\epsilon}H_1}_{C^{0,\gamma}} + c\abs{P_{<j_\epsilon}H_1}_{\dot C^{1,\gamma}} 
\le c\,\epsilon^{-1}\norm{P_{<j_\epsilon}H_1}_{C^{0,\gamma}} + c\,2^{j_\epsilon}\abs{P_{<j_\epsilon}H_1}_{\dot C^{0,\gamma}} \\
&\le c\,\epsilon^{-1}\norm{H_1}_{C^{0,\gamma}} \le c(\kappa_{*,\gamma},c_\Gamma)\,\epsilon^{1-\gamma} \norm{h}_{C^{0,\alpha}}\,.
\end{align*}
\hfill \qedsymbol
}


\section{Bounds for full Neumann data}\label{sec:w}
Here we prove Lemma \ref{lem:w} bounding the full Neumann data $w(s,\theta)$ appearing in \eqref{eq:SB_DtN_strt_pert} in terms of the $\theta$-independent Dirichlet data $v(s)$. Our approach will rely on a different layer potential representation of the full Dirichlet-to-Neumann map $v=u\big|_{\Gamma_\epsilon} \mapsto w=\frac{\p u}{\p\bm{n}_x}\big|_{\Gamma_\epsilon}$ for $u$ harmonic in $\Omega_\epsilon$. 

Following \cite[Chapter 6]{kress1989linear}, given $\varphi(\bx')$, $\bx'=\X(s')+\epsilon\be_r(s',\theta')\in\Gamma_\epsilon$, we begin by defining the modified double layer potential
\begin{align*}
\mc{D}'[\varphi](\bx) &= \int_{\Gamma_\epsilon}\bigg(K_\mc{D}(\bx,\bx')\varphi(\bx') + \frac{\varphi(\bx')}{\abs{\bx-\X(s')}} \bigg)\, dS_{x'}\,, \qquad  \bx\in\Omega_\epsilon\,,
\end{align*}
where $K_{\mc{D}}$ is as in \eqref{eq:KD}, $\X(s')$ is the nearest point on the filament centerline to $\bx'$, and $dS_{x'}$ denotes the surface element with respect to $\bx'$ on $\Gamma_\epsilon$. The form of the modified kernel is based on the method used in \cite{keaveny2011applying} for a double layer formulation of the Stokes flow about a slender body and is designed to eliminate the null space of the operator $\frac{1}{2}{\bf I}+\mc{D}$ along $\Gamma_\epsilon$, which consists of constants.
%
For $\bx\in \Gamma_\epsilon$, we will use the same notation to denote 
\begin{equation}\label{eq:mod_double_layer}
\mc{D}'[\varphi](\bx) = \int_{\Gamma_\epsilon}K_\mc{D}(\bx,\bx')\varphi(\bx')\, dS_{x'} + \int_{\Gamma_\epsilon} \frac{\varphi(\bx')}{\abs{\bx-\X(s')}} \, dS_{x'}\,.
\end{equation}
Note that for $\bx\in\Gamma_\epsilon$, $\abs{\bx-\X(s')}\ge\epsilon$. 
Given Dirichlet data $v(\bx)$, $\bx\in \Gamma_\epsilon$, we have that
\begin{equation}\label{eq:uDprime}
u(\bx) = \mc{D}'[\varphi](\bx) \qquad \text{for }\bx\in \Omega_\epsilon
\end{equation}
is solution of exterior Dirichlet problem 
\begin{align*}
\Delta u=0 \quad \text{in }\Omega_\epsilon\,, \qquad u\big|_{\Gamma_\epsilon} = v(\bx)
\end{align*}
 provided that the density $\varphi$ on $\Gamma_\epsilon$ satisfies the integral equation
\begin{equation}\label{eq:sol_Dirichlet}
v(\bx) = \textstyle (\frac{1}{2}{\bf I}+\mc{D}')[\varphi](\bx)\,,\qquad \bx\in\Gamma_\epsilon \,,
\end{equation}
where the $\frac{1}{2}{\bf I}$ comes from the exterior jump relation \eqref{eq:ext_jump}. This may be seen by noting that if $u(\bx)$ is given by \eqref{eq:uDprime} for $\varphi$ satisfying $(\frac{1}{2}{\bf I}+\mc{D}')[\varphi]=0$ on $\Gamma_\epsilon$, then $u\big|_{\Gamma_\epsilon}=0$ and $u(\bx)=\mc{D}'[\varphi](\bx)\sim\frac{1}{\abs{\bx}}$ as $\abs{\bx}\to\infty$, so by uniqueness for the exterior Dirichlet problem, $u=0$ in $\Omega_\epsilon$. Furthermore, by continuity of $\frac{\p u}{\p\bm{n}_x}$ across $\Gamma_\epsilon$ (see \cite[Theorem 6.20]{kress1989linear}), we have that \emph{within} the slender body $\Gamma_\epsilon$, $u$ is harmonic with zero Neumann data. In particular, $u\equiv{\rm constant}$ in $\Sigma_\epsilon$, and therefore $\varphi\equiv{\rm constant}$ on $\Gamma_\epsilon$. However, since $\varphi$ is constant, using the jump relation $\int_{\Gamma_\epsilon}K_\mc{D}(\bx,\bx')\,dS_{x'}=-\frac{1}{2}$ for $\bx\in \Gamma_\epsilon$ (see \cite[Example 6.17]{kress1989linear}), we have $0=(\frac{1}{2}{\bf I}+\mc{D}')[\varphi]=\int_{\Gamma_\epsilon} \frac{\varphi}{\abs{\bx-\X(s')}}\,dS_{x'}$, and thus $\varphi=0$ due to the correction term.

Along $\Gamma_\epsilon$, we denote the normal derivative of the double layer potential by 
\begin{align*}
\mc{T}[\varphi](\bx) = -\frac{\p}{\p \bm{n}_x}\mc{D}'[\varphi](\bx)\,, \quad \bx\in \Gamma_\epsilon\,.
\end{align*}
Here again the minus sign arises because $\bm{n}_x$ points out from the slender body $\Sigma_\epsilon$. The operator $\mc{T}$ is hypersingular and must be understood as a finite part integral \cite{steinbach2007numerical}: 
\begin{equation}\label{eq:Tdef}
\mc{T}[\varphi](\bx) = {\rm p.v.}\int_{\Gamma_\epsilon}K_\mc{T}(\bx,\bx')\big(\varphi(\bx')-\varphi(\bx)\big)\, dS_{x'} + \int_{\Gamma_\epsilon}\frac{(\bx-\X(s'))\cdot\bm{n}_x}{\abs{\bx-\X(s')}^3}\varphi(\bx')\,dS_{x'}\,,
\end{equation}
where the kernel $K_\mc{T}$ is given by
\begin{equation}\label{eq:KT}
K_\mc{T}(\bx,\bx')= -\frac{\p}{\p\bm{n}_x}\frac{\p\mc{G}}{\p \bm{n}_{x'}} = -\frac{1}{4\pi}\bigg(\frac{\bm{n}_x\cdot\bm{n}_{x'}}{\abs{\bx-\bx'}^3} -3\frac{(\bx-\bx')\cdot\bm{n}_x \,(\bx-\bx')\cdot\bm{n}_{x'}}{\abs{\bx-\bx'}^5}\bigg)\,.
\end{equation}
The representation \eqref{eq:Tdef} may be obtained by first noting that within the exterior domain $\Omega_\epsilon$, the double layer kernel $K_\mc{D}$ integrates to zero (see \cite[Example 6.17]{kress1989linear}):
\begin{align*}
\int_{\Gamma_\epsilon}K_\mc{D}(\by,\bx')\,dS_{x'} =0\,, \quad \by\in \Omega_\epsilon\,.
\end{align*}
In particular, inserting a constant-in-$\by$ density $\varphi(\bx)$, $\bx\in \Gamma_\epsilon$, we have
\begin{align*}
\nabla_y\int_{\Gamma_\epsilon}K_\mc{D}(\by,\bx')\varphi(\bx)\,dS_{x'} =0\,, \quad \by\in \Omega_\epsilon\,,\;\; \bx\in \Gamma_\epsilon\,.
\end{align*}
We may then consider the first term of $\mc{T}$ as 
\begin{align*}
 -\lim_{\by\to\bx}\bm{n}_x\cdot\nabla_y\int_{\Gamma_\epsilon}K_\mc{D}(\by,\bx')\big(\varphi(\bx')-\varphi(\bx)\big)\,dS_{x'}\,, \quad \by\in \Omega_\epsilon\,,\;\; \bx\in \Gamma_\epsilon\,.
\end{align*}

Given $v(\bx)$, $\bx\in\Gamma_\epsilon$, we thus obtain a representation for the full Dirichlet-to-Neumann (DtN) map as
\begin{equation}\label{eq:usual_DtN}
w(\bx)= \textstyle \mc{T}(\frac{1}{2}{\bf I}+\mc{D}')^{-1}[v](\bx)\,, \qquad \bx\in \Gamma_\epsilon\,,
\end{equation}
i.e. as the normal derivative of the solution \eqref{eq:sol_Dirichlet} to the exterior Dirichlet problem with data $v(\bx)$.

Using the representation \eqref{eq:usual_DtN}, we proceed to prove Lemma \ref{lem:w} in three steps, summarized in the following three lemmas.
\begin{lemma}[Modified double layer smoothing]\label{lem:Dsmooths}
Let $0<\gamma<\beta<1$. Given a fiber centerline $\X(s)\in C^{2,\beta}$ as in section \ref{subsec:geom} and a surface density $\varphi\in C^{0,\gamma}(\Gamma_\epsilon)$, the modified double layer operator $\mc{D}'$ given by \eqref{eq:mod_double_layer} satisfies  
\begin{equation}
\norm{\mc{D}'[\varphi]}_{C^{1,\gamma}} \le 
c(\kappa_{*,\textcolor{black}{\gamma^+}},c_\Gamma)\,\epsilon^{-1-\gamma} \norm{\varphi}_{C^{0,\gamma}}
\end{equation}
\textcolor{black}{for any $\gamma^+\in(\gamma, \beta]$.}
\end{lemma}
The proof of Lemma \ref{lem:Dsmooths} appears in section \ref{subsec:mod_doub}.
In addition, we show the following uniform bound for $(\frac{1}{2}{\bf I}+\mc{D}')^{-1}$ with respect to the filament centerline shape:
\begin{lemma}[Uniform bounds in $\kappa$]\label{lem:DL_inverse}
Let $0<\gamma<\beta<1$. Given a fixed fiber radius $\epsilon>0$, for any filament centerline $\X(s)\in C^{2,\beta}$ satisfying the non-self-intersection condition \eqref{eq:cGamma}, the inverse $(\frac{1}{2}{\bf I}+\mc{D}')^{-1}$ involving the modified double layer is uniformly bounded with respect to the centerline shape; i.e. given $\varphi\in C^{0,\gamma}$, we have 
\begin{equation}\label{eq:DL_inv}
\norm{\textstyle (\frac{1}{2}{\bf I}+\mc{D}')^{-1}[\varphi]}_{C^{0,\gamma}} \le c(\epsilon,\kappa_{*,\textcolor{black}{\gamma^+}},c_\Gamma) \norm{\varphi}_{C^{0,\gamma}} 
\end{equation}
\textcolor{black}{for any $\gamma^+\in(\gamma, \beta]$.}
\end{lemma}
The proof is given via a compactness argument in section \ref{subsec:doub_inv}. Unlike most of the other bounds we obtain, the $\epsilon$-dependence of the constant in \eqref{eq:DL_inv} is not explicit.

The final lemma regards the mapping properties of the hypersingular operator $\mc{T}$. In particular, we may extract a leading order piece plus a smoother remainder as follows. 
\begin{lemma}[Hypersingular operator bounds]\label{lem:hypersingular}
Let $\X\in C^{2,\alpha}$ be as in section \ref{subsec:geom} and consider $\mc{T}$ as in \eqref{eq:Tdef}. Given $\varphi\in C^{1,\alpha}(\Gamma_\epsilon)$, we may decompose $\mc{T}[\varphi]$ as
\begin{align*}
\mc{T}[\varphi] = \mc{T}_0[\varphi]+\mc{T}_{+}[\varphi]
\end{align*}
where
\begin{equation}\label{eq:RT_lem}
\begin{aligned}
\norm{\mc{T}_0[\varphi]}_{C^{0,\alpha}} &\le c(\kappa_{*,\alpha},c_\Gamma)\bigg(\norm{\p_s\varphi}_{C^{0,\alpha}}+ \textstyle\norm{\frac{1}{\epsilon}\p_\theta\varphi}_{C^{0,\alpha}} \bigg)\\
\norm{\mc{T}_+[\varphi]}_{C^{0,\alpha}} &\le c(\kappa_{*,\alpha},c_\Gamma)\,\epsilon^{-1-\alpha}\norm{\varphi}_{L^\infty}\,.
\end{aligned}
\end{equation}
\end{lemma}
The proof of Lemma \ref{lem:hypersingular} appears in section \ref{subsec:hyper}.\\

We may combine Lemmas \ref{lem:Dsmooths}, \ref{lem:DL_inverse}, and \ref{lem:hypersingular} to show Lemma \ref{lem:w} as follows. 
\begin{proof}[Proof of Lemma \ref{lem:w}]
Given $\theta$-independent Dirichlet data $v(s)$ in \eqref{eq:sol_Dirichlet}, we may solve for the intermediate density $\varphi(\bx)=\varphi(s,\theta)$ as 
\begin{align*}
\varphi(s,\theta) &= 2v(s)-\mc{M}[v(s)]\,, \qquad
\mc{M}=\textstyle \mc{D}'(\frac{1}{2}{\bf I}+\mc{D}')^{-1}\,. 
\end{align*}
Using Lemmas \ref{lem:Dsmooths} and \ref{lem:DL_inverse}, we have that 
\begin{align*}
\norm{\mc{M}[v]}_{C^{1,\gamma}}\le c(\kappa_{*,\textcolor{black}{\gamma^+}},\epsilon)\norm{v}_{C^{0,\gamma}}\,. 
\end{align*}
Note that we will not make use of the full $C^{1,\alpha}$ regularity of $v(s)$ because (1) it is not strictly necessary for our main goal of simply showing that certain remainder terms are smoother, since we will just take $\gamma>\alpha$, and (2) it would involve an even nastier computation of derivatives along $\Gamma_\epsilon$. 

We may use Lemma \ref{lem:hypersingular} to write $w(s,\theta)$ from \eqref{eq:usual_DtN} as
\begin{align*}
w(s,\theta) &= \mc{T}[2v(s)] -\mc{T}\mc{M}[v(s)]\\
&= 2\mc{T}_0[v(s)]+ 2\mc{T}_+[v(s)] - \mc{T}_0\mc{M}[v(s)] - \mc{T}_+\mc{M}[v(s)]\,.
\end{align*}
Since $v\in C^{1,\alpha}(\T)$, we have
\begin{align*}
2\norm{\mc{T}_0[v]}_{C^{0,\alpha}} \le c(\kappa_{*,\alpha},c_\Gamma)\norm{\p_sv}_{C^{0,\alpha}}\,,
\end{align*}
while each of the other terms are smoother. In particular, for $\alpha<\gamma<\textcolor{black}{\gamma^+}\le\beta$, we have
\begin{align*}
2\norm{\mc{T}_+[v]}_{C^{0,\gamma}} &\le c(\kappa_{*,\gamma},c_\Gamma)\,\epsilon^{-1-\gamma}\norm{v}_{L^\infty} \\
\norm{\mc{T}_0\mc{M}[v]}_{C^{0,\gamma}} &\le c(\kappa_{*,\gamma},c_\Gamma)\bigg(\norm{\p_s\mc{M}[v]}_{C^{0,\gamma}}+ \epsilon^{-1-\gamma}\norm{\p_\theta\mc{M}[v]}_{C^{0,\gamma}} \bigg)
\le c(\kappa_{*,\textcolor{black}{\gamma^+}},c_\Gamma,\epsilon)\norm{v}_{C^{0,\gamma}} \\
\norm{\mc{T}_+\mc{M}[v]}_{C^{0,\gamma}} &\le c(\kappa_{*,\gamma},c_\Gamma)\epsilon^{-1-\gamma}\norm{\mc{M}[v]}_{L^\infty}
\le c(\kappa_{*,\textcolor{black}{\gamma^+}},c_\Gamma,\epsilon)\norm{v}_{C^{0,\gamma}}\,.
\end{align*}
Combining the above bounds, we obtain Lemma \ref{lem:w}.
\end{proof}

The remainder of this section is devoted to the proofs of Lemmas \ref{lem:Dsmooths}, \ref{lem:DL_inverse}, and \ref{lem:hypersingular}.

\subsection{Proof of Lemma \ref{lem:Dsmooths}: modified double layer}\label{subsec:mod_doub}
To show the mapping properties of the modified double layer operator $\mc{D}'$ given by \eqref{eq:mod_double_layer}, we begin by writing
\begin{align*}
\mc{D}'[\varphi] &= \mc{R}_{\mc{D}}[\varphi] + \mc{D}_2[\varphi]\,,\\
\mc{D}_2[\varphi] &:= \frac{1}{4\pi}\int_{-1/2}^{1/2}\int_{-\pi}^{\pi}\overline{K}_\mc{D}\,\varphi(s-\bars,\theta-\bartheta) \,\epsilon\, d\bartheta d\bars \\ 
&\qquad+\int_{-1/2}^{1/2}\int_{-\pi}^{\pi}\frac{\varphi(s-\bars,\theta-\bartheta)}{\abs{\X(s)-\X(s-\bars)+\epsilon\be_r(s,\theta)}}\,\mc{J}_\epsilon(s-\bars,\theta-\bartheta)\, d\bartheta d\bars\,,
\end{align*}
where $\mc{R}_\mc{D}$ is as in \eqref{eq:RD_def} and satisfies Lemma \ref{lem:RS_and_RD}, i.e.
\begin{equation}\label{eq:RD_lem_again}
\norm{\mc{R}_\mc{D}[\varphi]}_{C^{1,\gamma}} \le c(\kappa_{*,\textcolor{black}{\gamma^+}},c_\Gamma)\,\epsilon^{-\textcolor{black}{\gamma^+}} \norm{\varphi}_{C^{0,\gamma}}\,. 
\end{equation}
 It thus remains to estimate $\mc{D}_2$.
Using the representation \eqref{eq:barKD} of $\overline{K}_{\mc{D}}$, we have
\begin{align*}
\abs{\mc{D}_2[\varphi]} &\le \norm{\varphi}_{L^\infty}\bigg(c\int_{-1/2}^{1/2}\int_{-\pi}^{\pi}\frac{\epsilon \sin^2(\frac{\bartheta}{2})}{|\barR|^3} \,\epsilon\, d\bartheta d\bars 
+c(\kappa_*)\int_{-1/2}^{1/2}\int_{-\pi}^{\pi}\frac{1}{\epsilon}\,\epsilon\,d\bartheta d\bars\bigg)
\le c(\kappa_*,c_\Gamma)\norm{\varphi}_{L^\infty}\,.
\end{align*}

Furthermore, recalling \eqref{eq:ds_barKD} and \eqref{eq:dtheta_barKD}, we may calculate
\begin{align*}
\p_s\overline{K}_\mc{D} = -6\frac{\epsilon\sin^2(\frac{\bartheta}{2}) \,\bars}{|\barR|^5}\,, \quad
 \frac{1}{\epsilon}\p_\theta \overline K_\mc{D} = \frac{12\epsilon^2\sin^3(\frac{\bartheta}{2})\cos(\frac{\bartheta}{2})}{|\barR|^5} - \frac{2\sin(\frac{\bartheta}{2})\cos(\frac{\bartheta}{2})}{|\barR|^3}\,.
\end{align*}
Then, using Lemma \ref{lem:odd_nm}, we have
\begin{align*}
\abs{\p_s\mc{D}_2[\varphi]} &\le \bigg|\frac{1}{4\pi}\int_{-1/2}^{1/2}\int_{-\pi}^{\pi}6\frac{\epsilon\sin^2(\frac{\bartheta}{2}) \,\bars}{|\barR|^5}\,\varphi(s-\bars,\theta-\bartheta) \,\epsilon\, d\bartheta d\bars\bigg|\\
&\quad + \norm{\varphi}_{L^\infty}\int_{-1/2}^{1/2}\int_{-\pi}^{\pi}\frac{\big|(1+\epsilon\wh\kappa)\be_{\rm t}+\epsilon\kappa_3\be_\theta\big|}{\epsilon^3}\,\abs{\mc{J}_\epsilon}\,d\bartheta d\bars 
\le c(\kappa_*)\,\epsilon^{-1}\norm{\varphi}_{C^{0,\gamma}}
\end{align*}
as well as
\begin{align*}
 \textstyle \abs{\frac{1}{\epsilon}\p_\theta\mc{D}_2[\varphi]} &\le \abs{\frac{1}{4\pi}\int_{-1/2}^{1/2}\int_{-\pi}^{\pi}\bigg(\frac{12\epsilon^2\sin^3(\frac{\bartheta}{2})\cos(\frac{\bartheta}{2})}{|\barR|^5} - \frac{2\sin(\frac{\bartheta}{2})\cos(\frac{\bartheta}{2})}{|\barR|^3} \bigg)\varphi(s-\bars,\theta-\bartheta) \,\epsilon\, d\bartheta d\bars }\\
 &\quad +\norm{\varphi}_{L^\infty}\int_{-1/2}^{1/2}\int_{-\pi}^{\pi}\frac{1}{\epsilon^3}\,\abs{\mc{J}_\epsilon}\,d\bartheta d\bars
\le c\,\epsilon^{-1}\norm{\varphi}_{C^{0,\gamma}}\,.
\end{align*}
Finally, using case (2) of Lemma \ref{lem:alpha_est}, we may obtain $\dot C^{0,\gamma}$ bounds for derivatives of $\mc{D}_2[\varphi]$:
\begin{align*}
\abs{\p_s\mc{D}_2[\varphi]}_{\dot C^{0,\gamma}} &\le \abs{\frac{1}{4\pi}\int_{-1/2}^{1/2}\int_{-\pi}^{\pi}6\frac{\epsilon\sin^2(\frac{\bartheta}{2}) \,\bars}{|\barR|^5}\,\varphi(s-\bars,\theta-\bartheta) \,\epsilon\, d\bartheta d\bars}_{\dot C^{0,\gamma}} \\
&\quad + c(\kappa_{*,\gamma})\norm{\varphi}_{L^\infty}\int_{-1/2}^{1/2}\int_{-\pi}^{\pi}\frac{1}{\epsilon^{2+\gamma}}\,d\bartheta d\bars 
\le c(\kappa_{*,\gamma})\,\epsilon^{-1-\gamma}\norm{\varphi}_{C^{0,\gamma}}
\end{align*}
and
\begin{align*}
 \textstyle \abs{\frac{1}{\epsilon}\p_\theta\mc{D}_2[\varphi]}_{\dot C^{0,\gamma}} &\le \abs{\frac{1}{4\pi}\int_{-1/2}^{1/2}\int_{-\pi}^{\pi}\bigg(\frac{12\epsilon^2\sin^3(\frac{\bartheta}{2})\cos(\frac{\bartheta}{2})}{|\barR|^5} - \frac{2\sin(\frac{\bartheta}{2})\cos(\frac{\bartheta}{2})}{|\barR|^3} \bigg)\varphi(s-\bars,\theta-\bartheta) \,\epsilon\, d\bartheta d\bars }_{\dot C^{0,\gamma}}\\
 &\quad + c(\kappa_{*,\gamma})\norm{\varphi}_{L^\infty}\int_{-1/2}^{1/2}\int_{-\pi}^{\pi}\frac{1}{\epsilon^{2+\gamma}}\,d\bartheta d\bars 
\le c(\kappa_{*,\gamma})\,\epsilon^{-1-\gamma}\norm{\varphi}_{C^{0,\gamma}}\,.
\end{align*}
Combining the estimates for $\mc{D}_2$ with \eqref{eq:RD_lem_again}, we obtain Lemma \ref{lem:Dsmooths}.
\hfill\qedsymbol

\subsection{Proof of Lemma \ref{lem:DL_inverse}: inverse double layer}\label{subsec:doub_inv}
Throughout, we consider the surface density $\varphi=\varphi(s,\theta)$ along $\Gamma_\epsilon$ as a function of the surface parameters $s$ and $\theta$. 
Using the uniform boundedness principle\footnote{The uniform boundedness principle applies if the underlying Banach space is the same. Technically the space $C^{0,\gamma}(\Gamma_\epsilon)$ is different for each curve $\X(s)$; however, these spaces may be identified across different $\X(s)$ using the norm \eqref{eq:dot_Calpha_eps}.}, it will suffice to fix $\varphi\neq 0$ in $C^{0,\gamma}(\Gamma_\epsilon)$ and show that  
\begin{align*}
\textstyle \norm{(\frac{1}{2}{\bf I}+\mc{D}')^{-1}[\varphi]}_{C^{0,\gamma}} \le c_\varphi(\kappa_{*,\textcolor{black}{\gamma^+}},c_\Gamma)\norm{\varphi}_{C^{0,\gamma}}
\end{align*}
for any centerline curve $\X(s)$ satisfying $\norm{\kappa}_{C^{0,\textcolor{black}{\gamma^+}}}\le \kappa_{*,\textcolor{black}{\gamma^+}}$ as well as the non-self-intersection condition \eqref{eq:cGamma} with the same $c_\Gamma$. \emph{A priori} the constant $c_\varphi$ may depend on the function $\varphi$; however, since $\varphi$ is arbitrary, uniform boundedness will imply that the bound in fact holds for a constant that is uniform in $\varphi$. 

We proceed by contradiction. Fix $\norm{\varphi}_{C^{0,\gamma}}=1$. If no such bound exists, then we may select a sequence of curves $\X_j$, $j\in \N$, each satisfying \eqref{eq:cGamma} and $\norm{(\X_j)_{ss}}_{C^{0,\textcolor{black}{\gamma^+}}}=\norm{\kappa_j}_{C^{0,\textcolor{black}{\gamma^+}}}\le \kappa_{*,\textcolor{black}{\gamma^+}}$, such that
\begin{align*}
\textstyle \norm{(\frac{1}{2}{\bf I}+\mc{D}_j')^{-1}[\varphi]}_{C^{0,\gamma}} >j\,,
\end{align*}  
where $\mc{D}_j'$ denotes the modified double layer operator \eqref{eq:mod_double_layer} on curve $j$.
Defining
\begin{align*}
h_j =\frac{\varphi}{\|(\frac{1}{2}{\bf I}+\mc{D}_j')^{-1}[\varphi]\|_{C^{0,\gamma}}}\,, \quad
g_j = \frac{(\frac{1}{2}{\bf I}+\mc{D}_j')^{-1}[\varphi]}{\|(\frac{1}{2}{\bf I}+\mc{D}_j')^{-1}[\varphi]\|_{C^{0,\gamma}}}\,,
\end{align*}
and noting that $\norm{g_j}_{C^{0,\gamma}}=1$, we have that, by assumption,
\begin{equation}\label{eq:j_contradiction}
\textstyle (\frac{1}{2}{\bf I}+\mc{D}_j')[g_j] = h_j \to 0 \quad \text{in }C^{0,\gamma}\,.
\end{equation}
Note that by Lemma \ref{lem:Dsmooths}, $\mc{D}_j'[g_j]$ is smoother and is bounded in $C^{1,\gamma}$ uniformly in $j$. Thus by \eqref{eq:j_contradiction}, along a subsequence $j_\ell$, $g_{j_\ell}-2h_{j_\ell}$ converges strongly in $C^{0,\gamma}$ to some limit $g_\infty$ with $\norm{g_\infty}_{C^{0,\gamma}}=1$.

Now, since $\norm{(\X_j)_{ss}}_{C^{0,\textcolor{black}{\gamma^+}}}=\norm{\kappa_j}_{C^{0,\textcolor{black}{\gamma^+}}}\le \kappa_{*,\textcolor{black}{\gamma^+}}$, there exists a subsequence $\X_{j_k}$ of curves converging (after a possible translation) strongly in $C^{2,\gamma}$, $0<\gamma<\textcolor{black}{\gamma^+}$, to some limit curve $\X_\infty(s)$ which also satisfies the assumption \eqref{eq:cGamma}. In particular, using the form \eqref{eq:KD} of the double layer kernel $K_\mc{D}$, we have, for any $\psi\in C^{0,\gamma}(\Gamma_\epsilon)$,
\begin{align*}
&(\mc{D}_{j_k}'-\mc{D}_\infty')[\psi]\\
&= \frac{1}{4\pi}\int_{-1/2}^{1/2}\int_{-\pi}^\pi\bigg(\frac{(\X_{j_k}(s)-\X_{j_k}(s'))\cdot\be_{r,j_k}(s',\theta') +\epsilon(\be_{r,j_k}(s,\theta)\cdot\be_{r,j_k}(s',\theta')-1)}{\abs{\X_{j_k}(s)-\X_{j_k}(s')+\epsilon(\be_{r,j_k}(s,\theta)-\be_{r,j_k}(s',\theta'))}^3}\mc{J}_{\epsilon,j_k}(s',\theta')\\
&\quad - \frac{(\X_\infty(s)-\X_\infty(s'))\cdot\be_{r,\infty}(s',\theta') +\epsilon(\be_{r,\infty}(s,\theta)\cdot\be_{r,\infty}(s',\theta')-1)}{\abs{\X_\infty(s)-\X_\infty(s')+\epsilon(\be_{r,\infty}(s,\theta)-\be_{r,\infty}(s',\theta'))}^3}\mc{J}_{\epsilon,\infty}(s',\theta')\bigg)\,\psi(s',\theta')\,d\theta'ds \\
&\quad+ \int_{-1/2}^{1/2}\int_{-\pi}^\pi\bigg(\frac{\mc{J}_{\epsilon,j_k}(s',\theta')}{\abs{\X_{j_k}(s)-\X_{j_k}(s')+\epsilon\be_{r,j_k}(s,\theta)}}- \frac{\mc{J}_{\epsilon,\infty}(s',\theta')}{\abs{\X_\infty(s)-\X_\infty(s')+\epsilon\be_{r,\infty}(s,\theta)}}\bigg) \psi(s',\theta')\,d\theta' ds'\,,
 \end{align*} 
 where each of $\be_{r,j}$, $\mc{J}_{\epsilon,j}$, $\wh\kappa_{j}$ are defined along the curve $\X_j$. From this we can see that $\abs{(\mc{D}_{j_k}'-\mc{D}_\infty')[\psi]}\to 0$ as $j_k\to\infty$. By a diagonalization argument, we therefore have that 
\begin{align*}
\mc{D}_{j}'g_{j} \to D_\infty'g_\infty 
\end{align*}
along some subsequence, and $g_\infty$ satisfies $(\frac{1}{2}{\bf I}+\mc{D}_\infty')[g_\infty]=0$. By injectivity of the modified double layer (see discussion below \eqref{eq:sol_Dirichlet}), we have $g_\infty=0$, which contradicts $\norm{g_\infty}_{C^{0,\gamma}}=1$. 
\hfill\qedsymbol \\

\subsection{Proof of Lemma \ref{lem:hypersingular}: bounds for hypersingular operator}\label{subsec:hyper}
We begin by writing down a more detailed expression for the kernel $K_{\mc{T}}$ given by \eqref{eq:KT}. Using \eqref{eq:Qexpand}, \eqref{eq:Rnxprime}, and \eqref{eq:Rnx}, we have
\begin{align*}
K_\mc{T} &= -\frac{1}{4\pi}\bigg(\frac{\bm{n}_x\cdot\bm{n}_{x'}}{\abs{\bR}^3} -3\frac{\bR\cdot\bm{n}_x \,\bR\cdot\bm{n}_{x'}}{\abs{\bR}^5}\bigg)  \\
&= -\frac{1}{4\pi}\bigg(\frac{1-2\sin^2(\frac{\bartheta}{2}) +\bars^2 Q_{0,5}}{\abs{\bR}^3} -\frac{12\epsilon^2\sin^4(\frac{\bartheta}{2}) + 6\epsilon\sin^2(\frac{\bartheta}{2})\bars^2(Q_{\rm n}-Q_{\rm n'}) + 3\bars^4Q_{\rm n}Q_{\rm n'}}{\abs{\bR}^5}\bigg)\,.
\end{align*} 
We consider the hypersingular part of the operator $\mc{T}$ in \eqref{eq:Tdef} in terms of the surface parameters $s$ and $\theta$:
\begin{align*}
\mc{T}[\varphi] &= {\rm p.v.}\int_{-1/2}^{1/2}\int_{-\pi}^\pi K_{\mc{T}}(s,\theta,\bars,\bartheta)\,\big(\varphi(s-\bars,\theta-\bartheta)-\varphi(s,\theta)\big)\,\mc{J}_\epsilon(s-\bars,\theta-\bartheta)\, d\bartheta d\bars \\
&\qquad + \int_{\Gamma_\epsilon}\frac{(\bx-\X(s'))\cdot\bm{n}_x}{\abs{\bx-\X(s')}^3}\varphi(\bx')\,dS_{x'} \,.
\end{align*}
We may then decompose $\mc{T}$ as follows:
\begin{align*} 
\mc{T}[\varphi] &= \mc{T}_0[\varphi]+\mc{T}_{+}[\varphi]\,,\\
\mc{T}_0[\varphi] &= -\frac{1}{4\pi}\,{\rm p.v.}\int_{-1/2}^{1/2}\int_{-\pi}^{\pi} \frac{1}{\abs{\bR}^3}\,\big(\varphi(s-\bars,\theta-\bartheta)-\varphi(s,\theta)\big)\,\mc{J}_\epsilon\, d\bartheta d\bars\,,\\
\mc{T}_{+}[\varphi] &= \frac{1}{4\pi}\int_{-1/2}^{1/2}\int_{-\pi}^{\pi}\bigg(\frac{12\epsilon^2\sin^4(\frac{\bartheta}{2}) + 6\epsilon\sin^2(\frac{\bartheta}{2})\bars^2(Q_{\rm n}-Q_{\rm n'}) + 3\bars^4Q_{\rm n}Q_{\rm n'}}{\abs{\bR}^5} \\
&\qquad +\frac{2\sin^2(\frac{\bartheta}{2}) -\bars^2 Q_{0,5}}{\abs{\bR}^3}\bigg) \, \big(\varphi(s-\bars,\theta-\bartheta)-\varphi(s,\theta)\big)\,\mc{J}_\epsilon\, d\bartheta d\bars 
+ \int_{\Gamma_\epsilon}\frac{(\bx-\X(s'))\cdot\bm{n}_x}{\abs{\bx-\X(s')}^3}\varphi(\bx')\,dS_{x'}\,.
\end{align*}
Here $\mc{T}_0$ captures the main behavior of $\mc{T}$ while $\mc{T}_{+}$ is smoother. We begin by estimating $\mc{T}_{+}$. Using Lemma \ref{lem:basic_est}, we have
\begin{align*}
\abs{\mc{T}_+} &\le c(\kappa_*)\norm{\varphi}_{L^\infty}\int_{-1/2}^{1/2}\int_{-\pi}^{\pi}\bigg(\frac{\epsilon^{-2}+\epsilon^{-1}+1}{\abs{\bR}}+\frac{1}{\epsilon^2}\bigg) \,\epsilon d\bartheta d\bars 
\le c(\kappa_*,c_\Gamma)\,\epsilon^{-1}\norm{\varphi}_{L^\infty}\,.
\end{align*}
Using case (1) of Lemma \ref{lem:alpha_est}, we may also estimate
\begin{align*}
\abs{\mc{T}_+}_{\dot C^{0,\alpha}} 
\le c(\kappa_{*,\alpha},c_\Gamma)\,\epsilon^{-1-\alpha}\norm{\varphi}_{L^\infty}\,.
\end{align*}

We next turn to bounds for $\mc{T}_0$. Since $\varphi\in C^{1,\alpha}(\Gamma_\epsilon)$, we may write 
\begin{equation}\label{eq:varphi_expand}
\begin{aligned}
\varphi(s-\bars,\theta-\bartheta)- \varphi(s,\theta) &= \textstyle -\bars \,\p_s\varphi(s,\theta) - \epsilon\bartheta\,\frac{1}{\epsilon}\p_\theta\varphi(s,\theta) \\
&\qquad+ \abs{\bars}^{1+\alpha} Q_{\varphi,s}(s,\theta,\bars,\bartheta) + |\epsilon\bartheta|^{1+\alpha} Q_{\varphi,\theta}(s,\theta,\bars,\bartheta) \,,
\end{aligned}
\end{equation} 
where 
\begin{equation}\label{eq:Qvarphi}
\norm{Q_{\varphi,s}}_{L^\infty}\le c\norm{\p_s\varphi}_{C^{0,\alpha}} \,, 
\qquad
\norm{Q_{\varphi,\theta}}_{L^\infty}\le \textstyle c\norm{\frac{1}{\epsilon}\p_\theta\varphi}_{C^{0,\alpha}} \,.
\end{equation}
We may then write 
\begin{align*}
\mc{T}_0 &= J_1+J_2\,,\\
J_1 &= \frac{1}{4\pi}{\rm p.v.}\int_{-1/2}^{1/2}\int_{-\pi}^\pi \frac{1}{\abs{\bR}^3}\bigg(\bars\p_s\varphi(s,\theta)+\epsilon\bartheta\, \textstyle \frac{1}{\epsilon}\p_\theta\varphi(s,\theta) \bigg)\,\mc{J}_\epsilon(s-\bars,\theta-\bartheta)\, d\bartheta d\bars\\
J_2 &= -\frac{1}{4\pi}\int_{-1/2}^{1/2}\int_{-\pi}^\pi\frac{1}{\abs{\bR}^3}\big( \abs{\bars}^{1+\alpha} Q_{\varphi,s} + |\epsilon\bartheta|^{1+\alpha} Q_{\varphi,\theta}\big)\,\mc{J}_\epsilon(s-\bars,\theta-\bartheta)\, d\bartheta d\bars \,.
\end{align*}
We begin with $L^\infty$ estimates for $J_1$ and $J_2$. First, using Lemma \ref{lem:odd_nm}, we may estimate $J_1$ as 
\begin{align*}
 \abs{J_1}&\le c\,\norm{\p_s\varphi}_{L^\infty}\abs{{\rm p.v.}\int_{-1/2}^{1/2}\int_{-\pi}^\pi \frac{\bars}{\abs{\bR}^3}\,\mc{J}_\epsilon\, d\bartheta d\bars}
 + c\,\norm{\textstyle \frac{1}{\epsilon}\p_\theta\varphi}_{L^\infty}\displaystyle\abs{{\rm p.v.}\int_{-1/2}^{1/2}\int_{-\pi}^\pi \frac{\epsilon\bartheta}{\abs{\bR}^3}\,\mc{J}_\epsilon\, d\bartheta d\bars} \\ 
 &\le c(\kappa_{*,\alpha},c_\Gamma)\,\epsilon^\alpha\big(\norm{\p_s\varphi}_{L^\infty}+\textstyle \norm{\frac{1}{\epsilon}\p_\theta\varphi}_{L^\infty}\big)\,.
\end{align*}  
 We may next use Lemma \ref{lem:basic_est} to estimate $J_2$ as 
\begin{align*}
\abs{J_2}&\le c(\kappa_*)\big(\norm{Q_{\varphi,s}}_{L^\infty}+\norm{Q_{\varphi,\theta}}_{L^\infty}\big)\int_{-1/2}^{1/2}\int_{-\pi}^\pi\frac{1}{\abs{\bR}^{2-\alpha}}\,\epsilon d\bartheta d\bars\\
&\le c(\kappa_*,c_\Gamma)\,\epsilon^\alpha\big(\norm{\p_s\varphi}_{C^{0,\alpha}}+\textstyle \norm{\frac{1}{\epsilon}\p_\theta\varphi}_{C^{0,\alpha}}\big)\,.
\end{align*}
In total we have
\begin{equation}\label{eq:J_C0est}
\abs{\mc{T}_0} \le c(\kappa_{*,\alpha},c_\Gamma)\,\epsilon^\alpha\big(\norm{\p_s\varphi}_{C^{0,\alpha}}+\textstyle \norm{\frac{1}{\epsilon}\p_\theta\varphi}_{C^{0,\alpha}}\big)\,.
\end{equation}

To prove the $\dot C^{0,\alpha}$ estimate for $\mc{T}_0$, we may follow the approach of Lemma \ref{lem:alpha_est} while being a bit more careful with the structure of remainder terms. We again let $\bR$, $\bR_0$ denote
\begin{align*}
\bR = \bR(s,\theta,\bars,\bartheta), \quad \bR_0 = \bR(s_0+s,\theta_0+\theta,s_0+\bars,\theta_0+\bartheta)\,,
\end{align*}
and aim to bound the expression 
\begin{align*}
\wh{\mc{T}_0}&:= \mc{T}_0[\varphi](s_0+s,\theta_0+\theta)-\mc{T}_0[\varphi](s,\theta) \\
&= -\frac{1}{4\pi}{\rm p.v.}\int_{-s_0-1/2}^{-s_0+1/2}\int_{-\theta_0-\pi}^{-\theta_0+\pi} \frac{1}{\abs{\bR_0}^3}\,\big(\varphi(s-\bars,\theta-\bartheta)-\varphi(s_0+s,\theta_0+\theta)\big)\,\mc{J}_\epsilon(s-\bars,\theta-\bartheta) \,d\bartheta d\bars \\
&\qquad +\frac{1}{4\pi} {\rm p.v.}\int_{-1/2}^{1/2}\int_{-\pi}^\pi \frac{1}{\abs{\bR}^3} \,\big(\varphi(s-\bars,\theta-\bartheta)-\varphi(s,\theta)\big)\,\mc{J}_\epsilon(s-\bars,\theta-\bartheta) \,d\bartheta d\bars\,.
\end{align*}
We begin by noting two different expansions for $\varphi$ about $(s_0+s,\theta_0+\theta)$. In particular, we write
\begin{equation}\label{eq:varphi_expand2}
\varphi(s-\bars,\theta-\bartheta)- \varphi(s_0+s,\theta_0+\theta) 
=(s_0+\bars) H_{\varphi,s} + \epsilon(\theta_0+\bartheta)H_{\varphi,\theta}\,,
\end{equation}
where we will make use of two different forms of the remainder terms. Expansion 1 is given by
\begin{equation}\label{eq:rem1}
\begin{aligned}
(s_0+\bars)H_{\varphi,s} &= -(s_0+\bars)\p_s\varphi(s_0+s,\theta_0+\theta)\\
&\qquad +\abs{s_0+\bars}^{1+\alpha}Q_{\varphi,s}(s_0+s,\theta_0+\theta,s_0+\bars,\theta_0+\bartheta)\\ 
\epsilon(\theta_0+\bartheta)H_{\varphi,\theta}&= \textstyle -\epsilon(\theta_0+\bartheta)\,\frac{1}{\epsilon}\p_\theta\varphi(s_0+s,\theta_0+\theta)\\
&\qquad +\abs{\epsilon(\theta_0+\bartheta)}^{1+\alpha}Q_{\varphi,\theta}(s_0+s,\theta_0+\theta,s_0+\bars,\theta_0+\bartheta)\,,
\end{aligned}
\end{equation}
while expansion 2 is given by 
\begin{equation}\label{eq:rem2}
\begin{aligned}
(s_0+\bars)H_{\varphi,s} &= -(s_0+\bars)\p_s\varphi(s,\theta)+\abs{s_0}^{1+\alpha}Q_{\varphi,s}(s,\theta,s_0+s,\theta_0+\theta)\\
&\qquad +\abs{\bars}^{1+\alpha}Q_{\varphi,s}(s,\theta,\bars,\bartheta)\\ 
\epsilon(\theta_0+\bartheta)H_{\varphi,\theta}&=\textstyle -\epsilon(\theta_0+\bartheta)\,\frac{1}{\epsilon}\p_\theta\varphi(s,\theta)+\abs{\epsilon\theta_0}^{1+\alpha}Q_{\varphi,\theta}(s,\theta,s_0+s,\theta_0+\theta) \\
&\qquad+ |\epsilon\bartheta|^{1+\alpha}Q_{\varphi,\theta}(s,\theta,\bars,\bartheta)\,.
\end{aligned}
\end{equation}
Here each $Q_{\varphi,\mu}$ is as in \eqref{eq:Qvarphi}.

Again as in Lemma \ref{lem:alpha_est}, we first consider the case $\sqrt{s_0^2+\epsilon^2\theta_0^2}\ge \epsilon$. Then, using the first expansion \eqref{eq:rem1} of the remainder terms, we may use the $L^\infty$ estimate \eqref{eq:J_C0est} from above to obtain 
\begin{align*}
\abs{\wh{\mc{T}}_0}&\le \abs{\mc{T}_0[\varphi](s_0+s,\theta_0+\theta)} + \abs{\mc{T}_0[\varphi](s,\theta)}  
\le c(\kappa_{*,\alpha},c_\Gamma)\,\epsilon^\alpha\big(\norm{\p_s\varphi}_{C^{0,\alpha}}+\textstyle \norm{\frac{1}{\epsilon}\p_\theta\varphi}_{C^{0,\alpha}}\big) \\
&\le c(\kappa_{*,\alpha},c_\Gamma)\,\sqrt{s_0^2+\epsilon^2\theta_0^2}^{\,\alpha}\big(\norm{\p_s\varphi}_{C^{0,\alpha}}+\textstyle \norm{\frac{1}{\epsilon}\p_\theta\varphi}_{C^{0,\alpha}}\big)\,.
\end{align*}

We next consider the case $\sqrt{s_0^2+\epsilon^2\theta_0^2}< \epsilon$. 
Using \eqref{eq:Tdef} and the expansions \eqref{eq:varphi_expand} and \eqref{eq:varphi_expand2}, we may write 
\begin{align*}
\wh{\mc{T}}_0 &= J_{\alpha,{\rm s}}+J_{\alpha,\theta} \,,\\
J_{\alpha,{\rm s}}&= -\frac{1}{4\pi}{\rm p.v.}\int_{-1/2-s_0}^{1/2-s_0}\int_{-\pi-\theta_0}^{\pi-\theta_0}\frac{1}{\abs{\bR_0}^3}(s_0+\bars)H_{\varphi,s}\,\mc{J}_\epsilon(s-\bars,\theta-\bartheta)\,d\bartheta d\bars \\
&\quad + \frac{1}{4\pi}{\rm p.v.}\int_{-1/2}^{1/2}\int_{-\pi}^{\pi}\frac{1}{\abs{\bR}^3}\big(-\bars\p_s\varphi(s,\theta)+\abs{\bars}^{1+\alpha}Q_{\varphi,s}(s,\theta,\bars,\bartheta) \big) \,\mc{J}_\epsilon(s-\bars,\theta-\bartheta)\,d\bartheta d\bars\\
J_{\alpha,\theta}&= -\frac{1}{4\pi}{\rm p.v.}\int_{-1/2-s_0}^{1/2-s_0}\int_{-\pi-\theta_0}^{\pi-\theta_0} \frac{1}{\abs{\bR_0}^3}\epsilon(\theta_0+\bartheta)H_{\varphi,\theta}\,\mc{J}_\epsilon(s-\bars,\theta-\bartheta) d\bartheta d\bars\\
&\quad  + \frac{1}{4\pi}{\rm p.v.}\int_{-1/2}^{1/2}\int_{-\pi}^{\pi} \frac{1}{\abs{\bR}^3}\bigg(\textstyle -\epsilon\bartheta \,\frac{1}{\epsilon}\p_\theta\varphi(s,\theta)+|\epsilon\bartheta|^{1+\alpha}Q_{\varphi,\theta}(s,\theta,\bars,\bartheta) \bigg)\,\mc{J}_\epsilon(s-\bars,\theta-\bartheta) d\bartheta d\bars\,.
\end{align*}
To bound $J_{\alpha,{\rm s}}$ and $J_{\alpha,\theta}$, we split the integrals into two regions:
\begin{align*}
I_1 &= \big\{ (\bars,\bartheta)\; : \; \sqrt{(s_0+\bars)^2+\epsilon^2(\theta_0+\bartheta)^2}\le 4\sqrt{s_0^2+\epsilon^2\theta_0^2} \big\}\,,\\
I_2 &= \big\{ (\bars,\bartheta)\; : \; \sqrt{(s_0+\bars)^2+\epsilon^2(\theta_0+\bartheta)^2}> 4\sqrt{s_0^2+\epsilon^2\theta_0^2} \big\}\,.
\end{align*}
For $j={\rm s},\theta$, let $J_{\alpha,j,1}$ denote the integral of the integrand of $J_{\alpha,j}$ over the region $I_1$. Note that $\sqrt{\bars^2+\epsilon^2\bartheta^2}$ then belongs to the region $I_1'$ where $\sqrt{\bars^2+\epsilon^2\bartheta^2}\le 5\sqrt{s_0^2+\epsilon^2\theta_0^2}$. Using the first expansion \eqref{eq:rem1} for $\varphi$ about $(s_0+s,\theta_0+\theta)$ and recalling the definition \eqref{eq:Reven}, \eqref{eq:Rsq2} of $\bR_{\rm even}$, we have  
\begin{align*}
\abs{J_{\alpha,{\rm s},1}} 
&\le \abs{\p_s\varphi(s_0+s,\theta_0+\theta)\,{\rm p.v.}\iint_{I_1} \frac{s_0+\bars}{\abs{\bR_0}^3}\,\mc{J}_\epsilon(s-\bars,\theta-\bartheta)\, d\bartheta d\bars}
 + \iint_{I_1}\frac{\abs{s_0+\bars}^{1+\alpha}}{\abs{\bR_0}^3}\abs{Q_{\varphi,s}} \,\abs{\mc{J}_\epsilon} d\bartheta d\bars \\
&\qquad + \abs{\p_s\varphi(s,\theta)\,{\rm p.v.}\iint_{I_1'} \frac{\bars}{\abs{\bR}^3}\,\mc{J}_\epsilon(s-\bars,\theta-\bartheta)\, d\bartheta d\bars}
 + \iint_{I_1'} \frac{\abs{\bars}^{1+\alpha}}{\abs{\bR}^3}\abs{Q_{\varphi,s}}\,\abs{\mc{J}_\epsilon} d\bartheta d\bars \\
&\le \norm{\p_s\varphi}_{L^\infty}\abs{{\rm p.v.}\iint_{I_1} (s_0+\bars)\frac{\mc{J}_\epsilon(s-\bars,\theta-\bartheta)-\mc{J}_\epsilon(s_0+s,\theta_0+\theta)}{\abs{\bR_0}^3}d\bartheta d\bars} \\
&\qquad + \norm{\p_s\varphi}_{L^\infty}\abs{{\rm p.v.}\iint_{I_1} (s_0+\bars)\bigg(\frac{1}{\abs{\bR_0}^3}-\frac{1}{\abs{\bR_{\rm even,0}}^3}\bigg) \,\mc{J}_\epsilon(s_0+s,\theta_0+\theta) \,d\bartheta d\bars}\\
&\qquad + \norm{\p_s\varphi}_{L^\infty}\abs{{\rm p.v.}\iint_{I_1'} \bars\,\frac{\mc{J}_\epsilon(s-\bars,\theta-\bartheta)-\mc{J}_\epsilon(s,\theta)}{\abs{\bR}^3} d\bartheta d\bars}  \\
&\qquad + \norm{\p_s\varphi}_{L^\infty}\abs{{\rm p.v.}\iint_{I_1'} \bars\bigg(\frac{1}{\abs{\bR}^3}-\frac{1}{\abs{\bR_{\rm even}}^3}\bigg)\,\mc{J}_\epsilon(s,\theta)\, d\bartheta d\bars}  \\
&\qquad + c(\kappa_*)\norm{\p_s\varphi}_{C^{0,\alpha}}\bigg(\iint_{I_1}\frac{1}{\abs{\bR_0}^{2-\alpha}} \,\epsilon d\bartheta d\bars + \iint_{I_1'} \frac{1}{\abs{\bR}^{2-\alpha}}\,\epsilon d\bartheta d\bars\bigg) \\
%
%
%
%
%
&\le c(\kappa_{*,\alpha})\norm{\p_s\varphi}_{C^{0,\alpha}}\bigg(\iint_{I_1}\frac{1}{\abs{\bR_0}^{2-\alpha}} \epsilon\, d\bartheta d\bars + \iint_{I_1'} \frac{1}{\abs{\bR}^{2-\alpha}} \epsilon\, d\bartheta d\bars\bigg) \\
&\le c(\kappa_{*,\alpha})\norm{\p_s\varphi}_{C^{0,\alpha}}\bigg(\iint_{\rho\le4\sqrt{s_0^2+\epsilon^2\theta_0^2}}\frac{1}{\rho^{2-\alpha}} \,\rho d\rho d\phi + \iint_{\rho\le 5\sqrt{s_0^2+\epsilon^2\theta_0^2}} \frac{1}{\rho^{2-\alpha}}\,\rho d\rho d\phi\bigg) \\
&\le c(\kappa_{*,\alpha},c_\Gamma)\norm{\p_s\varphi}_{C^{0,\alpha}}\sqrt{s_0^2+\epsilon^2\theta_0^2}^{\,\alpha} \,,
\end{align*}
where we are using the notation $\bR_{\rm even,0}$ to denote $\bR_{\rm even}(s_0+s,\theta_0+\theta,s_0+\bars,\theta_0+\bartheta)$, and in the second-to-last line we have switched to polar coordinates as in the proof \eqref{eq:alphaest2_1} of Lemma \ref{lem:alpha_est}. 
Similarly, we may show that $J_{\alpha,\theta,1}$ satisfies 
\begin{align*}
\abs{J_{\alpha,\theta,1}} &\le c(\kappa_{*,\alpha},c_\Gamma)\,\epsilon^{-1-\alpha}\textstyle\norm{\p_\theta\varphi}_{C^{0,\alpha}} \sqrt{s_0^2+\epsilon^2\theta_0^2}^{\,\alpha} \,.
\end{align*}

Finally, we consider the integrals $J_{\alpha,{\rm s},2}$ and $J_{\alpha,\theta,2}$ over the region $I_2$. Note that $\sqrt{\bars^2+\epsilon^2\bartheta^2}$ belongs to the region $I_2'$ where $\sqrt{\bars^2+\epsilon^2\bartheta^2}> 3\sqrt{s_0^2+\epsilon^2\theta_0^2}$. 
Using the second expansion \eqref{eq:rem2} of $\varphi$ about $(s_0+s,\theta_0+\theta)$, we have that $J_{\alpha,{\rm s},2}$ may be written 
\begin{align*}
J_{\alpha,{\rm s},2}&= 
-\p_s\varphi(s,\theta)\iint_{I_2} \bigg(\frac{(s_0+\bars)}{\abs{\bR_0}^3}-\frac{\bars}{\abs{\bR}^3} \bigg)\,\big(\mc{J}_\epsilon(s-\bars,\theta-\bartheta)-\mc{J}_\epsilon(s,\theta)\big)\, d\bartheta d\bars \\
&\qquad -\p_s\varphi(s,\theta)\iint_{I_2} \bigg(\frac{(s_0+\bars)}{\abs{\bR_0}^3}-\frac{(s_0+\bars)}{\abs{\bR_{\rm even,0}}^3} + \frac{\bars}{\abs{\bR_{\rm even}}^3}-\frac{\bars}{\abs{\bR}^3} \bigg)\,\mc{J}_\epsilon(s,\theta)\, d\bartheta d\bars \\
&\qquad + \abs{s_0}^{1+\alpha}Q_{\varphi,s}(s,\theta,s_0+s,\theta_0+\theta)\iint_{I_2} \frac{1}{\abs{\bR_0}^3}\,\mc{J}_\epsilon(s-\bars,\theta-\bartheta)\, d\bartheta d\bars\,,
\end{align*}
where we are again using oddness in the middle line to insert $\bR_{\rm even}$ as in \eqref{eq:Reven}.
Then, expanding the differences $\frac{(s_0+\bars)}{\abs{\bR_0}^3}-\frac{(s_0+\bars)}{\abs{\bR_{\rm even,0}}^3}$ and $\frac{\bars}{\abs{\bR_{\rm even}}^3}-\frac{\bars}{\abs{\bR}^3}$ using \eqref{eq:Rsq2}, we have
\begin{align*}
\abs{J_{\alpha,{\rm s},2}} 
&\le c(\kappa_{*,\alpha})\norm{\p_s\varphi}_{L^\infty}\iint_{I_2} \abs{\frac{(s_0+\bars)}{\abs{\bR_0}^3}-\frac{\bars}{\abs{\bR}^3}}\,\big(\abs{\bars}^\alpha+|\epsilon\bartheta|^\alpha\big)\,\epsilon\, d\bartheta d\bars \\
&\quad + \norm{\p_s\varphi}_{L^\infty}\bigg|\iint_{I_2} \bigg(\frac{\bars^5Q_{R,1} +\epsilon\bars^4 Q_{R,2}+ \epsilon^2 \bars^3\sin(\frac{\bartheta}{2})Q_{R,3} }{\abs{\bR_{\rm even}}\abs{\bm{R}}(\abs{\bm{R}}+\abs{\bR_{\rm even}})}\sum_{j=0}^{2}\frac{1}{\abs{\bR}^j\abs{\bR_{\rm even}}^{2-j}} \\
&\quad - \frac{(s_0+\bars)^5Q_{R,1} +\epsilon(s_0+\bars)^4 Q_{R,2}+ \epsilon^2 (s_0+\bars)^3\sin(\frac{\theta_0+\bartheta}{2})Q_{R,3} }{\abs{\bR_{\rm even,0}}\abs{\bm{R}_0}(\abs{\bm{R}_0}+\abs{\bR_{\rm even,0}})}\sum_{j=0}^{2}\frac{1}{\abs{\bR_0}^j\abs{\bR_{\rm even,0}}^{2-j}}\bigg) \mc{J}_\epsilon(s,\theta) d\bartheta d\bars\bigg| \\
&\quad + \abs{s_0}^{1+\alpha}\norm{\p_s\varphi}_{C^{0,\alpha}}\iint_{I_2} \frac{1}{\abs{\bR_0}^3}\,\epsilon d\bartheta d\bars\\
&\le c(\kappa_{*,\alpha})\norm{\p_s\varphi}_{L^\infty}\iint_{I_2} \frac{\abs{s_0}}{\abs{\bR_0}^{3-\alpha}}\,\epsilon\, d\bartheta d\bars + c(\kappa_*)\norm{\p_s\varphi}_{L^\infty}\iint_{I_2} \frac{\sqrt{s_0^2+\epsilon^2\theta_0^2}}{\abs{\barR_0}^2}\,\epsilon d\bartheta d\bars\\
&\quad + c(\kappa_{*,\alpha})\norm{\p_s\varphi}_{L^\infty}\iint_{I_2} \frac{\sqrt{s_0^2+\epsilon^2\theta_0^2}^\alpha}{\abs{\barR_0}}\,\epsilon d\bartheta d\bars 
+ \norm{\p_s\varphi}_{C^{0,\alpha}}\iint_{I_2} \frac{\abs{s_0}^{1+\alpha}}{\abs{\bR_0}^3}\,\epsilon d\bartheta d\bars\\
&\le c(\kappa_{*,\alpha})\norm{\p_s\varphi}_{C^{0,\alpha}}\iint_{\rho>4\sqrt{s_0^2+\epsilon^2\theta_0^2}}\bigg(\frac{\abs{s_0}}{\rho^{3-\alpha}}+ \frac{\sqrt{s_0^2+\epsilon^2\theta_0^2}}{\rho^2}+\frac{\sqrt{s_0^2+\epsilon^2\theta_0^2}^{\,\alpha}}{\rho}+\frac{\abs{s_0}^{1+\alpha}}{\rho^3} \bigg)\,\rho d\rho d\phi \\
&\le c(\kappa_{*,\alpha},c_\Gamma)\norm{\p_s\varphi}_{C^{0,\alpha}}\bigg(\abs{s_0}\big(1+\sqrt{s_0^2+\epsilon^2\theta_0^2}^{\,-1+\alpha}\big)+\sqrt{s_0^2+\epsilon^2\theta_0^2}\big(1+ \abs{\log(s_0^2+\epsilon^2\theta_0^2)}\big) \\
&\qquad +\sqrt{s_0^2+\epsilon^2\theta_0^2}^{\,\alpha}\big(1+\sqrt{s_0^2+\epsilon^2\theta_0^2}\big)+ \abs{s_0}^{1+\alpha}\big(1+\sqrt{s_0^2+\epsilon^2\theta_0^2}^{\,-1}\big)\bigg) \\
&\le c(\kappa_{*,\alpha},c_\Gamma)\sqrt{s_0^2+\epsilon^2\theta_0^2}^{\,\alpha}\norm{\p_s\varphi}_{C^{0,\alpha}}\,.
\end{align*}
Here we have used the closeness of the kernels with respect to $s_0$ and $\epsilon\theta_0$ to obtain the second inequality, and have switched to polar coordinates in the third inequality. 
Similarly, we may show that $J_{\alpha,\theta,2}$ satisfies 
\begin{align*}
\abs{J_{\alpha,\theta,2}}&\le c(\kappa_{*,\alpha},c_\Gamma)\sqrt{s_0^2+\epsilon^2\theta_0^2}^{\,\alpha}\,\textstyle\norm{\frac{1}{\epsilon}\p_\theta\varphi}_{C^{0,\alpha}}\,.
\end{align*}
In total, we have 
\begin{equation}\label{eq:hatT0_est}
\abs{\wh{\mc{T}}_0} \le c(\kappa_{*,\alpha},c_\Gamma)\sqrt{s_0^2+\epsilon^2\theta_0^2}^{\,\alpha}\bigg(\norm{\p_s\varphi}_{C^{0,\alpha}}+ \textstyle \norm{\frac{1}{\epsilon}\p_\theta\varphi}_{C^{0,\alpha}} \bigg)\,.
\end{equation}
Combining \eqref{eq:J_C0est} and \eqref{eq:hatT0_est}, we obtain Lemma \ref{lem:hypersingular}.
\hfill\qedsymbol \\



\section{Slender body NtD in H\"older spaces}\label{sec:SBNtD_Holder}
In this section we prove Lemma \ref{lem:SB_PDE_holder} regarding H\"older space estimates for the slender body Neumann-to-Dirichlet map $\mc{L}_\epsilon$ about a general curved filament.

Given a slender filament $\Sigma_\epsilon$ as in section \ref{subsec:geom} with $\X(s)\in C^{2,\alpha}$ and slender body Neumann data $f(s)\in C^{0,\alpha}(\T)$, 
we begin by considering the solution $u$ to the slender body PDE \eqref{eq:SB_PDE} in $\Omega_\epsilon=\R^3\backslash\overline{\Sigma_\epsilon}$.
It will be useful to recall some results from the $L^2$-based solution theory for the slender body PDE\footnote{The solution theory in \cite{closed_loop} is developed in the context of the Stokes equations; the analogous theory for the Laplace setting follows by nearly identical and somewhat simpler arguments.} developed in \cite{closed_loop}. In particular, we recall the following bound (\cite[Theorem 1.2]{closed_loop}):
\begin{lemma}[$L^2$-based SB PDE solution theory]\label{lem:L2_SBPDE}
Given a slender body $\Sigma_\epsilon$ as in section \ref{subsec:geom} with $\X(s)\in C^2$ and given $f\in L^2(\T)$, the solution $u$ to \eqref{eq:SB_PDE} belongs to 
\begin{align*}
D^{1,2}(\Omega_\epsilon):= \{u\in L^6(\Omega_\epsilon)\,:\, \nabla u\in L^2(\Omega_\epsilon) \}
\end{align*} 
and satisfies the estimate 
\begin{equation}\label{eq:L2est_SB_PDE}
\norm{u}_{D^{1,2}(\Omega_\epsilon)}\equiv \norm{\nabla u}_{L^2(\Omega_\epsilon)} \le c(\kappa_*)\abs{\log\epsilon}^{1/2}\norm{f}_{L^2(\T)}\,.
\end{equation}
\end{lemma}

For $r_*=r_*(c_\Gamma,\kappa_*)<\frac{1}{2\kappa_*}$ as in section \ref{subsec:geom}, we consider the curved annular region 
\begin{equation}\label{eq:O_region}
\mc{O}_{r_*}= \big\{ \bx\in \Omega_\epsilon \;:\; \epsilon<\text{ dist}(\bx,\Gamma_0) < r_*\big\}
\end{equation}
about the filament $\Sigma_\epsilon$. Note that since $u$ belongs to $L^6(\Omega_\epsilon)$, the estimate \eqref{eq:L2est_SB_PDE} implies that $u$ belongs to $H^1(\mc{O}_{r_*})$ and satisfies 
\begin{equation}\label{eq:L2est_SB_PDE_2}
\norm{u}_{H^1(\mc{O}_{r_*})}\le c(|\mc{O}_{r_*}|)\norm{u}_{D^{1,2}(\mc{O}_{r_*})} \le c(\kappa_*,|\mc{O}_{r_*}|)\abs{\log\epsilon}^{1/2}\norm{f}_{L^2(\T)}\,.
\end{equation} 
\textcolor{black}{
Within $\mc{O}_{r_*}$, we consider the weak form of the slender body PDE \eqref{eq:SB_PDE}. Given $\phi\in H^1(\mc{O}_{r_*})$ with $\phi\big|_{\Gamma_\epsilon}=\phi(s)$ and $\phi\big|_{\p\mc{O}_{r_*}\backslash\Gamma_\epsilon}=0$, the weak solution $u\in H^1(\mc{O}_{r_*})$ to \eqref{eq:SB_PDE} satisfies
\begin{equation}\label{eq:weaku}
 \int_{\mc{O}_{r_*}}\nabla u\cdot\nabla\phi\,d\bx = \int_\T f(s)\phi(s)\,ds\,.
\end{equation}
For $\bx\in \mc{O}_{r_*}$, we may define a $C^{1,\beta}$ change of variables $\Phi(\bx)$ mapping the tube to a straight cylinder. In particular, for $\bx=\X(s)+r\be_r(s,\theta)$ about the curved centerline $\X(s)$, we may take $\Phi(\bx) = s\be_z+r\be_r(\theta)=s\be_z+r\cos\theta\be_x+r\sin\theta\be_y$ where $(\be_z,\be_x,\be_y)$ are Cartesian basis vectors about a straight centerline. 
 We may calculate 
\begin{align*}
\nabla\Phi^{-1} &= (1-r\wh\kappa)\be_{\rm t}\otimes\be_z+r\kappa_3\be_\theta(s,\theta)\otimes\be_z+\be_\theta(s,\theta)\otimes\be_\theta(\theta) + \be_r(s,\theta)\otimes\be_r(\theta) \,,
\end{align*}
from which we may calculate
\begin{align*}
\nabla\Phi(\nabla\Phi)^{\rm T}\circ \Phi^{-1} &= \frac{1+r^2\kappa_3^2}{(1-r\wh\kappa)^2}\be_z\otimes\be_z - \frac{r\kappa_3}{1-r\wh\kappa}\big(\be_\theta(\theta)\otimes\be_z+\be_z\otimes\be_\theta(\theta)\big) \\
&\qquad  +\be_\theta(\theta)\otimes\be_\theta(\theta) +\be_r(\theta)\otimes\be_r(\theta)\,.
\end{align*}
Defining ${\bm A}(s,r,\theta)=\frac{1}{\abs{\det\nabla\Phi}}\nabla\Phi(\nabla\Phi)^{\rm T}\circ \Phi^{-1}$, we note that within the region $\mc{O}_{r_*}$, we have $\frac{1}{c(\kappa_{*,\alpha})}<\norm{{\bm A}}_{C^{0,\alpha}(\mc{O}_{r_*})}\le c(\kappa_{*,\alpha})$ and $\bm{A}$ depends smoothly on $\theta$. 
}

\textcolor{black}{
Letting $\psi = \phi\circ\Phi^{-1}$ and noting that for $\by\in \Phi(\Gamma_\epsilon)$, $\psi(\by)=\phi\big|_{\bx\in\Gamma_\epsilon}$ is independent of $\theta$, we may rewrite \eqref{eq:weaku} in straightened coordinates $\by=\Phi(\bx)$ as
\begin{equation}\label{eq:weaku_str}
\int_{\Phi(\mc{O}_{r_*})} \bm{A}\nabla u \cdot\nabla\psi\,d\by = \int_\T f(s) \psi(s)\,ds\,.
\end{equation}
We will rely on the formulation \eqref{eq:weaku_str} to prove Lemma \ref{lem:SB_PDE_holder}, but first introduce some auxiliary lemmas. 
}

\textcolor{black}{
Following the construction of \cite[Chapter 5]{giaquinta2013introduction}, we will build a Campanato-type function space designed for the slender body PDE. 
For $\rho< r_*/2$, we define the following annular region within $\Phi(\mc{O}_{r_*})$ about the straightened filament: 
\begin{align*}
A_\rho(s_0) &= \big\{ \by=s\be_z+r\be_r+\theta\be_\theta\in \Phi(\mc{O}_{r_*}) \; : \; s_0-\rho < s < s_0+\rho\,, \; \epsilon\le r < \rho+\epsilon\, , \; 0\le\theta<2\pi \big\}\,.
\end{align*}
Given a function $g(\by)$, $\by\in A_\rho(s_0)$, let $g_{s_0,\rho}$ denote the mean of $g$ over the annulus $A_\rho(s_0)$: 
\begin{align*}
g_{s_0,\rho} = \fint_{A_\rho(s_0)} g\,d\by\,.
\end{align*}
We then define the Campanato-type seminorm $[\cdot]_{\mc{A}^{2,\alpha}}$ and full norm $\norm{\cdot}_{\mc{A}^{2,\alpha}}$ by
\begin{equation}\label{eq:campanato}
\begin{aligned}
[g]_{\mc{A}^{2,\alpha}} &= \sup_{s_0\in \T,\,0<\rho<r_*/2}\rho^{-\alpha}\bigg(\fint_{A_\rho(s_0)}\abs{g - g_{s_0,\rho}}^2\,d\bx\bigg)^{1/2}\,, \\
\norm{g}_{\mc{A}^{2,\alpha}} &= \norm{g}_{L^2(\Phi(\mc{O}_{r_*}))} + [g]_{\mc{A}^{2,\alpha}}\,.
\end{aligned}
\end{equation}
}

\textcolor{black}{
We note the following relationship between the space $\mc{A}^{2,\alpha}$ and the H\"older space $C^{0,\alpha}$ in the annular region $\Phi(\mc{O}_{r_*})$ about the straightened filament. 
\begin{lemma}[Campanato-type norm bounds]\label{lem:camp}
Given a function $g$ in $\Phi(\mc{O}_{r_*})$ satisfying $g\big|_{\mc{C}_\epsilon}=g(s)$, a function of arclength only, we have that 
\begin{equation}\label{eq:Calpha_Aalpha}
\norm{g}_{C^{0,\alpha}(\T)}\le c\,\norm{g}_{\mc{A}^{2,\alpha}} \le c\,\norm{g}_{C^{0,\alpha}(\Phi(\mc{O}_{r_*}))}\,.
\end{equation}
Note that the $C^{0,\alpha}$ norm on the left hand side is over the filament centerline ($s\in \T$), while the $C^{0,\alpha}$ norm on the right hand side is over the entire annular region $\Phi(\mc{O}_{r_*})$. 
\end{lemma}
}

\begin{proof}
\textcolor{black}{
The second inequality is immediate: let $g\in C^{0,\alpha}(\Phi(\mc{O}_{r_*}))$ and consider $\by,\by'\in A_\rho(s_0)$. We have
\begin{align*}
\abs{g(\by)-g(\by')}\le c\,\abs{g}_{C^{0,\alpha}}\rho^\alpha
\end{align*}
for any choice of $\by'\in A_\rho(s_0)$. In particular, we have
\begin{align*}
\abs{g(\by)-g_{s_0,\rho}}\le c\,\abs{g}_{C^{0,\alpha}}\rho^\alpha\,,
\end{align*}
and thus
\begin{align*}
\rho^{-\alpha}\bigg(\fint_{A_\rho(s_0)}\abs{g - g_{s_0,\rho}}^2\,d\by\bigg)^{1/2}
&\le c\,\abs{g}_{C^{0,\alpha}}\,.
\end{align*}
}

\textcolor{black}{
Now consider $g\in \mc{A}^{2,\alpha}$ and take $0<\rho<R\le \frac{r_*}{2}$. Note that 
\begin{align*}
\abs{g_{s_0,R}-g_{s_0,\rho}}^2\le 2\abs{g(\by)-g_{s_0,R}}^2+2\abs{g(\by)-g_{s_0,\rho}}^2\,.
\end{align*}
Integrating over the smaller annulus $A_\rho(s_0)$ and noting that, for $\rho$ small, the volume of the annulus scales as $\rho^2$, we have 
\begin{align*}
\abs{g_{s_0,R}-g_{s_0,\rho}}^2 &\le \frac{c}{\rho^2}\bigg(\int_{A_R(s_0)}\abs{g-g_{s_0,R}}^2\,d\bx+ \int_{A_\rho(s_0)}\abs{g-g_{s_0,\rho}}^2\,d\bx \bigg)\\
&\le \frac{c}{\rho^2}\big(R^{2\alpha+2}+\rho^{2\alpha+2} \big)[g]_{\mc{A}^{2,\alpha}}^2 
\le c\frac{R^{2\alpha+2}}{\rho^2}[g]_{\mc{A}^{2,\alpha}}^2\,.
\end{align*}
Defining $R_k=\frac{R}{2^k}$, we then note that 
\begin{align*}
\abs{g_{s_0,R_k}-g_{s_0,R_{k+1}}} \le c\frac{R^\alpha}{2^{\alpha k}}[g]_{\mc{A}^{2,\alpha}}\,;
\end{align*}
in particular, the sequence of annuli forms a Cauchy sequence limiting to the $s=s_0$ cross section of the filament surface $\mc{C}_\epsilon$. Using that $g\big|_{\mc{C}_\epsilon}=g(s)$, by Lebesgue differentiation, the limit of the above sequence as $k\to\infty$ is $g\big|_{\mc{C}_\epsilon}(s_0)$, with a rate
\begin{equation}\label{eq:uR_convergence}
\abs{g_{s_0,R}-g\big|_{\mc{C}_\epsilon}(s_0)}\le c\,R^\alpha[g]_{\mc{A}^{2,\alpha}}\,.
\end{equation}
}

\textcolor{black}{
Now, for $\by,\by'\in\Phi(\mc{O}_{r_*})$, write $\by=s\be_z+r\be_r(\theta)$ and $\by'=s'\be_z+r'\be_r(\theta')$. Let $R=\abs{s-s'}$ be the centerline distance between the cross sections corresponding to $\by$ and $\by'$. For $R<\frac{r_*}{4}$, along the slender cylinder surface we may estimate  
\begin{equation}\label{eq:hyhyprime}
\abs{g(s_0)-g(s_0')} \le \abs{g(s_0)-g_{s_0,2R}} + \abs{g_{s_0,2R}-g_{s_0',2R}} + \abs{g(s_0')-g_{s_0',2R}}\,.
\end{equation} 
By the definition of $R$, we have $\abs{A_{2R}(s_0)\cap A_{2R}(s_0')}\ge \abs{A_R(s_0)}\ge c R^2$, since the volume of the annulus $A_R$ scales like $R^2$ when $R$ is small. Thus, integrating the constant $\abs{g_{s_0,2R}-g_{s_0',2R}}$ over $\abs{A_{2R}(s_0)\cap A_{2R}(s_0')}$, we have 
\begin{align*}
\abs{g_{s_0,2R}-g_{s_0',2R}} &\le \frac{1}{\abs{A_{2R}(s_0)\cap A_{2R}(s_0')}}\bigg(\fint_{A_{2R}(s_0)}\abs{g(\by)-g_{s_0,2R}}\,d\by + \fint_{A_{2R}(s_0')}\abs{g(\by)-g_{s_0',2R}}\,d\by \bigg) \\
&\le \frac{c}{R^2}\bigg(R^2\bigg(\fint_{A_{2R}(s_0)}\abs{g(\by)-g_{s_0,2R}}^2\,d\by\bigg)^{1/2} + R^2\bigg(\fint_{A_{2R}(s_0')}\abs{g(\by)-g_{s_0',2R}}^2\,d\by\bigg)^{1/2} \bigg) \\
&\le c\,R^{\alpha}[g]_{\mc{A}^{2,\alpha}} \,.
\end{align*}
Combining this estimate with \eqref{eq:uR_convergence} in \eqref{eq:hyhyprime}, we obtain 
\begin{align*}
\abs{g(s_0)-g(s_0')} &\le c\,\abs{s_0-s_0'}^{\alpha}[g]_{\mc{A}^{2,\alpha}}\,. 
\end{align*}
}
\end{proof}

\textcolor{black}{
In addition to the Campanato-type characterization of H\"older continuity about $\mc{C}_\epsilon$, we will require a series of intermediary results regarding the following special version of slender body PDE. Let $R>0$ and $s_0\in\T$ and consider the annular region $A_R(s_0)$ about $\mc{C}_\epsilon$. Suppose $\bm{B}(\theta)$ is a smooth, matrix-valued function supported in $A_R(s_0)$ depending only on the angle $\theta$ and satisfying $\frac{1}{c}\le \abs{\p_\theta^k{\bm B}}\le c$ for some $c>0$. Let $f_0$ be a known constant and $d$ be a given function on $\p A_R(s_0)\backslash \mc{C}_\epsilon$. 
We consider the weak solution $h\in H^1(A_R(s_0))$ to 
\begin{equation}\label{eq:v_pde}
\begin{aligned}
-\div (\bm{B}(\theta) \nabla h) &= 0 \qquad\quad \text{in }A_R(s_0)\\
\int_0^{2\pi}\bm{B}(\theta)\nabla h\cdot\bm{n}_y\, \epsilon \, d\theta &= f_0 \qquad\;\, \text{on }\mc{C}_\epsilon \\
h\big|_{\Gamma_\epsilon} &= h(s)\,, \quad \text{unknown but independent of }\theta \\
h &= d \qquad\quad \text{on } \p A_R(s_0)\backslash \mc{C}_\epsilon\,.
\end{aligned}
\end{equation}
We have that $h$ satisfies the following proposition.
\begin{proposition}[Higher regularity for $h$]\label{prop:high_reg}
Let $0<\rho<R$ and consider the annulus $A_\rho(s_0)\subset A_R(s_0)$. The solution $h$ to \eqref{eq:v_pde} satisfies
\begin{equation}
\norm{h}_{H^4(A_\rho(s_0))} \le c(\epsilon,\kappa_*)\norm{h}_{H^1(A_R(s_0))}\,.
\end{equation} 
\end{proposition}
}

\begin{proof}
\textcolor{black}{
We consider the weak form of \eqref{eq:v_pde}. For $g\in H^1(A_R(s_0))$ with $g\big|_{\mc{C}_\epsilon}=g(s)$ and $g\big|_{\p A_R(s_0)\backslash \mc{C}_\epsilon}=0$, we have that a weak solution $h$ to \eqref{eq:v_pde} satisfies
\begin{equation}\label{eq:v_weakform}
\int_{A_R(s_0)}{\bm B}\nabla h\cdot\nabla g \,d\by = f_0\int_{\abs{s-s_0}\le R} g(s)\,ds\,.
\end{equation}
For some $R'<R$, define a smooth cutoff function $\eta$ supported up to the filament surface $\mc{C}_\epsilon$ and satisfying
\begin{equation}\label{eq:eta_R}
\eta(\bx) = \begin{cases}
1\,, & \bx\in A_{R'}(s_0)  \\
0\,, & \bx\not\in A_R(s_0)\,.
\end{cases}
\end{equation}
Within $A_R(s_0)$, we consider the derivatives $\p_s$, $\frac{1}{r}\p_\theta$ in directions tangential to $\mc{C}_\epsilon$. We will take $\p_k^2(\eta^2h)$, $\p_k=\p_s,\frac{1}{r}\p_\theta$ as our test function in \eqref{eq:v_weakform} (this does not \emph{a priori} belong to $H^1$ but may be justified using finite differences). 
Then for $\p_k=\p_s,\frac{1}{r}\p_\theta$, we have
\begin{align*}
\int_{A_R(s_0)} \bm{B}\nabla h\cdot\nabla(\p_k^2(\eta^2 h))\,d\bx = f_0\int_{\abs{s-s_0}< R} \p_k^2(\eta^2h)\,ds = 0 \,.
\end{align*}
Note that, in straight cylindrical coordinates, $\p_s$ commutes with $\nabla$ while $\abs{[\nabla,\frac{1}{r}\p_\theta]g}\le \frac{1}{\epsilon}\abs{\nabla g}$. 
Since $\bm{B}(\theta)$ is bounded below, after commuting and integrating by parts, we obtain
\begin{align*}
\int_{A_R(s_0)} \eta^2\abs{\nabla\p_k h}^2\,d\by &\le 
c(\epsilon)\int_{A_R(s_0)} \bigg( \eta\big(\abs{\bm{B}} + \abs{\p_k\bm{B}}\big)\abs{\nabla\p_k h} \abs{\nabla h}\\
&\qquad + \abs{\bm{B}}\abs{\nabla h}\big(\abs{\nabla\p_k^2\eta^2}\abs{h}+\abs{\nabla^2\eta^2}\abs{\nabla h} \big)\bigg)\,d\by   \,.
\end{align*}
Using Young's inequality on the right hand side, over the smaller set $A_{R'}(s_0)$, we may bound
\begin{equation}\label{eq:tan_reg}
\norm{\nabla\p_k h}_{L^2(A_{R'}(s_0))} 
\le c(\epsilon)\norm{h}_{H^1(A_R(s_0))}\,, \qquad \p_k=\p_s,\textstyle \frac{1}{r}\p_\theta\,.
\end{equation}
We then write out
\begin{equation}\label{eq:curved_laplace}
\div(\bm{B}(\theta)\nabla h) =  \frac{1}{r}\frac{\p}{\p r}\bigg(r (\bm{B}\nabla h)\cdot\be_r\bigg) + \frac{1}{r}\frac{\p}{\p\theta}\bigg((\bm{B}\nabla h)\cdot\be_\theta \bigg)+\frac{\p}{\p s}\bigg((\bm{B}\nabla h)\cdot\be_z \bigg)\,,
\end{equation}
where $\nabla h = \p_r h \be_r + \frac{1}{r}\p_\theta h\be_\theta + \p_s h \be_z$ and $\abs{\frac{1}{r}\p_\theta\bm{B}(\theta)}\le \frac{c}{\epsilon}$. We may then combine the tangential regularity estimate \eqref{eq:tan_reg} with the fact that $h$ satisfies $\div(\bm{B}(\theta)\nabla h)=0$ in $A_{R'}$ to bound second derivatives in directions normal to the filament surface: 
\begin{align*}
 \norm{\p_{rr}h}_{L^2(A_{R'})} 
 &\le c(\epsilon)\bigg(\norm{\nabla h}_{L^2(A_R)}+ \norm{\nabla\p_s h}_{L^2(A_{R'})}
 +\textstyle \norm{\nabla(\frac{1}{r}\p_\theta h)}_{L^2(A_{R'})} \bigg) \,.
 \end{align*} 
In total, we obtain an $H^2$ bound for $h$ over the region $A_{R'}(s_0)$:
\begin{equation}\label{eq:H2bd}
\norm{h}_{H^2(A_{R'}(s_0))} 
\le c(\epsilon)\norm{h}_{H^1(A_R(s_0))}\,.
\end{equation}
}

\textcolor{black}{
We may iterate this procedure over slightly smaller annuli. Within $A_{R'}(s_0)$, using that $\p_s$ commutes with $\div$ and $\bm{B}(\theta)\nabla$, we have that $\p_s h$ (weakly) satisfies
\begin{equation}\label{eq:ps_v_eqn}
\begin{aligned}
\div(\bm{B}(\theta)\nabla \p_s h) &= 0 \quad \text{in }A_{R'}(s_0) \\
\int_0^{2\pi}\bm{B}(\theta)\nabla (\p_s h)\cdot\bm{n}_y\,\epsilon\,d\theta &=0 \quad \text{on }\mc{C}_\epsilon \\
\p_sh\big|_{\mc{C}_\epsilon} &= \p_sh(s)\,, \quad \text{independent of }\theta \,.
\end{aligned}
\end{equation}
Furthermore, we have that $\frac{1}{\epsilon}\p_\theta h$ satisfies
\begin{equation}\label{eq:ptheta_v_eqn}
\begin{aligned}
\textstyle \div(\bm{B}(\theta)\nabla \frac{1}{r}\p_\theta h) &= \textstyle [\div(\bm{B}(\theta)\nabla),\frac{1}{r}\p_\theta]h \quad \text{in }A_{R'}(s_0) \\
\textstyle \frac{1}{\epsilon}\p_\theta h\big|_{\mc{C}_\epsilon}&=0 \,.
\end{aligned}
\end{equation} 
Here we note that the right hand side commutator satisfies
\begin{equation}\label{eq:commutator}
\textstyle \norm{[\div(\bm{B}(\theta)\nabla),\frac{1}{r}\p_\theta]h}_{L^2(A_{R'}(s_0))}
\le c(\epsilon)\norm{h}_{H^2(A_{R'}(s_0))}\,.
\end{equation}
We consider the weak form of \eqref{eq:ps_v_eqn} and \eqref{eq:ptheta_v_eqn} as in \eqref{eq:v_weakform}, this time with test function $g=\p_\ell^2\eta_2^2(\p_k h)$, $\p_\ell,\p_k\in\{\p_s,\frac{1}{r}\p_\theta\}$ where $\eta_2$ is a smooth cutoff satisfying, for $R''<R'$,
\begin{align*}
\eta_2(\bx)= \begin{cases}
1 & \bx \in A_{R''}(s_0) \\
0 & \bx \not\in A_{R'}(s_0).
\end{cases}
\end{align*}
Again, this $g$ does not \emph{a priori} belong to $H^1$ but may be justified using finite differences. After a series of integration by parts we obtain a similar bound to \eqref{eq:tan_reg}:
\begin{equation}\label{eq:tan_reg2}
\norm{\nabla\p_\ell\p_k h}_{L^2(A_{R''}(s_0))} 
\le c(\epsilon)\bigg(\norm{h}_{H^1(A_R)}+\norm{h}_{H^2(A_{R'})} \bigg)\,.
\end{equation}
}

\textcolor{black}{
To obtain an $H^3$ estimate for the normal direction to the straightened slender body, we note that within $A_{R''}$, $h$ satisfies
\begin{align*}
\div(\bm{B}(\theta)\nabla \p_r h)  &= [\div(\bm{B}(\theta)\nabla),\p_r]h \,,
\end{align*}
where $[\div(\bm{B}(\theta)\nabla),\p_r]h$ also satisfies \eqref{eq:commutator}. Using \eqref{eq:curved_laplace}, we have that $\p_{rrr}h$ may be written entirely in terms of functions which belong to $L^2$, and may thus be bounded as 
\begin{align*}
\norm{\p_{rrr}h}_{L^2(A_{R''})} 
&\le c(\epsilon)\bigg(\textstyle \norm{\frac{1}{r}\p_\theta^2 \nabla h}_{L^2(A_{R''})}  
+\norm{\p_s^2 \nabla h}_{L^2(A_{R''})}  
+\norm{\frac{1}{r}\p_\theta\p_s \nabla h}_{L^2(A_{R''})}
+ \norm{h}_{H^2(A_{R'})}  \bigg)\,.
\end{align*}
In total, using \eqref{eq:H2bd}, we obtain the full $H^3$ bound 
\begin{align*}
\norm{h}_{H^3(A_{R''})}
&\le c(\epsilon)\norm{h}_{H^1(A_R)}\,.
\end{align*}
We may perform this entire procedure one more time to obtain, for $R'''<R''<R'<R$, the bound
\begin{align*}
\norm{h}_{H^4(A_{R'''}(s_0))}
&\le c(\epsilon)\norm{h}_{H^1(A_R(s_0))}\,.
\end{align*}
}
\end{proof}

\textcolor{black}{
Using Proposition \ref{prop:high_reg}, we may show that $h$ satisfies the following.
\begin{corollary}[Caccioppoli inequality]\label{cor:caccio}
For $0<R'<R$, the solution $h$ to \eqref{eq:v_pde} satisfies
\begin{equation}\label{eq:caccio}
\int_{A_{R'}(s_0)}\abs{\nabla \p_s h}^2\,d\by \le c(\epsilon)\frac{1}{(R-R')^2}\int_{A_R(s_0)}\abs{\p_s h - (\p_s h)_{s_0,R}}^2\,d\by\,.
\end{equation}
\end{corollary}
}

\begin{proof}
\textcolor{black}{
We have that $\p_s h$ satisfies \eqref{eq:ps_v_eqn} within $A_R(s_0)$ for some $R>0$. Furthermore, $\p_s h-\lambda$ also satisfies \eqref{eq:ps_v_eqn} for any constant $\lambda$ since $(\p_s h-\lambda)\big|_{\Gamma_\epsilon}$ is still independent of $\theta$. Let $\eta$ be a cutoff function as in \eqref{eq:eta_R}. Multiplying \eqref{eq:ps_v_eqn} by $\eta^2(\p_s h-\lambda)$ and integrating by parts, we obtain
\begin{align*}
0&=\int_{A_R(s_0)}\bm{B}(\theta)\nabla(\p_sh-\lambda)\cdot\nabla(\eta^2 (\p_sh-\lambda))\,d\by\\
 &= \int_{A_R(s_0)}\eta^2\bm{B}(\theta)\nabla(\p_sh-\lambda)\cdot\nabla(\p_sh-\lambda)\,d\by + 2\int_{A_R(s_0)}\eta\bm{B}(\theta)\nabla(\p_sh-\lambda)\cdot\nabla\eta \,(\p_sh-\lambda)\,d\by\,.
\end{align*}
Therefore, using that $\bm{B}$ is bounded below, we have
\begin{align*}
\int_{A_R(s_0)}\abs{\eta\nabla(\p_sh-\lambda)}^2\,d\by \le c\int_{A_R(s_0)}\abs{\nabla\eta}\abs{\bm{B}}\abs{\eta\nabla(\p_sh-\lambda)} \abs{\p_sh-\lambda}\,d\by\,.
\end{align*}
Using Young's inequality and that $\eta=1$ on $A_{R'}(s_0)$ while $\abs{\nabla\eta}\sim\frac{1}{R-R'}$, we have
\begin{align*}
\int_{A_{R'}(s_0)}\abs{\nabla(\p_sh-\lambda)}^2\,d\by \le \frac{c}{(R-R')^2}\int_{A_R(s_0)}\abs{\p_sh-\lambda}^2\,d\by\,.
\end{align*}
Taking $\lambda=(\p_s h)_{s_0,R}$, we obtain \eqref{eq:caccio}.
} 
\end{proof}

\textcolor{black}{
Using Proposition \ref{prop:high_reg} and Corollary \ref{cor:caccio}, we may prove the following rate-dependent estimate for the solution $h$ to \eqref{eq:v_pde}.
\begin{lemma}\label{lem:v_est}
For $0<\rho<R/2$, the solution $h$ to \eqref{eq:v_pde} satisfies 
\begin{equation}\label{eq:v_est}
\fint_{A_\rho(s_0)} \abs{\p_sh - (\p_s h)_{s_0,\rho}}^2 \,d\by 
\le c(\epsilon)\bigg(\frac{\rho}{R}\bigg)^2\fint_{A_R(s_0)}\abs{\p_s h - (\p_s h)_{s_0,R}}^2\,d\by\,.
\end{equation}
\end{lemma}
}

\begin{proof}
\textcolor{black}{
Using a Poincar\'e inequality, Sobolev embedding, and Corollary \ref{cor:caccio}, we have that $h$ satisfies 
\begin{align*}
\fint_{A_\rho(s_0)} \abs{\p_s h - (\p_s h)_{s_0,\rho}}^2 \,d\by &\le c(\epsilon)\,\rho^2\fint_{A_\rho(s_0)} \abs{\nabla\p_s h}^2 \,d\by 
\le c(\epsilon)\,\rho^2\sup_{\by\in A_\rho(s_0)} \abs{\nabla\p_s h}^2 \\ 
&\le c(\epsilon)\,\rho^2\sup_{\by\in A_{R/2}(s_0)} \abs{\nabla\p_s h}^2
\le c(\epsilon)\,\rho^2\norm{\nabla\p_s h}_{H^2(A_{R/2}(s_0))}^2\\
&
\le c(\epsilon)\,\rho^2\fint_{A_{3R/4}(s_0)}\abs{\nabla\p_s h}^2\,d\by \\
&\le c(\epsilon)\bigg(\frac{\rho}{R}\bigg)^2\fint_{A_R(s_0)}\abs{\p_s h - (\p_s h)_{s_0,R}}^2\,d\by\,.
\end{align*}
}
\end{proof}

\textcolor{black}{
Given Lemmas \ref{lem:camp} and \ref{lem:v_est}, we may proceed to the proof of Lemma \ref{lem:SB_PDE_holder}.
}
\begin{proof}[Proof of Lemma \ref{lem:SB_PDE_holder}]
\textcolor{black}{
We consider the weak solution $u$ to the slender body PDE, which satisfies \eqref{eq:weaku_str} in the straightened region $\Phi(\mc{O}_{r_*})$. For $0<R\le\frac{r_*}{2}$ and $s_0\in \T$, we consider the annular region $A_R(s_0)$ within $\Phi(\mc{O}_{r_*})$. 
}

\textcolor{black}{
Within $A_R(s_0)$, we write $u=h+q$ where $h$ satisfies the following version of the slender body PDE with frozen coefficients in $s$ and $r$. Given the matrix $\bm{A}(s,r,\theta)$, let $\bm{A}_0(\theta) = \bm{A}(s_0,\epsilon,\theta)$. We consider $h$ satisfying
\begin{equation}\label{eq:real_v_eqn} 
\begin{aligned}
-\div (\bm{A}_0 \nabla h) &= 0 \qquad\quad \text{in }A_R(s_0)\\
\int_0^{2\pi}\bm{A}_0\nabla h\cdot\bm{n}_y\, \epsilon d\theta &= f(s_0) \quad\; \text{on }\Phi(\Gamma_\epsilon) \\
h\big|_{\Phi(\Gamma_\epsilon)} &= h(s)\,, \quad \text{unknown but independent of }\theta \\
h &= u \qquad\quad \text{on } \p A_R(s_0)\backslash \Phi(\Gamma_\epsilon)\,.
\end{aligned}
\end{equation}
Since $\bm{A}_0$ depends smoothly on $\theta$, we may use Lemma \ref{lem:v_est} to obtain a bound for the oscillation of $\p_sh$ for any $\rho<R/2$:
\begin{equation}\label{eq:h_osc}
\int_{A_\rho(s_0)}\abs{\p_s h - (\p_s h)_{s_0,\rho}}^2 d\by \le c(\epsilon,\kappa_*)\bigg(\frac{\rho}{R} \bigg)^4\int_{A_R(s_0)}\abs{\p_s h - (\p_s h)_{s_0,R}}^2 d\by\,.
\end{equation}
Furthermore, rewriting \eqref{eq:weaku_str} as
\begin{align*}
\int_{A_{R}(s_0)}(\bm{A}-\bm{A}_0)\nabla u \cdot\nabla\psi\,d\by + \int_{A_{R}(s_0)}\bm{A}_0(\nabla h+\nabla q) \cdot\nabla\psi\,d\by &= \int_{\abs{s-s_0}\le R}f(s)\psi(s)\,ds\,,
\end{align*}
we may use the weak form of \eqref{eq:real_v_eqn} to obtain an equation for $\nabla q$:
\begin{align*}
\int_{A_{R}(s_0)}\bm{A}_0\nabla q \cdot\nabla\psi\,d\by 
&= -\int_{A_{R}(s_0)}(\bm{A}-\bm{A}_0)\nabla u \cdot\nabla\psi\,d\by    
+\int_{\abs{s-s_0}\le R}\big(f(s)-f(s_0)\big)\psi(s)\,ds\,.
\end{align*}  
}

\textcolor{black}{
Since $q=0$ on $\p A_{R}(s_0)\backslash \Phi(\Gamma_\epsilon)$ and $q=q(s)$ on $\Phi(\Gamma_\epsilon)$ by construction, using that the matrix $\bm{A}_0$ is bounded below, we may estimate 
\begin{equation}\label{eq:weq1}
\begin{aligned}
\int_{A_{R}}\abs{\nabla q}^2\,d\by 
&\le c(\kappa_*)\bigg(R^\alpha\norm{\bm{A}}_{C^{0,\alpha}(A_{R})}\norm{\nabla u}_{L^2(A_{R})} \norm{\nabla q}_{L^2(A_{R})} \\
&\qquad + R^{1/2+\alpha}\norm{f}_{C^{0,\alpha}(\T)}\norm{q\big|_{\mc{C}_\epsilon}}_{L^2(\abs{s-s_0}<R)}\bigg)\,.
\end{aligned}
\end{equation}
Note that, by a scaling argument, $q|_{\mc{C}_\epsilon}$ satisfies the following trace inequality within the annulus $A_{R}(s_0)$:
\begin{align*}
\int_{\abs{s-s_0}<R}\abs{q\big|_{\mc{C}_\epsilon}}^2\,ds 
&= \frac{1}{2\pi}\int_0^{2\pi}\int_{\abs{s-s_0}<R}\abs{q\big|_{\mc{C}_\epsilon}}^2\,ds d\theta 
\le c(\epsilon)\,R\int_{A_R(s_0)}\bigg(\abs{\nabla q}^2 + R^{-2}\abs{q}^2 \bigg)\,d\by \\
&\le c(\epsilon)\,R\int_{A_R(s_0)}\abs{\nabla q}^2 \,d\by\,.
\end{align*}
Using Young's inequality and the definition of $\bm{A}$, the estimate \eqref{eq:weq1} yields 
\begin{equation}\label{eq:weq2}
\int_{A_{R}}\abs{\nabla q}^2\,d\by 
\le c(\epsilon,\kappa_{*,\alpha})\bigg(R^{2\alpha} \int_{A_R}\abs{\nabla u}^2\,d\by+ R^{2+2\alpha}\norm{f}_{C^{0,\alpha}(\T)}^2\bigg)\,.
\end{equation}
}

\textcolor{black}{
Furthermore, by Proposition \ref{prop:high_reg}, we have that, for $\rho<R/2$, the function $h$ satisfies
\begin{align*}
\int_{A_\rho(s_0)}\abs{\nabla h}^2\,d\by \le c\,\rho^2\sup_{\by\in A_\rho(s_0)}\abs{\nabla h}^2 
\le c\,\rho^2\sup_{\by\in A_{R/2}(s_0)}\abs{\nabla h}^2 \\
\le c(\epsilon,R)\rho^2\norm{\nabla h}_{H^2(A_{R/2}(s_0))}^2\le c(\epsilon)\bigg(\frac{\rho}{R}\bigg)^2\norm{\nabla h}_{L^2(A_{R}(s_0))}^2\,,
\end{align*}
by scaling. From this we may obtain 
\begin{equation}\label{eq:u_with_R}
\begin{aligned}
\int_{A_\rho(s_0)}\abs{\nabla u}^2\,d\by &\le c(\epsilon)\bigg(\frac{\rho}{R}\bigg)^2\int_{A_R(s_0)}\abs{\nabla u}^2\,d\by + c(\epsilon)\int_{A_R(s_0)}\norm{\nabla q}^2\,d\by\\
&\le c(\epsilon,\kappa_{*,\alpha})\bigg(\bigg(\frac{\rho}{R}\bigg)^2 + R^{2\alpha}\bigg) \int_{A_R(s_0)}\abs{\nabla u}^2\,d\by+ R^{2+2\alpha}\,c(\epsilon,\kappa_{*,\alpha})\norm{f}_{C^{0,\alpha}(\T)}^2\,,
\end{aligned}
\end{equation}
where we have used the bound \eqref{eq:weq2}.
}

\textcolor{black}{
We now state a useful result from \cite[Lemma 5.13]{giaquinta2013introduction}. 
\begin{proposition}[Lemma 5.13 from \cite{giaquinta2013introduction}]\label{prop:giaq}
Given a nondecreasing function $\Psi:\R^+\to\R^+$ satisfying
\begin{align*}
\Psi(\rho) \le c_1\bigg(\bigg(\frac{\rho}{R}\bigg)^m+\omega\bigg)\Psi(R)+c_2R^\nu\,, \quad m>\nu>0 
\end{align*}
for all $0<\rho\le R\le \frac{r_*}{2}$, for $\omega$ sufficiently small (depending on $c_1,m,\nu$), we in fact have
\begin{equation}\label{eq:giaquinta_est}
\Psi(\rho) \le c_3\bigg(\frac{\Psi(R)}{R^\nu}+c_2\bigg)\rho^\nu 
\end{equation}
for all $0\le \rho\le R\le \frac{r_*}{2}$.
\end{proposition}
}

\textcolor{black}{
Applying Proposition \ref{prop:giaq} to \eqref{eq:u_with_R} with $\Psi(\rho)=\int_{A_\rho(s_0)}\abs{\nabla u}^2\,d\by$, $m=2$, and $\nu=2-\delta$ for small $\delta>0$, we obtain 
\begin{equation}\label{eq:new_nablau_est}
\int_{A_\rho(s_0)}\abs{\nabla u}^2\,d\by 
\le c(\epsilon,\kappa_{*,\alpha})\bigg(R^{-(2-\delta)}\int_{A_R(s_0)}\abs{\nabla u}^2\,d\by+ R^{2\alpha+\delta}\norm{f}_{C^{0,\alpha}(\T)}^2\bigg)\rho^{2-\delta}
\end{equation}
for any $0\le \rho\le R\le\frac{r_*}{2}$.
}

\textcolor{black}{
Using \eqref{eq:h_osc} and \eqref{eq:weq2}, for $0<\rho\le\wt \rho \le R\le \frac{r_*}{2}$ we may then estimate
\begin{equation}\label{eq:main_comp}
\begin{aligned}
\int_{A_\rho(s_0)} &\abs{\p_s u - (\p_s u)_{s_0,\rho}}^2\,d\by \le 
2\int_{A_\rho(s_0)} \abs{\p_s h - (\p_s h)_{s_0,\rho}}^2 d\by + 2\int_{A_\rho(s_0)}\abs{\p_s q - (\p_s q)_{s_0,\rho}}^2\,d\by \\
& \le c(\epsilon,\kappa_*)\bigg(\frac{\rho}{\wt \rho} \bigg)^4 \int_{A_{\wt \rho}(s_0)}\abs{\p_s h - (\p_s h)_{s_0,\wt \rho}}^2\,d\by + 4\int_{A_\rho(s_0)}\abs{\nabla q}^2\,d\by \\
& \le c(\epsilon,\kappa_*) \bigg(\frac{\rho}{\wt \rho} \bigg)^4 \int_{A_{\wt\rho}(s_0)}\abs{\p_s u - (\p_s u)_{s_0,\wt\rho}}^2\,d\by + c(\epsilon,\kappa_*)\int_{A_{\wt \rho}(s_0)}\abs{\nabla q}^2\,d\by \\
& \le c(\epsilon,\kappa_*) \bigg(\frac{\rho}{\wt \rho} \bigg)^4 \int_{A_{\wt\rho}(s_0)}\abs{\p_s u - (\p_s u)_{s_0,\wt\rho}}^2\,d\by \\
&\qquad + c(\epsilon,\kappa_{*,\alpha})\bigg(\wt\rho^{2\alpha}\int_{A_{\wt\rho}(s_0)}\abs{\nabla u}^2\,d\by + \wt\rho^{2+2\alpha}\norm{f}_{C^{0,\alpha}(\T)}^2\bigg) \\
&\le  c(\epsilon,\kappa_*) \bigg(\frac{\rho}{\wt\rho} \bigg)^4 \int_{A_{\wt\rho}(s_0)}\abs{\p_s u - (\p_s u)_{s_0,\wt\rho}}^2\,d\by 
 + c(\epsilon,\kappa_{*,\alpha})\,\wt\rho^{2+2\alpha-\delta}\,\wt C  \,,
\end{aligned}
\end{equation}
where $\wt C = R^{-(2-\delta)}\int_{A_{R}(s_0)}\abs{\nabla u}^2\,d\by+ R^\delta\norm{f}_{C^{0,\alpha}(\T)}^2$. 
}

\textcolor{black}{
By Proposition \ref{prop:giaq} with $\Psi(\rho)=\int_{A_\rho(s_0)} \abs{\p_s u - (\p_s u)_{s_0,\rho}}^2\,d\by$, $m=4$, and $\nu=2+2\alpha-\delta$, we have 
\begin{align*}
\int_{A_\rho(s_0)} \abs{\p_s u - (\p_s u)_{s_0,\rho}}^2\,d\by&\le \rho^{2+2\alpha-\delta}\,c(\epsilon,\kappa_{*,\alpha})\bigg(
R^{-(2+2\alpha-\delta)}\int_{A_{R}(s_0)}\abs{\p_s u - (\p_s u)_{s_0,R}}^2\,d\by + \wt C \bigg) 
\end{align*}
for any $0\le\rho\le R\le\frac{r_*}{2}$.
By Lemma \ref{lem:camp}, we thus have $\p_s u\big|_{\Phi(\Gamma_\epsilon)\cap \p A_R(s_0)}\in C^{0,\alpha-\frac{\delta}{2}}$ for small $\delta>0$. Since $u\big|_{\Phi(\Gamma_\epsilon)}$ is independent of $\theta$, we also have $\frac{1}{\epsilon}\p_\theta u\big|_{\Phi(\Gamma_\epsilon)\cap \p A_R(s_0)}\in C^{0,\alpha-\frac{\delta}{2}}$. Thus by classical elliptic regularity theory we have $u\in C^{1,\alpha-\frac{\delta}{2}}(A_R(s_0))$ throughout the annulus $A_R(s_0)$ with 
\begin{equation}\label{eq:uholder}
\norm{u}_{C^{1,\alpha-\frac{\delta}{2}}(A_R(s_0))} \le c(\epsilon,\kappa_{*,\alpha})\big(\norm{u}_{H^1(\mc{O}_{r_*})} +\norm{f}_{C^{0,\alpha}(\T)} \big)\,.
\end{equation}
In particular, $\nabla u\in L^\infty(A_R(s_0))$, and we may estimate
\begin{align*}
\int_{A_R(s_0)}\abs{\nabla u}^2\,d\by &\le c\,R^2\sup_{\by\in A_R(s_0)}\abs{\nabla u}^2\,.
\end{align*}
We may then improve the estimate \eqref{eq:main_comp} to 
\begin{equation}\label{eq:main_comp2}
\begin{aligned}
&\int_{A_\rho(s_0)} \abs{\p_s u - (\p_s u)_{s_0,\rho}}^2\,d\by \\
&\qquad \le  c(\epsilon,\kappa_*) \bigg(\frac{\rho}{R} \bigg)^4 \int_{A_{R}(s_0)}\abs{\p_s u - (\p_s u)_{s_0,R}}^2\,d\by 
 + c(\epsilon,\kappa_{*,\alpha})\,R^{2+2\alpha}\,\wt C_2  \,,
\end{aligned}
\end{equation}
where $\wt C_2=\sup_{\by\in A_R(s_0)}\abs{\nabla u}^2 + \norm{f}_{C^{0,\alpha}}^2$. Again applying Proposition \ref{prop:giaq}, we have
\begin{equation}\label{eq:final_AR_est}
\fint_{A_\rho(s_0)} \abs{\p_s u - (\p_s u)_{s_0,\rho}}^2\,d\by\le \rho^{2\alpha}\,c(\epsilon,\kappa_{*,\alpha})\bigg(
R^{-2\alpha}\fint_{A_R(s_0)}\abs{\p_s u - (\p_s u)_{s_0,R}}^2\,d\by + \wt C_2 \bigg) \,.
\end{equation}
Since we may cover the region $\Phi(\mc{O}_{r_*})$ by annuli $A_R(s_0)$, we thus obtain $\p_s u \in \mc{A}^{2,\alpha}$, and therefore, by Lemma \ref{lem:camp}, $u=u(s)\in C^{1,\alpha}(\T)$ along the filament surface $\Gamma_\epsilon$. The estimate \eqref{eq:SB_PDE_holder} follows from the bounds \eqref{eq:uholder} and \eqref{eq:L2est_SB_PDE_2} as well as Lemma \ref{lem:camp}.
}
\end{proof}

{\bf Acknowledgments.} \quad
L.O. acknowledges support from NSF Postdoctoral Fellowship DMS-2001959 and thanks Dallas Albritton for helpful discussion. 

\vspace{-0.2cm}
{\bf Data availability.}\quad
Data sharing not applicable to this article as no datasets were generated or analyzed during the current study. 


\bibliographystyle{abbrv} 
\bibliography{LaplaceBib}

\end{document}